\documentclass[12pt]{amsart}
\usepackage{amssymb,amsfonts,amsmath,bbm,mathrsfs,stmaryrd,mathtools}

\usepackage[]{fontenc}
\usepackage[utf8x]{inputenc}

\usepackage[letterpaper,margin=.6in]{geometry}
\usepackage[colorlinks,linkcolor=blue,citecolor=black!75!red]{hyperref}
\usepackage[shortlabels]{enumitem}
\setlist[enumerate,1]{label={(\arabic*)}}
\setlist[enumerate,2]{label={(\alph*)}}
\setlist[enumerate,3]{label={(\roman*)}}
\SetEnumitemKey{inline}{wide}

\usepackage{url}
\usepackage{color}

\usepackage[all]{xypic}

\usepackage{latexsym}

\usepackage{amssymb}
\usepackage{amsfonts}
\usepackage{amscd}
\usepackage{amsmath,amsthm}
\usepackage{verbatim}

\usepackage{tikz}
\usetikzlibrary{matrix,arrows,decorations.pathmorphing,fit,cd}
\tikzset{myboxgroup/.style={draw, densely dotted}} % style for the boxed groups

\usetikzlibrary{arrows,calc,decorations.pathreplacing,decorations.markings,intersections,shapes.geometric,through,fit,shapes.symbols,positioning,decorations.pathmorphing}

\usepackage{cleveref}

\newtheorem{lemma}{Lemma}[section]
\newtheorem{proposition}[lemma]{Proposition}
\newtheorem{theorem}[lemma]{Theorem}
\newtheorem{corollary}[lemma]{Corollary}

{
}

\theoremstyle{definition}

\newtheorem{example}[lemma]{Example}
\newtheorem{definition}[lemma]{Definition}

\theoremstyle{remark}
\newtheorem{remark}[lemma]{Remark}

\makeatletter
\let\xx@thm\@thm
\AtBeginDocument{\let\@thm\xx@thm}
\makeatother

\crefname{section}{section}{sections}
\Crefname{section}{Section}{Sections}
\crefformat{section}{#2section~#1#3}
\Crefformat{section}{#2Section~#1#3}

\crefformat{subsection}{#2\S#1#3}
\Crefformat{subsection}{#2\S#1#3}
\crefrangeformat{subsection}{\S\S#3#1#4--#5#2#6}
\Crefrangeformat{subsection}{\S\S#3#1#4--#5#2#6}
\crefmultiformat{subsection}{\S\S#2#1#3}{ and~#2#1#3}{, #2#1#3}{ and~#2#1#3}
\Crefmultiformat{subsection}{\S\S#2#1#3}{ and~#2#1#3}{, #2#1#3}{ and~#2#1#3}
\crefrangemultiformat{subsection}{\S\S#3#1#4--#5#2#6}{ and~#3#1#4--#5#2#6}{, #3#1#4--#5#2#6}{ and~#3#1#4--#5#2#6}
\Crefrangemultiformat{subsection}{\S\S#3#1#4--#5#2#6}{ and~#3#1#4--#5#2#6}{, #3#1#4--#5#2#6}{ and~#3#1#4--#5#2#6}

\crefname{definition}{Definition}{Definitions}
\crefformat{definition}{#2Definition~#1#3}
\Crefformat{definition}{#2Definition~#1#3}

\crefname{definitionnodiamond}{Definition}{Definitions}
\crefformat{definitionnodiamond}{#2Definition~#1#3}
\Crefformat{definitionnodiamond}{#2Definition~#1#3}

\crefname{example}{Example}{Examples}
\crefformat{example}{#2Example~#1#3}
\Crefformat{example}{#2Example~#1#3}

\crefname{examplenodiamond}{Example}{Examples}
\crefformat{examplenodiamond}{#2Example~#1#3}
\Crefformat{examplenodiamond}{#2Example~#1#3}

\crefname{remark}{Remark}{Remarks}
\crefformat{remark}{#2Remark~#1#3}
\Crefformat{remark}{#2Remark~#1#3}

\crefname{remarks}{remark}{remarks}
\crefformat{remarks}{#2remark~#1#3} 
\Crefformat{remarks}{#2Remark~#1#3} 

\crefname{caution}{Caution}{Cautions}
\crefformat{caution}{#2Caution~#1#3}
\Crefformat{caution}{#2Caution~#1#3}

\crefname{remarknodiamond}{Remark}{Remarks}
\crefformat{remarknodiamond}{#2Remark~#1#3}
\Crefformat{remarknodiamond}{#2Remark~#1#3}

\crefname{convention}{Convention}{Conventions}
\crefformat{convention}{#2Convention~#1#3}
\Crefformat{convention}{#2Convention~#1#3}

\crefname{notation}{Notation}{Notations}
\crefformat{notation}{#2Notation~#1#3}
\Crefformat{notation}{#2Notation~#1#3}

\crefname{notationnodiamond}{Notation}{Notations}
\crefformat{notationnodiamond}{#2Notation~#1#3}
\Crefformat{notationnodiamond}{#2Notation~#1#3}

\crefname{lemma}{Lemma}{Lemmas}
\crefformat{lemma}{#2Lemma~#1#3}
\Crefformat{lemma}{#2Lemma~#1#3}

\crefname{proposition}{Proposition}{Propositions}
\crefformat{proposition}{#2Proposition~#1#3}
\Crefformat{proposition}{#2Proposition~#1#3}

\crefname{corollary}{Corollary}{Corollaries}
\crefformat{corollary}{#2Corollary~#1#3}
\Crefformat{corollary}{#2Corollary~#1#3}

\crefname{theorem}{Theorem}{Theorems}
\crefformat{theorem}{#2Theorem~#1#3}
\Crefformat{theorem}{#2Theorem~#1#3}

\crefname{assumption}{Assumption}{Assumptions}
\crefformat{assumption}{#2Assumption~#1#3}
\Crefformat{assumption}{#2Assumption~#1#3}

\crefname{enumi}{}{}
\Crefname{enumi}{}{}
\creflabelformat{enumi}{#2#1#3}
\crefrangeformat{enumi}{#3#1#4--#5#2#6}
\Crefrangeformat{enumi}{#3#1#4--#5#2#6}
\crefname{enumii}{}{}
\Crefname{enumii}{}{}
\creflabelformat{enumii}{#2#1#3}
\crefname{enumiii}{}{}
\Crefname{enumiii}{}{}
\creflabelformat{enumiii}{#2#1#3}

\makeatletter
\renewcommand{\p@enumii}{}
\renewcommand{\p@enumiii}{}
\makeatother

\crefname{equation}{}{}
\crefformat{equation}{(#2#1#3)}
\Crefformat{equation}{(#2#1#3)}

\crefname{align}{}{}
\crefformat{align}{(#2#1#3)}
\Crefformat{align}{(#2#1#3)}

\crefname{proofstep}{Step}{Steps}
\crefformat{proofstep}{#2Step~#1#3}
\Crefformat{proofstep}{#2Step~#1#3}

\crefname{table}{Table}{Tables}
\crefformat{table}{#2Table~#1#3}
\Crefformat{table}{#2Table~#1#3}

\newcommand{\reqed}{\hfill \ensuremath{\lozenge}}

\numberwithin{equation}{section}
\renewcommand{\theequation}{\thesection-\arabic{equation}}

\newcommand\numberthis{\addtocounter{equation}{1}\tag{\theequation}}

\def\AA{{\mathbb A}}

\def\CC{{\mathbb C}} 
 
\def\EE{{\mathbb E}} 
 
\def\GG{{\mathbb G}}

\def\KK{{\mathbb K}}

\def\PP{{\mathbb P}}

\def\VV{{\mathbb V}}

\def\ZZ{{\mathbb Z}}

\newcommand{\bbC}{\mathbb{C}}

\def\0ol{{\bar 0}}
\def\1ol{{\bar 1}}
\def\2ol{{\bar 2}}
\def\ol2{{\bar 2}}
\def\3ol{{\bar 3}}
\def\4ol{{\bar 4}}
\def\5ol{{\bar 5}}
\def\6ol{{\bar 6}}
\def\7ol{{\bar 7}}
\def\8ol{{\bar 8}}
\def\9ol{{\bar 9}}

\def\bold0{{\bf 0}}
\def\bold1{{\bf 1}}
\def\bold2{{\bf 2}}
\def\bold3{{\bf 3}}
\def\bold4{{\bf 4}}
\def\bold5{{\bf 5}}
\def\bold6{{\bf 6}}
\def\bold7{{\bf 7}}
\def\bold8{{\bf 8}}
\def\bold9{{\bf 9}}

\def\P2Skly{\PP^2_{Skly}}

\def\a{\alpha}
\def\b{\beta}
\def\d{\delta}

\def\g{\gamma}

\def\s{\sigma}

\def\D{\Delta}

\def\fd{{\mathfrak d}}

\def\fm{{\mathfrak m}}

\def\cal{\mathcal}

\def\cA{{\cal A}}
\def\cB{{\cal B}}
\def\cC{{\cal C}}

\def\cE{{\cal E}}
\def\cF{{\cal F}}
\def\cG{{\cal G}}

\def\cI{{\cal I}}

\def\cL{{\cal L}}

\def\cN{{\cal N}}
\def\cO{{\cal O}}

\def\cT{{\cal T}}

\def\cZ{{\cal Z}}

\def\Spec{\operatorname{Spec}}

\def\Proj{\operatorname{Proj}}

\def\End{\operatorname {End}}
\def\Ext{\operatorname {Ext}}
\def\Hom{\operatorname {Hom}}
\def\Aut{\operatorname{Aut}}

\def\id{\operatorname{id}}

\def\codim{\operatorname{codim}}
\def\coh{{\sf coh}}

\def\coker{{\rm coker}}

\def\Div{{\rm Div}}

\def\im{\operatorname{im}}

\def\lcm{{\rm lcm}}
\def\ker{\operatorname{ker}}

\DeclareMathOperator{\rank}{rank}

\DeclareMathOperator{\Sec}{Sec}

\DeclareMathOperator{\Sing}{Sing}

\def\dirlim{\mathop{\vtop{\baselineskip -100pt\lineskip -1pt\lineskiplimit 0pt
      \setbox0\hbox{lim}\copy0\hbox to \wd0{\rightarrowfill}}}\limits}
\def\invlim{\mathop{\vtop{\baselineskip -100pt\lineskip -1pt\lineskiplimit 0pt
      \setbox0\hbox{lim}\copy0\hbox to \wd0{\leftarrowfill}}}\limits}

\def\I11{{1 \kern -0.8pt \! \mbox{l}}}
\def\mumu{{\mu\kern-4.2pt\mu}}
\def\bfmu{{\mu\kern-4.2pt\mu}}
\def\2slash{\backslash \! \backslash}

\DeclareMathOperator{\Bun}{Bun}

\let\originalleft\left
  \let\originalright\right
\renewcommand{\left}{\mathopen{}\mathclose\bgroup\originalleft}
  \renewcommand{\right}{\aftergroup\egroup\originalright}

\makeatletter
\def\l@subsection{\@tocline{2}{0pt}{2.75pc}{5pc}{}}
\makeatother

\begin{document}

\title[The symplectic leaves for the elliptic Poisson bracket]{The symplectic leaves for the elliptic Poisson bracket on projective space defined by Feigin-Odesskii and Polishchuk}

\author{Alex Chirvasitu, Ryo Kanda, and S. Paul Smith}

\address[Alex Chirvasitu]{Department of Mathematics, University at
  Buffalo, Buffalo, NY 14260-2900, USA.}  \email{achirvas@buffalo.edu}

\address[Ryo Kanda]{Department of Mathematics, Graduate School of Science, Osaka Metropolitan University, 3-3-138, Sugimoto, Sumiyoshi, Osaka, 558-8585, Japan}
\email{ryo.kanda.math@gmail.com}

\address[S. Paul Smith]{Department of Mathematics, Box 354350,
  University of Washington, Seattle, WA 98195, USA.}
\email{smith@math.washington.edu}



\begin{abstract}
This paper determines the symplectic leaves for a remarkable Poisson structure on $\mathbb{C}\mathbb{P}^{n-1}$ discovered by
Feigin and Odesskii, and, independently, by Polishchuk. The Poisson bracket is determined by a holomorphic line bundle of degree 
$n \ge 3$ on a compact Riemann surface of genus one or, equivalently, by an elliptic normal curve $E\subseteq\mathbb{C}\mathbb{P}^{n-1}$. 
The symplectic leaves are described in terms of higher secant varieties to $E$.
\end{abstract}

\maketitle

\tableofcontents{}

%%%%%%%%%%%%%%%%%%%%%%%%%%%%%%%%%%%%%%%%%%%%%%%%%%%%%%%%%%%%%%%%
%%%%%%%%%%%%%%%%%%%%%%%%%%%%%%%%%%%%%%%%%%%%%%%%%%%%%%%%%%%%%%%%
\section{Introduction}
%%%%%%%%%%%%%%%%%%%%%%%%%%%%%%%%%%%%%%%%%%%%%%%%%%%%%%%%%%%%%%%%
%%%%%%%%%%%%%%%%%%%%%%%%%%%%%%%%%%%%%%%%%%%%%%%%%%%%%%%%%%%%%%%%

We always work over the field of complex numbers, $\CC$.

Always, $E=(E,0,+)$ denotes an elliptic curve over $\CC$.

An elliptic normal curve $E \subseteq \PP^{n-1}$ is a degree-$n$ elliptic curve that is not contained in any hyperplane.

We always assume $n \ge 3$.

%%%%%%%%%%%%%%%%%%%%%%%%%%%%%%%%%%%%%%%%%%%%%%%%%%%%%%%%%%%%%%%%
\subsection{``Elliptic'' Poisson brackets on projective spaces}
%%%%%%%%%%%%%%%%%%%%%%%%%%%%%%%%%%%%%%%%%%%%%%%%%%%%%%%%%%%%%%%%

In 1998, Feigin and Odesskii \cite{FO98} and, independently, Polishchuk \cite{pl98} discovered a remarkable family of Poisson structures on projective spaces: in short, an elliptic normal curve $E \subseteq \PP^{n-1}$ determines a Poisson bracket, $\Pi_E$, on $\PP^{n-1}$.

Given an elliptic normal curve $E \subseteq \PP^{n-1}$, 
Feigin-Odesskii defined $\Pi_E$ via an explicit, though mysterious,  formula involving theta functions. 
In sharp contrast, Polishchuk defined $\Pi_E$ in abstract terms without  explicit formulas.
Twenty years later Hua and Polishchuk showed these two Poisson structures are the same \cite[Thm.~5.2]{HP1}. 

Given an elliptic normal curve $E \subseteq \PP^{n-1}$, $n \ge 3$, we write $\Pi_E$ for the associated Poisson bracket on $\PP^{n-1}$. It is often called the
Feigin-Odesskii, or Feigin-Odesskii-Sklyanin, bracket, and is often denoted by $q_{n,1}(E)$. Polishchuk's name should also be attached to $\Pi_E$.
We define the line bundle $\cL:=\cO_{\PP^{n-1}}(1) \vert_E$. As a consequence, 
$\PP^{n-1}$ can be identified with $\PP H^0(E,\cL)^*$ and hence, via Serre duality, with 
the set of isomorphism classes of non-split extensions\footnote{See \Cref{sect.appx.extns} for the definition of isomorphic extensions.
If we fix $\cA$ and $\cC$ such that $\End(\cA) \cong \End(\cC) \cong \CC$, then $\PP\Ext^1(\cC,\cA)$ is in natural bijection with the set of isomorphism
classes of extensions that are isomorphic to a non-split extension of the form $0 \to \cA \to \cB \to \cC \to 0$ (\cref{prop.isom.extns}). These definitions make sense in a larger context: 
if $X$ is a projective scheme over an algebraically closed field $\Bbbk$, one can fix $\cA,\cC \in \coh(X)$ such that $\End(\cA) \cong \End(\cC) \cong \Bbbk$
and address the question of determining the homological leaves in $\PP\Ext^1(\cC,\cA)$. 
Another interesting case would be to replace $\coh(X)$ by the category of finite-dimensional left modules over a fixed finite dimensional 
$\Bbbk$-algebra.
} 
of $\cL$ by $\cO_E$; i.e., $\PP^{n-1}$
can be identified with 
\begin{equation*}
\PP_\cL \; := \; \PP\Ext^1(\cL,\cO_E),
\end{equation*}
 and $\Pi_E$ can be defined in terms of the geometry associated to such extensions.

We always identify $E$ with its image under the composition $E \to \PP H^0(E,\cL)^* \to  \PP\Ext^1(\cL,\cO_E)=\PP_\cL$.

Given a Poisson manifold $(X,\Pi)$, its \textsf{symplectic leaves} are the members of the unique partition of $X$ into connected immersed submanifolds $Y \subseteq X$ such that for each $x \in Y$,  $T_x Y=$ the image of $\Pi_x: T_x^*X \longrightarrow T_xX$ 
\cite[Prop.~1.8, p.~6; (2.10), p.~29; Thm.~4.1, p.~63]{clm} (equivalently: the maximal connected immersed submanifolds $Y\subseteq X$ with this property \cite[Prop.~4.11, p.~68]{clm}). 
Each such $Y$ is a symplectic manifold with respect to the restriction of $\Pi$.

Since $\Pi_E$ was first discovered several groups of mathematicians have wanted a description of the symplectic leaves in terms of the 
geometry related to $E$ as a subvariety of $\PP_\cL$.
This paper solves the problem for all $n\ge 3$ (the solution for $n=3,4$ is folklore).

The introduction to \cite{Pym17} surveys  Poisson structures and symplectic leaves on projective varieties. 

%%%%%%%%%%%%%%%%%%%%%%%%%%%%%%%%%%%%%%%%%%%%%%%%%%%%%%%%%%%%%%%%
\subsection{The homological leaves $L(\cE)$}
%%%%%%%%%%%%%%%%%%%%%%%%%%%%%%%%%%%%%%%%%%%%%%%%%%%%%%%%%%%%%%%%

An extension of $\cL$ by $\cO_E$ is a rank-two locally free $\cO_E$-module whose determinant is $\cL$. 
Given an extension $\xi \in \PP_\cL$, represented by the non-split exact sequence $0 \to \cO_E \to \cE \to \cL \to 0$, 
we call $\cE$ the {\sf middle term} of $\xi$ and denote it by $m(\cE)$. 
Thus we have a set map
\begin{equation*}
m:\PP_\cL \to \Bun(2,\cL) \;:=\;   
\{
\text{rank-two locally free $\cO_E$-modules $\cE$ such that $\det\cE\cong\cL$}
\},
\end{equation*} 
For each $\cE \in \Bun(2,\cL)$ we define
\begin{equation}
\label{L(cE).defn}
L(\cE) \; :=\; \{\xi \in\PP_\cL \; | \;  m(\xi)\cong \cE  \} \; = \; m^{-1}(\cE).
\end{equation} 
We call $L(\cE)$ a {\sf homological leaf}. Clearly, $\PP_{\cL}$ is the disjoint union of the $L(\cE)$'s.

Even if one ignores the ``symplectic'' origin of the problem, the classification of homological leaves is a natural problem in a wide range of settings.
A solution to such a problem consists of
\begin{enumerate}
\item 
 a ``meaningful list'' of those $\cE$'s for which $L(\cE) \ne \varnothing$; 
 \item
 a geometric description of each $L(\cE)$;
 \item
 a determination of whether $L(\cE)$ is a quasi-affine, quasi-projective, affine, or projective variety;
\item
answers to questions like: what is the dimension of $L(\cE)$? what is the singular locus of $L(\cE)$? when is one leaf contained in the closure of another? And so on.
\end{enumerate} 
We answer these and other questions about the $L(\cE)$'s.

In  their 1998 paper, Feigin and Odesskii claimed without proof that the symplectic leaves for $\Pi_E$ are precisely the homological leaves
 \cite[Thm.~1, p.~66]{FO98}. 
Here we prove that the $L(\cE)$'s are indeed the symplectic leaves (\cref{th:leleaves}).
 Before proving this one must show  that every $L(\cE)$ is smooth of even dimension (\cref{th:smth}).

Before stating our main results in \cref{sect.main.results}, we need the notation in \cref{ssect.secant.vars}.

%%%%%%%%%%%%%%%%%%%%%%%%%%%%%%%%%%%%%%%%%%%%%%%%%%%%%%%%%%%%%%%%
\subsection{Higher secant varieties and partial secant varieties to $E$}
\label{ssect.secant.vars}
%%%%%%%%%%%%%%%%%%%%%%%%%%%%%%%%%%%%%%%%%%%%%%%%%%%%%%%%%%%%%%%%

Let 
\begin{equation*}
  E^{[d]} \; :=\; \text{the $d^{\rm \, th}$ symmetric power of $E$}.
\end{equation*}
We write $(\!(x_1,\ldots,x_d)\!)$  for the image in $E^{[d]}$ of $(x_1,\ldots,x_d) \in E^d$.
An effective divisor of  degree $d$ can be thought of in three ways: 
as a divisor, as a closed subscheme of $E$ having length $d$, and as a point in $E^{[d]}$.
The morphism 
\begin{equation*}
  \sigma:E^{[d]} \to E,  \qquad \s(\!(x_1,\ldots,x_d)\!):= x_1+\cdots+x_d,
\end{equation*}
presents $E^{[d]}$ as a $\PP^{d-1}$-bundle over $E$.\footnote{The symbol $\s$ reminds us of the symbol $\Sigma$ and the word ``sum''.}
We also define $\s(\cN):=\s(D)$ if $\cN \cong \cO_E(D)$. 

All intersections in this paper are  {\it scheme-theoretic} intersections.
 
Given an effective divisor $D$ on $E$ we write $\overline{D}$ for its linear span; i.e.,
\begin{equation*}
\overline{D}  \; =  \;\text{the smallest linear subspace $L \subseteq \PP_\cL$ such that $L \cap E$ contains $D$}.
\end{equation*}
   We call $\overline{D}$ a {\sf secant plane} or a {\sf $d$-secant} if its dimension is $d$.\footnote{The following facts are well-known (see \cref{prop.sec.low.deg,le:fisher}).
  If $\deg D \ge n+1$, then $\overline{D}=\PP_\cL$. If  $\deg D  =n$ and $\cO_E(D) \not\cong \cL$, then $\overline{D}=\PP_\cL$.  
  If  $\deg D  =n$ and $\cO_E(D) \cong \cL$, then $\dim \overline{D}=n-2$ and $\overline{D} \cap E = D$.
   If  $\deg D  =n -1$, then $\dim \overline{D}=n-2$ and $\deg (\overline{D} \cap E) = n$. 
  If  $\deg D \le n-2$, then $\dim \overline{D}=\deg D -1$ and $\overline{D} \cap E = D$.}
  
The $d^{\, \rm th}$ {\sf secant variety} to $E$ is
\begin{equation}
  \label{eq:secant.union}
  \Sec_d(E) \; :=\; \bigcup_{D \in E^{[d]}} \overline{D} \, .
\end{equation}
It is well-known that $\Sec_d(E)$ is a closed irreducible subvariety of $\PP_\cL$ of dimension $\min\{2d-1,n-1\}$; see \cite[Lem.~1 and the Theorem on p.~266]{Lan84} or \cite[p.~11]{fisher2006} or  \cite[Prop.~10.11]{3264}.

For each $x \in E$, we define 
\begin{equation}\label{eq:exd}
  E^{[d]}_x \; :=\; \s^{-1}(x) = \{(\!(x_1,\ldots,x_d)\!) \; | \; x_1+\cdots+x_d =x\}
\end{equation}
and the {\sf partial secant variety}
\begin{equation}
  \label{eq:union.equiv.eff.divisors}
  \Sec_{d,x}(E)   \; :=\; \bigcup_{D \in \s^{-1}(x)}  \overline{D} \, .
\end{equation}
By \cref{prop.dim.Sec.dz.2}, $\Sec_{d,x}(E)$ is irreducible and its dimension  is $\min\{2d-2,n-1\}$.
By \cite[p.~18]{acgh1},    $E^{[d]}_x \cong \PP^{d-1}$.  Divisors $D$ and $D'$ in $E^{[d]}$ are linearly equivalent if and only if $\s(D)=\s(D')$. 
Thus, if $D$ is any divisor in $E^{[d]}_x$, then $E^{[d]}_x=|D|$ and 
\begin{equation}
  \label{eq:union.lin.equiv.divisors}
  \Sec_{d,x}(E)   \; =\; \bigcup_{D' \in |D|}  \overline{D'}.
\end{equation}

%%%%%%%%%%%%%%%%%%%%%%%%%%%%%%%%%%%%%%%%%%%%%%%%%%%%%%%%%%%%%%%%
\subsection{Main results}
\label{sect.main.results}
%%%%%%%%%%%%%%%%%%%%%%%%%%%%%%%%%%%%%%%%%%%%%%%%%%%%%%%%%%%%%%%%
We often indicate whether $n$ is even or odd by saying ``if $n=2r$'' or ``if $n=2r+1$''.

We define  $\Omega:=\{x \in E \; | \; 2x=\s(\cL)\}$.
It is a coset for the 2-torsion subgroup, $E[2]$.

 \begin{theorem}
 [\cref{th:leleaves}]
 \label{thm.main.sym.leaves.=.homol.leaves}
The symplectic leaves for $(\PP_\cL,\Pi_E)$ are  the homological leaves $L(\cE)$.\footnote{The proof of this theorem makes essential use of a result of Hua and Polishchuk \cite[Prop.~2.3]{HP3}.}
  \end{theorem}

\begin{theorem}
\label{thm.main.0}
Assume $n \ge 3$. 
\begin{enumerate}
  \item 
 If $n=2r+1$, then the symplectic leaves   are
 \begin{enumerate}
  \item[-] 
  $\PP_\cL-\Sec_{r}(E)$ and
  \item[-] 
$\Sec_{d,x}(E) -\Sec_{d-1}(E)$ for each $x \in E$ and each integer $1 \le d \le r$.
\end{enumerate}
  \item 
    if $n=2r$, then the symplectic leaves   are
 \begin{enumerate}
  \item[-] 
  $\Sec_{d,x}(E) -\Sec_{d-1}(E)$ for each $x \in E$ and each integer  $1 \le d \le r-1$ and
  \item[-] 
   $\Sec_{r,x}(E) -\Sec_{r-1}(E)$ for each $x \in  E- \Omega$ and
  \item[-]
 for each $\omega \in \Omega$,  the set of $\xi \in \Sec_{r,\omega}(E) -\Sec_{r-1}(E)$ that lie on a unique $r$-secant, which is the smooth locus of 
 $\Sec_{r,\omega}(E) -\Sec_{r-1}(E)$, 
 and
    \item[-]
 for each $\omega \in \Omega$,   the set of $\xi \in \Sec_{r,\omega}(E) -\Sec_{r-1}(E)$ that lie on infinitely many $r$-secants, which is the 
 singular locus of 
 $\Sec_{r,\omega}(E) -\Sec_{r-1}(E)$ .
\end{enumerate}   
\end{enumerate}
\end{theorem}

Given an integer $d \in [1,\frac{n}{2}]$, a point $x \in E$, and a divisor  $D \in E^{[d]}_x$, we define $\cE_{d,x}:=\cO_E(D) \oplus \cL(-D)$.
Its isomorphism class  does not depend on the choice of $D$. 

Suppose $n=2r$.  
 For each $\omega \in \Omega$, we define $\cL_\omega:=\cO_E(D)$ where  $D \in E^{[r]}_\omega$
 (the isomorphism class of $\cL_\omega$ does not depend on the choice of $D$) and define 
 $\cE_\omega:=$ the unique-up-to-isomorphism non-split extension of $\cL_\omega$ by $\cL_\omega$ (see \cref{ssect.bun.notn}). 
 (We note that  $\cE_{r,\omega}\cong\cL_{\omega}\oplus\cL_{\omega}$).

When $n=2r+1$, we write $\cE_o$ for the unique-up-to-isomorphism indecomposable locally free $\cO_E$-module of rank two and 
determinant $\cL$.

\begin{theorem}[\cref{prop.leaves.for.decomp.Es,thm.good.Es,cor.good.E's}]
\label{thm.good.Es.intro}
$\phantom{x}$

\noindent
If $\cE \in \Bun(2,\cL)$, then  $L(\cE)\ne \varnothing$ if and only if either
\begin{enumerate}
\item[-]
$\cE \cong \cN_1 \oplus \cN_2$  where $\cN_1$ and $\cN_2$ are invertible $\cO_{E}$-modules of positive degree or
  \item[-] 
  $\cE$ is indecomposable.  
\end{enumerate}
In other words, 
\begin{itemize}
\item[-] 
if $n=2r+1$, then $L(\cE)\ne \varnothing$ if and only if
\begin{equation*}
	\cE \;\in\; \bigl\{ \cE_{d,x} \; \big\vert \; \text{$1\le d \le r$ and $x \in E$} \bigr\} \cup \bigl\{ \cE_o \bigr\};
\end{equation*}	
  \item[-]
  if $n=2r$, then $L(\cE)\ne \varnothing$ if and only if
\begin{equation*}
	\cE \;\in\; \bigl\{ \cE_{d,x} \; \big\vert \; \text{$1\le d\le r$ and $x \in E$} \bigr\} \cup \bigl\{ \cE_\omega \; \big\vert \; \omega \in \Omega \bigr\}.
\end{equation*}
\end{itemize}
\end{theorem}

The next result matches up the leaves in \Cref{thm.main.0} with the $\cE$'s  in \cref{thm.good.Es.intro}. 

\begin{theorem}[\cref{th:splitall,th:nsplitodd,th:nsplitev}]
\label{thm.main} 
\leavevmode
\begin{enumerate}
	\item
	\label{item.main.odd} 
	If $n=2r+1$, then
	\begin{enumerate}
		\item[-]
		 $L(\cE_{d,x}) =\Sec_{d,x}(E) \, - \, \Sec_{d-1}(E)$ for all $x \in E$ and $1\le d \le r$;
		\item[-]
		 $L(\cE_o) = \PP_\cL -  \Sec_r(E)$.
	\end{enumerate}
	\item\label{item.main.even} 
	If $n=2r$, then
	\begin{enumerate}
		\item[-]
		$L(\cE_{d,x}) =\Sec_{d,x}(E) \, - \, \Sec_{d-1}(E)$ if $1 \le d \le r-1$, or $d=r$ and $x \notin \Omega$;
		\item[-]
		 $L(\cE_\omega)\sqcup L(\cE_{r,\omega}) \;=\; \Sec_{r,\omega}(E) \, - \,  \Sec_{r-1}(E)$ if $\omega\in\Omega$;
		\item[-]
		\label{item.main.even.inf} 
		$L(\cE_{r,\omega})$ consists of those points in $\Sec_{r,\omega}(E) - \Sec_{r-1}(E)$ that lie on at least two, and 
		hence infinitely many, distinct $r$-secant planes, for each $\omega\in\Omega$;
		\item[-]
		\label{item.main.even.unq} 
		$L(\cE_\omega)$ is a dense open subset of $\Sec_{r,\omega}(E)$ and consists of those points in  
		$\Sec_{r,\omega}(E) - \Sec_{r-1}(E)$ that lie on a {\it unique} $r$-secant plane, for each $\omega\in\Omega$.
	\end{enumerate}
\end{enumerate}
\end{theorem}

\begin{theorem}
[\cref{th:splitall,th:nsplitodd,th:nsplitev}]
\label{prop.dim.LE}
\leavevmode
\begin{enumerate}
	\item If $n=2r+1$, then
	\begin{enumerate}
		\item[-]
		 $\dim L(\cE_{d,x}) = 2d-2$ if $1 \le d \le r$;
		\item[-]
		 $\dim L(\cE_o) = 2r$.
	\end{enumerate}
	\item If $n=2r$, then
	\begin{enumerate}
		\item[-]
		 $\dim L(\cE_{d,x}) = 2d-2$ if $1 \le d \le r-1$, or $d=r$ and $x \notin \Omega$;
		\item[-]
		 $\dim L(\cE_{r,\omega}) = n-4$ if $\omega\in\Omega$;
		\item[-]
		 $\dim L(\cE_\omega) = n-2$ if $\omega\in\Omega$.
	\end{enumerate}
\end{enumerate}
\end{theorem}

In other words, when $n=2r+1$,  the homological leaves are   
\begin{enumerate}
    \item 
      the individual points $x \in E$, $\{x\}=\Sec_{1,x}(E)=L(\cE_{1,x})$, and
    \item 
      $\Sec_{d,x}(E)-\Sec_{d-1}(E)$ for each $x \in E$ and $2 \le d \le r$, which has dimension $2d-2$, and
   \item
   $\PP_\cL -    \Sec_{r}(E)$.
    \end{enumerate}  
When $n=2r$,  the  homological leaves are   
    \begin{enumerate}  
    \item
       the individual points $x \in E$, $\{x\}=\Sec_{1,x}(E)=L(\cE_{1,x})$, and
    \item
 $\Sec_{d,x}(E)-\Sec_{d-1}(E)$ for each $x \in E$ and $2 \le d \le r$, which has dimension $2d-2$, and
  \item
 $\Sec_{r,x}(E)-\Sec_{r-1}(E)$ for each $x \in E-\Omega$, which has dimension $n-2$, and
 \item
the four $L(\cE_{r,\omega})$'s of dimension $n-4$ described in \cref{thm.main}\cref{item.main.even}, and
\item
the four $L(\cE_\omega)$'s of dimension $n-2$ described in \cref{thm.main}\cref{item.main.even}.
\end{enumerate}

In particular, every point on $E$ is a homological leaf, and these are the only 0-dimensional leaves when 
$n \ne 4$. When $n=4$  there are four extra ones, namely
 $L(\cL_\omega \oplus \cL_\omega)=L(\cE_{2,\omega})$ for $\omega \in \Omega$.

One can determine when one leaf is contained in the closure of another from the inclusions
\begin{equation}
  \label{eq:Y.sec.inclusions}
  \Sec_{d,x}(E) \;  \subseteq \; \Sec_d(E) \;  \subseteq \; \Sec_{d+1,y}(E)   \subseteq \; \Sec_{d+1}(E)
\end{equation}
which hold for all $x,y \in E$ and all $d$.

The dimensions in \Cref{prop.dim.LE} follow from the fact that if $1 \le d < \frac{n}{2}-1$, then 
the dimensions of the varieties in \cref{eq:Y.sec.inclusions} are $2d-2$, $2d-1$, $2d$, and $2d+1$; 
\Cref{prop.dim.Sec.dz} shows that $\dim \Sec_{d,z}(E) =2d-2$. (In \cref{ssect.secant.vars} there is a reference for the known equality $\dim \Sec_{d}(E) =\min\{2d-1,n-1\}$.)

\begin{proposition}
[\cref{pr:affine-reductive}]
\leavevmode
\begin{enumerate}
	\item If $n=2r+1$, then $L(\cE_{o})$ is affine.
	\item If $n=2r$, then $L(\cE_{r,x})$ is affine for all $x\in E$, and the $L(\cE_{\omega})$'s are quasi-affine but not affine.
\end{enumerate}
\end{proposition}

Symplectic leaves are, by definition, smooth \cite[Thm.~4.1, p.~63]{clm}. 
However, before we can show that the $L(\cE)$'s are the symplectic leaves we first need to know 
they are smooth, and that is not a priori obvious. However, \Cref{prop.8.15.GvB-H} shows that $\Sec_{d,x}(E) -\Sec_{d-1}(E)$
is smooth if either $d<\frac{n}{2}$ or, when $n=2r$, if $d=\frac{n}{2}$ and $x \notin \Omega$; i.e., that 
$L(\cE_{d,x})$ is smooth when either $d<\frac{n}{2}$ or $d=\frac{n}{2}$ and $x \notin \Omega$. 

\begin{theorem}
[\cref{th:smth}]
Every $L(\cE)$ is a smooth quasi-projective variety of even dimension. 
\end{theorem}

%%%%%%%%%%%%%%%%%%%%%%%%%%%%%%%%%%%%%%%%%%%%%%%%%%%%%%%%%%%%%%%%
\subsection{Some special cases}
\label{ssect.low.dim}
%%%%%%%%%%%%%%%%%%%%%%%%%%%%%%%%%%%%%%%%%%%%%%%%%%%%%%%%%%%%%%%%

\begin{enumerate}
\item
  By definition,  $\Sec_0(E)=\varnothing$, $\Sec_1(E)=E$, $\Sec_2(E)$ is the union of all the secant lines.
\item  
If $d \ge \frac{n}{2}$, then $\Sec_d(E)=\PP_\cL$ .
\item
Each $x \in E$ is a 0-dimensional leaf, $\{x\}=L(\cE_{1,x})$ where $\cE_{1,x}\cong \cO_E(x) \oplus \cL(-x)$.  
\item 
  When $n=3$, the $0$-dimensional leaves are the individual points of $E$  and $\PP_\cL-E$ is the unique 2-dimensional leaf. 
\item 
  When $n=4$, $\{\Sec_{2,x}(E) \; | \; x \in E\} = \{\text{quadrics containing $E$}\}$.
  The union of these quadrics is $\PP_\cL=\PP^3$. If $E \subseteq \PP^3$ is such that 
  $\cO(1)|_E \cong \cO_E(4\!\cdot\!(0))$,   then $\Omega=E[2]$,   $\Sec_{2,x}(E)=\Sec_{2,-x}(E)$  for all $x$, and $\Sec_{2,x}(E)$ 
   is singular if and only if  $x \in E[2]$. 
When $n=4$, the symplectic leaves are the individual points of $E$, the vertices of the four singular quadrics that contain $E$, 
  the complement of $E \cup\{\text{the vertex}\}$  in each singular quadric, and the complement of $E$ in each smooth quadric that contains $E$.
We refer the reader to \cite[\S 3]{LS93} for some elementary arguments regarding the lines on the quadrics that contain $E$, and to \cite[pp.~27--29]{Hulek86} for more about $E \subseteq \PP^3$.
\item
When  $n \ge 5$, the only 0-dimensional leaves are the points of $E$.
\item
When $n=5$,  $\Sec_2(E)$ is a quintic hypersurface (see, e.g., \cite[Prop.~VIII.2.2(iii),  \S VIII.2.4, \S VIII.2.5]{Hulek86}).  
Its complement, $\PP^4 -\Sec_2(E)$, is the unique 4-dimensional leaf; it is $L(\cE_o)$ where $\cE_o$ is the unique-up-to-isomorphism
rank-two indecomposable locally free $\cO_E$-module whose determinant is isomorphic to $\cL$. The 3-dimensional variety $\Sec_2(E)-E$ is 
the disjoint union of the (2-dimensional) partial secant varieties $\Sec_{2,x}(E)-E$, indexed by $x \in E$.
$\Sec_{2,x}(E)$ is the union of the secant lines $\overline{y,x-y}$, $y \in E$. 
\end{enumerate}

%%%%%%%%%%%%%%%%%%%%%%%%%%%%%%%%%%%%%%%%%%%%%%%%%%%%%%%%%%%%%%%%
\subsection{Content and Methods}
%%%%%%%%%%%%%%%%%%%%%%%%%%%%%%%%%%%%%%%%%%%%%%%%%%%%%%%%%%%%

Secants to elliptic curves, other curves, and higher dimensional varieties  have been well-studied for over a century.
We have little to add, but collect the necessary preliminaries in \Cref{sect.secants}. Perhaps the only new result in that section 
is \Cref{prop.8.15.GvB-H} which shows that almost all the partial secant varieties $\Sec_{d,x}(E)-\Sec_{d-1}(E)$ are smooth, and therefore complex manifolds.
This is a necessary prerequisite for showing that those $\Sec_{d,x}(E)-\Sec_{d-1}(E)$ are symplectic leaves.
The only exceptions are $\Sec_{r,\omega}(E)-\Sec_{r-1}(E)$ when $n=2r$ and $\omega \in\Omega$: those are the union of two symplectic leaves (see the last paragraph of this subsection).

In \Cref{sect.rank.2.buns} we determine those $\cE \in \Bun(2,\cL)$ for which $L(\cE) \ne \varnothing$.
A key step in doing that is the examination of the subvariety $X(\cE) \subseteq \Hom(\cO_E,\cE)$ that consists of the homomorphisms 
$f:\cO_E \to \cE$ whose cokernel is invertible, and hence isomorphic to $\cL$. Every $\xi \in L(\cE)$ is of the form 
$0 \to \cO_E \to \cE \to \coker(f) \to 0$ for some  $f \in X(\cE)$. 
In terms of of vector bundles, $X(\cE)$ consists of the homomorphisms $f:\cO_E \to \cE$ that
correspond to embeddings of the trivial bundle $E \times \CC$ in the rank-two vector bundle $\VV(\cE)$  that corresponds to $\cE$.
The left action of $\Aut(\cE)$ on $ \Hom(\cO_E,\cE)$ leaves $X(\cE)$ stable and the fibers of the map $X(\cE) \to L(\cE)$ are the 
$\Aut(\cE)$-orbits. 

In \Cref{subse:smth} we will show that  $L(\cE)$ is the geometric quotient $X(\cE)/\Aut(\cE)$. 
The proof that  $L(\cE)$ is the geometric quotient $X(\cE)/\Aut(\cE)$, and smooth, involves the examination of various tangent spaces and differentials.
Those results are collected in \Cref{sect.symp.leaves} so as not to interrupt the  main ideas.
We end \Cref{subse:smth} with \Cref{prop.ssnsmth} which clarifies the words ``almost all'' that appear in the first paragraph of this section: it 
shows that if $n=2r$ and $\omega \in \Omega$, then the singular locus of $\Sec_{r,\omega}(E)-\Sec_{r-1}(E)$  is 
$L(\cL_\omega \oplus \cL_\omega)$ and its smooth locus is $L(\cE_\omega)$. 

\Cref{sect.symp.leaves} completes the proof that the $L(\cE)$'s are the symplectic leaves.

It is easy to see that $\Aut(\cE)$ acts freely on $X(\cE)$ so $\dim L(\cE)=n-\dim \Aut(\cE)$. 
It is easy to compute $\Aut(\cE)$ and hence $\dim L(\cE)$ for the relevant $\cE$'s (those in \Cref{ssect.assump}).
Several other properties of $L(\cE)$ follow from the fact that $L(\cE) \cong X(\cE)/\Aut(\cE)$.

\Cref{sec.xi.D.bar.sec} concerns the relation between the homological leaves $L(\cE)$ and the secant and partial secant varieties.
The key to relating these is \cref{prop.bertram}, which was inspired by an erroneous remark in \cite{Ber92} 
(see \cref{re:bertr}). \Cref{prop.bertram} says that, for an effective divisor $D'$ on $E$, a point $\xi \in \PP_\cL$ belongs to $\overline{D}$ for some effective $D\sim D'$ if and only if there is a non-zero map $m(\xi) \to \cO_E(D')$.
Furthermore, if $D$ has minimal degree such that $\xi \in \overline{D}$, then there is an {\it epimorphism} $m(\xi) \to  \cO_E(D)$ (\cref{thm.bertram}). 

The leaves $L(\cE_\omega)$ and $L(\cL_\omega \oplus \cL_\omega)$, which occur when $n$ is even and $\omega \in \Omega$, present some 
special difficulties. When $n=4$, these leaves lie on the singular quadrics that contain $E$; indeed, the union of those four quadrics is the disjoint 
union of $E$ and the eight leaves of the form $L(\cE_\omega)$ and $L(\cL_\omega \oplus \cL_\omega)$. 
Whenever $n=2r$, these eight leaves play a similar role: the singular locus of $\Sec_{r,\omega}(E)-\Sec_{r-1}(E)$ is  
$L(\cL_\omega \oplus \cL_\omega)$ and its smooth locus is $L(\cE_\omega)$.

%%%%%%%%%%%%%%%%%%%%%%%%%%%%%%%%%%%%%%%%%%%%%%%%%%%%%%%%%%%%%%%%
\subsection{Definitions of $\Pi_E$}
%%%%%%%%%%%%%%%%%%%%%%%%%%%%%%%%%%%%%%%%%%%%%%%%%%%%%%%%%%%%%%%%
Almost 25 years after $\Pi_E$ was first defined, Polishchuk  gave a simple explicit formula for $(\PP^{n-1},\Pi_E)$  in  \cite[Thm.~A(1)]{pol2022}, namely 
\begin{equation}
\label{PB.Polishchuks.formula}
  \{x_i,x_j\} \; =\;  \Omega_{ij}, \qquad 1 \le i,j \le n,
\end{equation}
 where $x_1,\ldots,x_n$ are coordinate functions on $\PP_\cL$ and the $\Omega_{ij}$'s 
 are determined by the defining equation(s) for the largest secant variety 
 $\Sec_d(E) \subsetneq \PP^{n-1}$ in the following way.
When $n=2r+1$, $\Sec_r(E)$ is given by a degree-$n$ equation, $F=0$ say, and $(\Omega_{ij})$ is the unique-up-to-scaling skew-symmetric
$n \times n$ matrix of quadratic forms such that the relations 
\begin{equation*}
\sum_{i=1}^n \frac{\partial F}{\partial x_i} \Omega_{ij} \;=\; 0
\end{equation*}
generate the module of syzygies for $\big(\frac{\partial F}{\partial x_1}, \ldots, \frac{\partial F}{\partial x_n}\big)$. 
That the module of syzygies can be generated in this way is due to Fisher \cite[Thm.~1.1]{fis18}.
When $n=2r+2$, $\Sec_{r}(E)$  has codimension two: it is given by equations $F_1=F_2=0$  with $\deg F_1=\deg F_2=r+1$, 
and,  by \cite[Thm.~1.1]{fis18}, 
the module of syzygies between the columns of the $2 \times n$ matrix 
$\big(\frac{\partial F_a}{\partial x_1}, \ldots, \frac{\partial F}{\partial x_n}\big)$ is generated by the relations
\begin{equation*}
\sum_{i=1}^n \frac{\partial F_a}{\partial x_i} \Omega_{ij} \;=\; 0,
\end{equation*}
for a unique-up-to-scaling skew-symmetric  $n \times n$ matrix $(\Omega_{ij})$ of quadratic forms.
The starting point for the formula in \cref{PB.Polishchuks.formula} is a result of Hua and Polishchuk, \cite[Lem.~2.1, Prop.~5.8]{HP3}, which shows 
that the value, $\Pi_\xi$, of $\Pi_E$ at a point $\xi \in \PP_\cL$ is be given by an explicit triple Massey product. 

We never use an explicit formula for the Poisson bracket. Indeed, the Poisson bracket only enters our arguments in the proof of 
\Cref{th:leleaves} where we use a property of it that was proved by Hua and Polishchuk in \cite[Prop.~2.3]{HP3}. 

Our  interest in $(\PP^{n-1},\Pi_E)$ arises from Feigin and Odesskii's original definition of $\Pi_E$ 
in terms of a family of non-commutative deformations, $Q_{n,1}(E,\eta)$,  of the polynomial ring on $n$ variables. 
Here $\eta$ is a complex number and $Q_{n,1}(E,0)$ is the polynomial ring on $n$ variables. The precise definition of $Q_{n,1}(E,\eta)$ need not concern us 
here, so we refer the interested reader to Feigin and Odesskii's papers \cite{FO-Kiev,FO89,FO98}, and to our papers \cite{CKS1,CKS4,CKS6}. 
It suffices to say that  since $Q_{n,1}(E,\eta)$ is a  quotient of the free algebra $\CC \langle x_1,\ldots,x_n \rangle$ modulo quadratic relations, 
the formula 
\begin{equation}\label{eq:defn.PB}  
  \{x_i,x_j\} \; :=\; \lim_{\eta \to 0} \frac{[x_i,x_j]}{\eta}
\end{equation}
extends to a Poisson bracket on the polynomial ring $Q_{n,1}(E,0)$. Since $[x_i,x_j]$ is homogeneous of degree two, so is  $\{x_i,x_j\}$. It follows that the Poisson bracket on $Q_{n,1}(E,0)$ induces a Poisson structure on $\Proj Q_{n,1}(E,0) \cong \PP^{n-1}$.  
In \cite[\S5.2]{HP1}, Hua and Polishchuk give explicit formulas for $\{x_i,x_j\}$ and $\{t_i,t_j\}$ which, like the defining relations for $Q_{n,1}(E,\eta)$,
involves theta functions, where $t_i=\frac{x_i}{x_0}$. We do not need these so we omit them.

%%%%%%%%%%%%%%%%%%%%%%%%%%%%%%%%%%%%%%%%%%%%%%%%%%%%%%%%%%%%%%%%
\subsection{The singular locus of $\Sec_{d}(E)$}
%%%%%%%%%%%%%%%%%%%%%%%%%%%%%%%%%%%%%%%%%%%%%%%%%%%%%%%%%%%%%%%%

Although we do not need it, we note that if $d < \frac{n}{2}$, then
\begin{equation}
\label{eq:sing.secd}
\Sing ( \Sec_d(E)) \; = \; \Sec_{d-1}(E),
\end{equation}
where $\Sing(-)$ denotes the singular locus.
At first, the history of this result confused us. The fact that $\Sec_d(E) - \Sec_{d-1}(E)$ is smooth follows from a result proved in 1992 by Bertram
\cite[Corollary on p.~440]{Ber92}.\footnote{That corollary concerns smooth projective curves of arbitrary genus.}
In 2004, Bertram's result was proved  again in \cite[Prop.~8.15]{gvb-hul} but that paper doesn't cite \cite{Ber92}.
The first paragraph in \cite{fis10} 
 says that  \cite{gvb-hul} proved the equality in \cref{eq:sing.secd} even though they only proved that 
$\Sing(\Sec_d(E)) \subseteq \Sec_{d-1}(E)$. The paper \cite{fis10} ``replaces'' the unpublished paper \cite{fisher2006}; 
the first paragraph in \cite{fis10}, which is essentially the same as that in  \cite{fisher2006}, states \cref{eq:sing.secd}; 
 just after \cite[Thm.~1.4]{fisher2006} it is said that \cref{eq:sing.secd} is a consequence of  \cite[Thm.~1.4]{fisher2006}
 but, paraphrasing Fisher, the proof is omitted because it is closely related to that in \cite[Prop.~8.15]{gvb-hul}. 
Fortunately, (the very simple argument at) \cite[p.~18]{copp}  shows that if $X \subseteq \PP^{n-1}$ is any smooth irreducible projective variety such that
 $\Sec_d(X) \ne \PP^{n-1}$ and $X$ is not contained in any hyperplane, then
\begin{equation}
\label{eq:sing.secd.2}
\Sec_{d-1}(X)   \;  \subseteq \; \Sing ( \Sec_d(X)). 
\end{equation}
The equality  \cref{eq:sing.secd} follows from this and Bertram's result.

We will later use the fact that the argument at \cite[p.~18]{copp} shows more than \cref{eq:sing.secd.2}: 
it shows that the dimension of the tangent space to $\Sec_{d}(X)$ at a 
point $x \in \Sec_{d-1}(X)$ is $n-1$. 

%%%%%%%%%%%%%%%%%%%%%%%%%%%%%%%%%%%%%%%%%%%%%%%%%%%%%%%%%%%%%%%%
\subsection{Acknowledgements}
%%%%%%%%%%%%%%%%%%%%%%%%%%%%%%%%%%%%%%%%%%%%%%%%%%%%%%%%%%%%%%%%

A.C. was partially supported by NSF grant DMS-2001128.
R.K. was supported by JSPS KAKENHI Grant Numbers JP16H06337, JP17K14164, JP20K14288, and JP21H04994, Leading Initiative for Excellent Young Researchers, MEXT, Japan, and Osaka Central Advanced Mathematical Institute: MEXT Joint Usage/Research Center on Mathematics and Theoretical Physics JPMXP0619217849.
S.P.S. thanks JSPS (Japan Society for the Promotion of Science) for  financial support, and Osaka University and Osaka City University for their hospitality in August 2019 and March 2020 when some of this work was done.

We thank Tom Fisher, Mihai Fulger, S\'andor Kov\'acs, Sasha Polishchuk, and Brent Pym for useful conversations.

%%%%%%%%%%%%%%%%%%%%%%%%%%%%%%%%%%%%%%%%%%%%%%%%%%%%%%%%%%%%%%%%
%%%%%%%%%%%%%%%%%%%%%%%%%%%%%%%%%%%%%%%%%%%%%%%%%%%%%%%%%%%%%%%%
\section{Preliminary results about secants to $E \subseteq \PP_\cL$}
\label{sect.secants}
%%%%%%%%%%%%%%%%%%%%%%%%%%%%%%%%%%%%%%%%%%%%%%%%%%%%%%%%%%%%%%%%
%%%%%%%%%%%%%%%%%%%%%%%%%%%%%%%%%%%%%%%%%%%%%%%%%%%%%%%%%%%%%%%%

Secant varieties to $E \subseteq \PP^{n-1}$ have been well-studied. Partial secant varieties, not so much.
 For the most part, this section collects known results and presents them
in a form that is useful for this paper. The only new result is \Cref{prop.8.15.GvB-H} which shows that certain ``partial secant varieties'' are smooth.

%%%%%%%%%%%%%%%%%%%%%%%%%%%%%%%%%%%%%%%%%%%%%%%%%%%%%%%%%%%%%%%%
\subsection{The notation $n$, $E$, $\cL$, $H$, and $\Omega$}
\label{ssect.setup}
%%%%%%%%%%%%%%%%%%%%%%%%%%%%%%%%%%%%%%%%%%%%%%%%%%%%%%%%%%%%%%%%

For the rest of the paper, we fix an integer $n\geq 3$, a degree-$n$ elliptic normal curve $E \subseteq \PP^{n-1}$ over $\CC$, and define $\cL:=\cO_{\PP^{n-1}}(1)|_{E}$, which is an invertible $\cO_{E}$-module of degree $n$. Always, we make the identifications 
\begin{equation*}
\PP^{n-1} \; =\; \PP H^0(E,\cL)^* \; =  \; \PP \Ext^1(\cL,\cO_E) \; = :\; \PP_\cL.
\end{equation*}

Symmetric powers $E^{[d]}$, secant varieties $\Sec_{d}(E)$, their subvarieties $E^{[d]}_{x} \subseteq E^{[d]}$ and $\Sec_{d,x}(E)\subseteq \Sec_{d}(E)$, 
and related notations, such as the linear span $\overline{D}$ of an effective divisor $D$ and the summation map $\sigma:E^{[d]}\to E$, are defined in \cref{ssect.secant.vars}.

We also fix an effective divisor $H$ such that $\cO_E(H) \cong \cL$, 
and define 
\begin{equation*}
\Omega  \;   := \; \{\omega \in E \; | \; 2\omega=\s(H)\}.
\end{equation*}
Clearly, $\Omega$ is a coset of the $2$-torsion subgroup $E[2]=\{x\in E\;|\; 2x=0\}$, and does not depend on the choice of $H$.

%%%%%%%%%%%%%%%%%%%%%%%%%%%%%%%%%%%%%%%%%%%%%%%%%%%%%%%%%%%%%%%%
\subsection{The linear span of effective divisors on $E$}
\label{sect.div.E}
%%%%%%%%%%%%%%%%%%%%%%%%%%%%%%%%%%%%%%%%%%%%%%%%%%%%%%%%%%%%%%%%

Since $E$ is a smooth curve, every non-zero ideal in $\cO_E$ is a product of maximal ideals in a unique way.  
It follows that the map \begin{equation*}
  \{\text{length-$d$ subschemes of $E$}\} \; \longrightarrow \; E^{[d]}
\end{equation*}
that sends $\Spec(\cO_E/\fm_{x_1}^{r_1}\cdots \fm_{x_t}^{r_t})$ to the divisor $\sum_{i=1}^t r_i (x_i)$ is injective; it is obviously surjective, hence bijective.
Accordingly, we do not distinguish between closed subschemes of $E$ having length $d$, effective divisors of degree $d$, and points on $E^{[d]}$. 

Let $H$ be as in \cref{ssect.setup}, and set $z:=\s(H)$.  
Scheme-theoretic intersections of $E$ with the hyperplanes in $\PP_\cL$ are linearly equivalent to one another so all of them belong 
to $E^{[n]}_{z}$; every $D$ in $E^{[n]}_{z}$ arises in this way, and the morphism
$( \PP_\cL)^\vee \to E^{[n]}_{z}$, $L \mapsto E \cap L$, is an isomorphism whose inverse  is the map $D \mapsto \overline{D}$. 

Linear subspaces of $ \PP_\cL$ of dimension $d$ are called {\sf $d$-planes}, or {\sf planes} if we do not specify $d$.

\begin{proposition}\label{prop.sec.low.deg}
  If $d \le n-2$ and  $D \in E^{[d]}$, then $\overline{D}$ is the unique $(d-1)$-plane $L$ such that $E \cap L=D$.\footnote{A proof of a slightly weaker 
 version of this result,  which can easily be adapted to prove the proposition, can be found in \cite[Lem.~IV.1.1, p.~32]{Hulek86}.
 See, also, \cite[Lem.~13.2]{gvb-hul}.}
\end{proposition}

The case $d=n-1$ is a little different.

\begin{proposition}
  If $D \in E^{[n-1]}$, then $\overline{D}$ is the unique hyperplane $L$ such that $D \subseteq E \cap L$.
It follows that $E \cap  \overline{D} = D+(x)$ where $x= \s(H)-\s(D)$.
\end{proposition}

If $D_1$ and $D_2$ are effective divisors on $E$ we define
\begin{align*}
\gcd(D_1,D_2)  &\; :=\; D_1 \cap D_2,
\\
\lcm(D_1,D_2)  &\; :=\; D_1+D_2 - D_1 \cap D_2.
\end{align*}
If $L_1$ and $L_2$ are linear subspaces of $\PP_\cL$ we define
\begin{equation*}
\langle L_1,L_2 \rangle  \; :=\;  \text{the smallest linear subspace that contains $L_1 \cup L_2$}.
\end{equation*}

\begin{lemma}
  [Fisher]
  \cite[Lem.~2.6]{fis10} \cite[Lem.~9.4]{fis18}
  \label{le:fisher}
  If $D$, $D_1$, and $D_2$ are effective divisors on $E$,
  then  
  \begin{enumerate}
  \item\label{item.le.fisher.dim} 
    $\dim \overline{D} =
    \begin{cases}
      \deg D -1 & \text{if $ \deg D<n$,}
      \\
      n-2 & \text{if $D \sim H$,}
      \\
      n-1 & \text{otherwise.}
    \end{cases}
    $  
  \item\label{item.le.fisher.span} 
  $\langle \,  \overline{D_1},\overline{D_2}  \, \rangle = \overline{\lcm(D_1,D_2)}$.  
    $\phantom{\Big\vert}$
  \item\label{item.le.fisher.cap} 
    $\overline{D_1}\cap \overline{D_2}=  \overline{\gcd(D_1,D_2)}$   if $\deg(\lcm(D_1,D_2)) \le n$ and $\lcm(D_1,D_2) \not\sim H$.
  \end{enumerate}
\end{lemma}

\begin{proposition}
\label{prop.fish}
Let  $D_1$ and $D_2$ be effective divisors on $E \subseteq \PP^{n-1}$.  
\begin{enumerate}
\item\label{item.prop.fish.le}
If $\deg D_1+\deg D_2 <n $, then $D_1 \cap D_2=\varnothing$ if and only if $   \overline{D_1}\cap \overline{D_2} =\varnothing$.
\item\label{item.prop.fish.eq}
If $\deg D_1+\deg D_2 =n $ and $\lcm(D_1,D_2) \not\sim H$, then $D_1 \cap D_2=\varnothing$ if and only if $   \overline{D_1}\cap \overline{D_2} =\varnothing$.
\end{enumerate}
\end{proposition}
\begin{proof}
\cref{item.prop.fish.le}
Certainly $D_1 \cap D_2 = \varnothing$ if $\overline{D_1}\cap \overline{D_2}\ne \varnothing$.

Suppose the reverse implication is false. Then there are effective divisors $D_1$ and $D_2$ such that $D_1 \cap D_2=\varnothing$ but  
$\overline{D_1}\cap \overline{D_2} \ne \varnothing$. 
Since $D_1 \cap D_2=\varnothing$, $\lcm(D_1,D_2)=D_1+D_2$. Hence $\deg(\lcm(D_1,D_2))<n$ so, by \cref{le:fisher}\cref{item.le.fisher.cap}, 
$\overline{D_1}\cap \overline{D_2}=  \overline{\gcd(D_1,D_2)} = \overline{0} = \varnothing$. This is a contradiction so we conclude that the 
proposition holds. 

\cref{item.prop.fish.eq}
Assume $\deg D_1+\deg D_2 =n $ and $\lcm(D_1,D_2) \not\sim H$.

Certainly, $\overline{D_1}\cap \overline{D_2} = \varnothing$ implies $D_1 \cap D_2 = \varnothing$.
Conversely, if $D_1 \cap D_2 = \varnothing$, then $\lcm(D_1,D_2)=D_1+D_2 \not\sim H$ so, by   \cref{le:fisher}\cref{item.le.fisher.cap},
$\overline{D_1}\cap \overline{D_2}=  \overline{\gcd(D_1,D_2)} = \overline{0} = \varnothing$ . 
\end{proof}

\begin{remark}
\label{rem.n=4.cones}
The hypothesis  in \cref{prop.fish}\cref{item.prop.fish.le} that $\deg D_1+\deg D_2 <n$  can't be improved: if $n=4$ and $L_1$ and $L_2$ are distinct  lines on a 
singular quadric that contains $E$, then there are degree-two divisors $D_1$ and $D_2$ such that 
$L_1=\overline{D_1}$ and $L_2=\overline{D_2}$, and $D_1 \cap D_2 = \varnothing \ne \overline{D_1} \cap \overline{D_2}$. 
More precisely, if $n=4$ and $\omega \in \Omega$, then $\Sec_{2,\omega}(E)$ is a rank 3 quadric containing $E$, and 
a point in $\Sec_{2,\omega}(E)-E$ lies on either a unique secant line or on infinitely many;
those that lie on infinitely many secant lines are the vertices of the four quadric cones $\Sec_{2,\omega}(E)$.
\reqed
\end{remark}

%%%%%%%%%%%%%%%%%%%%%%%%%%%%%%%%%%%%%%%%%%%%%%%%%%%%%%%%%%%%%%%%
\subsection{The varieties $\Sec_{d,x}(E)$ and $\Sec_{d,x}(E)-\Sec_{d-1}(E)$}
%%%%%%%%%%%%%%%%%%%%%%%%%%%%%%%%%%%%%%%%%%%%%%%%%%%%%%%%%%%%%%%%

The next result says that every point in $ \Sec_{d,x}(E) -\Sec_{d-1}(E)$ belongs to $\overline{D}$ for a unique $D \in E^{[d]}_x$
when $1 \le d <\frac{n}{2}$. 

\begin{proposition}
\label{cor.Sec'.dz.disj.union}
Let $x \in E$.   If either $d \in [1,\frac{n}{2})$ or $d=\frac{n}{2}$ and $x \notin \Omega$, then 
  \begin{equation}
  \label{eq:Sec.dx.disj.union}
    \Sec_{d,x}(E) -\Sec_{d-1}(E) \;=\; \text{the disjoint union}  \quad \bigsqcup_{D \in  E^{[d]}_x} (\overline{D}-\Sec_{d-1}(E)). 
  \end{equation}
\end{proposition}
\begin{proof}
Let $d \in [1,\frac{n}{2})$ and suppose there are distinct divisors $D_1,D_2 \in E^{[d]}_x$ such that 
\begin{equation*} 
\big(\overline{D_1}-\Sec_{d-1}(E) \big) \cap \big(\overline{D_2}-\Sec_{d-1}(E) \big) \, \ne \, \varnothing.
\end{equation*}
Let $p$ be a point in this intersection.  Since $d<\frac{n}{2}$, $\deg D_1+\deg D_2 <n$ whence  $p\in \overline{D_1} \cap \overline{D_2}= \overline{\gcd(D_1,D_2)}$;
hence $D_1 \cap D_2 \ne \varnothing$. In particular, $\deg(D_1 \cap D_2) \le d-1$, so $p \in \Sec_{d-1}(E)$. This contradicts the choice of $p$ so we conclude there can be no such $p$.
The equality in \cref{eq:Sec.dx.disj.union}  therefore holds.

Now assume that $d=\frac{n}{2}$ and $x \notin \Omega$. Suppose $D_1,D_2 \in E^{[n/2]}_x$ and 
\begin{equation*}
p \in \big(\overline{D_1}-\Sec_{\tfrac{n}{2}-1}(E) \big) \cap \big(\overline{D_2}-\Sec_{\tfrac{n}{2}-1}(E) \big).
\end{equation*}
Since $x \notin \Omega$, $D_1+D_2 \not\sim H$.  
Therefore $p \in \overline{D_1} \cap \overline{D_2}= \overline{\gcd(D_1,D_2)}$
 where the equality follows from   \cref{le:fisher}\cref{item.le.fisher.cap}.  
Since $p \not\in \Sec_{\tfrac{n}{2}-1}(E)$, $\deg(\gcd(D_1,D_2))= \frac{n}{2}$ from which it follows that $D_1=\gcd(D_1,D_2) =D_2$. Hence
the union in \cref{eq:Sec.dx.disj.union} is disjoint as claimed. 
\end{proof}

If $n$ is even, then \cref{cor.Sec'.dz.disj.union} fails when $d=\frac{n}{2}$  and $x \in \Omega$ -- see \cref{lem.fd.xi} and \cref{prop.ssnsmth}: 
points in the singular locus of $ \Sec_{d,x}(E) -\Sec_{d-1}(E)$ belong to a pencil of $d$-secants $\overline{D}$, $D \in E^{[d]}_x$.

\begin{corollary}
\label{cor.Sec.d.disj.union}
  If $d \in [1,\frac{n}{2})$, then 
  \begin{equation}
  \label{eq:Sec.d.disj.union}
    \Sec_{d}(E) -\Sec_{d-1}(E) \;=\; \text{the disjoint union}  \quad \bigsqcup_{x \in E} (\Sec_{d,x}(E)-\Sec_{d-1}(E)). 
  \end{equation}
\end{corollary}
\begin{proof}
The left- and right-hand sides of \cref{eq:Sec.d.disj.union} are equal so we only need to check that the union on the right is disjoint.
Suppose $p \in  (\Sec_{d,x}(E)-\Sec_{d-1}(E)) \cap (\Sec_{d,y}(E)-\Sec_{d-1}(E))$. Then $p \in \overline{D}_1 \cap \overline{D_2}$ for some $D_1 \in E^{[d]}_x$ and 
some $D_2 \in E^{[d]}_y$. Since $\deg D_1+\deg D_2 <n$, $\overline{D_1} \cap \overline{D_2}= \overline{\gcd(D_1,D_2)}$. But $p \notin \Sec_{d-1}(E)$ so
$\deg(\gcd(D_1,D_2)) \ge d$; hence $\deg(\gcd(D_1,D_2)) =d$ and we conclude that $\gcd(D_1,D_2) =D_1=D_2$, whence $x=\s(D_1)=\s(D_2)=y$. The 
union is therefore disjoint.
\end{proof}

\begin{proposition}
\label{prop.dim.Sec.dz}
\label{prop.dim.Sec.dz.2}
\label{prop.S.dx.irred}
If $d \le \frac{n}{2}$, then $ \Sec_{d,x}(E)$ is irreducible of dimension 
$2d-2$ for all $x \in E$.\footnote{If $d>\tfrac{n}{2}$, then $\Sec_{d,x}(E)=\PP_{\cL}$ because $\Sec_{d,x}(E)\supseteq \Sec_{d-1}(E)=\PP_{\cL}$.}.
\end{proposition}
\begin{proof} 
Consider the incidence variety
$ \EE  := \{ (D,\xi) \in  E^{[d]}_x \times \PP_\cL   \; | \; \xi \in \overline{D}\} $
 and the projections
 \begin{equation}
 \label{eq:incid.var.EE}
 \xymatrix{
  & \EE \ar[dl]_\tau \ar[dr]^\nu
  \\
  E^{[d]}_x && \PP_\cL.
 }
 \end{equation}
 Since every $\overline{D}$ is a $\PP^{d-1}$ in $\PP_\cL$, $\EE$ is a $\PP^{d-1}$-bundle over $E^{[d]}_x\cong \PP^{d-1}$ and therefore smooth and irreducible 
 because  $E^{[d]}_x$ is. Since $\Sec_{d,x}(E)$ equals $\nu(\EE)$, it is also irreducible. Since $\dim E^{[d]}_x =d-1$,   
 $\dim \EE=2d-2$.  

If $1 \le d < \frac{n}{2}$, then  \cref{cor.Sec'.dz.disj.union} says that $\nu$ is 
injective on $\nu^{-1}( \Sec_{d,x}(E) - \Sec_{d-1}(E))$; i.e.,  $\nu$ is injective on a non-empty Zariski-open subset of 
$\EE$, so $\dim\Sec_{d,x}(E)=\dim\EE$.

If $n=2r$ is even, then  \cref{prop.room}\cref{item.room.eq} says that $\dim\Sec_{r,x}(E)=2r-2$.
\end{proof}

\subsubsection{Remark}
With regard to the last sentence in the proof of \Cref{prop.dim.Sec.dz}, \cref{cor.room} says that when $n=2r$, the varieties 
$\Sec_{r,x}(E)$, $x \in E$,  form a pencil of degree-$r$ hypersurfaces that contain $\Sec_{r-1}(E)$.
This is the generalization to higher dimensions of the familiar fact that the 
quartic elliptic normal curve $E \subseteq \PP^3$ is contained in a pencil of quadrics.

By \cite[Prop.~8.15]{gvb-hul}, $\Sec_{d}(E) - \Sec_{d-1}(E)$ is smooth when $d<\frac{n}{2}$. 
The proof of that result can be modified to give the next result.

\begin{proposition}
\label{prop.8.15.GvB-H}
Let $x \in E$. 
If $d<\frac{n}{2}$ or $d=\frac{n}{2}$ and $x \notin \Omega$, then $\Sec_{d,x}(E)-\Sec_d(E)$ is smooth.
\end{proposition}
\begin{proof}
We will follow the proof of \cite[Prop.~8.15]{gvb-hul}, 
but our notation differs from theirs: their $d-1$ is our $d$, 
their $\cO_E(nA)$ is our $\cL$, their $\Lambda_D$ is our $\overline{D}$, their $S^dE$ is our 
$E^{[d]}$; we replace their incidence variety $I \subseteq E^{[d]}\times \PP_\cL $ by the smaller incidence variety $\EE$ in \cref{eq:incid.var.EE};
our morphism $\nu$ is the restriction of their $\nu$ to $\EE \subseteq I$.
 As noted in \Cref{prop.dim.Sec.dz}, $E^{[d]}_x\cong\PP^{d-1}$ and $\EE$ is a $\PP^{d-1}$-bundle over $E^{[d]}_x$, therefore smooth and
  irreducible of dimension $2d-2$.

Let $Y:=\nu^{-1}(\Sec_{d,x}(E) - \Sec_{d-1}(E))$.

Since $Y$ is an open subscheme of $\EE$ it is smooth. We will show that $ \Sec_{d,x}(E) - \Sec_{d-1}(E)$ is smooth by showing that the 
restriction $\nu|_Y:Y \to \Sec_{d,x}(E) - \Sec_{d-1}(E)$ is an isomorphism. 
By  \cite[Cor.~14.10]{H92}, to show that $\nu|_Y$ is an isomorphism it suffices to show it is injective
 and that $d\nu_y:T_y Y \to T_{\nu(y)}\PP_\cL$ is injective for all $y \in Y$. That is what we will do. 
 
First, $\nu|_Y$ is injective. 
 If $\nu(D_1,\xi_1)=\nu(D_2,\xi_2)$, then $\xi_1=\xi_2$ and this point belongs to  $\overline{D_1} \cap \overline{D_2}$ so,  
 by \Cref{cor.Sec'.dz.disj.union}, $D_1=D_2$.

The rest of this proof will show that $d\nu_y$ is injective for all $y \in Y$.

As usual, $\cT_X$  denotes the tangent sheaf on a variety $X$, and $T_xX$ the tangent space  at a point $x \in X$.

Let $D \in E^{[d]}_x$ and identify $\overline{D}$ with $\tau^{-1}(D)  \subseteq \EE$. 
Since $\EE$ is a bundle over $E^{[d]}_x$, there is 
 a commutative diagram  (cf., \cite[Prop.~8.15]{gvb-hul}) of sheaves on $\overline{D}\cong \PP^{d-1}$ with exact rows 
 \begin{equation}
 \label{diag.replaces.gvb.diag.p53}
 \xymatrix{
 0 \ar[r] & \cT_{\overline{D}} \ar@{=}[d] \ar[r] & \cT_{\EE} \big\vert_{\overline{D}} \ar[d]^{d\nu}  \ar[r] &  
 T_D E^{[d]}_x  \otimes \cO_{\overline{D}} \ar[d]^{\phi} \ar[r] &0 
 \\
 0 \ar[r] & \cT_{\overline{D}} \ar[r] & \cT_{\PP_\cL} \big\vert_{\overline{D}} \ar[r]   &  \cN_{\overline{D}/\PP_\cL}  \ar[r] &0
 }
 \end{equation}
 in which $\cN_{\overline{D}/\PP_\cL}$ denotes the normal bundle to $\overline{D}$ in $\PP_\cL$. 
 (In the top row $\overline{D}$ denotes $\tau^{-1}(D)$.  In the second row it denotes a linear subspace of $\PP_\cL$.)
 To show that $d\nu_y$ is injective when $y =(D,\xi) \in Y$ (and  either $d <\frac{n}{2}$ or $d=\frac{n}{2}$ and $x \notin \Omega$)
 it suffices to show that $\phi_\xi$ is injective.

As noted in \cite[Prop.~8.15]{gvb-hul}, there is a canonical isomorphism between $T_D E^{[d]}$ and $H^0(E,\cO_D(D))$. 
Because our concern is $T_D E^{[d]}_x $ rather than $T_DE^{[d]}$, we need a slightly different result. 

\underline{Claim:} $T_D E^{[d]}_x$, the tangent space at $D$ to the hypersurface $E^{[d]}_x \subseteq E^{[d]}$, is the hyperplane in $T_D E^{[d]}$
that  is the image of the map $\a$ in the long exact cohomology sequence
      \begin{equation}
      \label{eq:defn.alpha}
        0 \, \to \, H^0(E,\cO_E)\,\to  \,H^0(E,\cO_E(D)) \, \stackrel{\a}{\longrightarrow} \, H^0(E,\cO_D(D)) \, (=T_D E^{[d]})  \, \stackrel{\b}{\longrightarrow} \, 
        H^1(E,\cO_E) \, \to \, 0.
      \end{equation}
\underline{Proof:} 
(This is a special case of \cite[Exer.~2.5(a), p.~17]{hrt_def}.)   
As in \cite[Prop.~2.6, p.~13]{hrt_def}, for example, $H^1(E,\cO_E)$ is the set of  isomorphism classes of infinitesimal deformations of a line bundle on 
$E$. On the other hand, $T_D E^{[d]}=H^0(E,\cO_D(D))$ is the set of infinitesimal deformations of $D$ as a 
subvariety of $E^{[d]}$, and  $\b$ is the map that sends a deformation of $D$ to the corresponding deformation of $\cO_E(D)$.
Hence $\ker(\b)$ is the set of infinitesimal deformations of $D$ that {\it do not change the isomorphism class of the corresponding line bundle};
 but the isomorphism class of $\cO_E(D)$ is completely determined by $\s(D)$, which is $x$,
so $\ker(\b)$ consists of the infinitesimal deformations of $D$ {\it as a subvariety of $E^{[d]}_x$}.
But the  infinitesimal deformations of $D$ as a subvariety of $E^{[d]}_x$ are in natural bijection with the elements of $T_D E^{[d]}_x$, so $\ker(\b)$, which is 
$\im(\a)$, equals $T_D E^{[d]}_x$. 
 $\lozenge$
 
As is well-known, 
 if $L$ is a linear subspace of a projective space $\PP$ and $\cI_L$  the ideal in $\cO_{\PP}$ vanishing on $L$, 
there is a canonical isomorphism $H^0(\PP,\cI_L(1))^* \otimes \cO_L(1) \to  \cN_{L/\PP}$. 
We will now describe $H^0(\PP,\cI_L(1))^*$ in terms of  $E$ when $L=\overline{D} \subseteq \PP_\cL$. 
The image of the natural map 
$H^0(\PP,\cI_L(1)) \to H^0(\PP,\cO_{\PP}(1))$ is the space of linear forms on $\PP$ that vanish on $L$. When $L=\overline{D} \subseteq \PP_\cL$, 
$H^0(\PP_\cL,\cO_{\PP_\cL}(1)) = H^0(E, \cL)$ and the condition that a linear form vanishes on $\overline{D}$ is the same as the condition that its 
restriction to $E$ vanishes on $D$, so the map $H^0(\PP_\cL,\cI_L(1)) \to H^0(\PP_\cL,\cO_{\PP_\cL}(1))$ is the same as the map
$H^0(E, \cL \otimes \cI_D) = H^0(E, \cI_D\cL) \to H^0(E, \cL)$, where $\cI_D$ is the ideal in $\cO_E$ vanishing on $D$. Hence the  map $\phi$ in \cref{diag.replaces.gvb.diag.p53} is 
 \begin{equation*}
 \phi:\im(\a) \otimes \cO_{\overline{D}} \, \longrightarrow \, \ \cN_{\overline{D}/\PP_\cL}
= H^0(E, \cI_D\cL)^* \otimes \cO_{\overline{D}} (1).
  \end{equation*} 
In particular, $\phi$ is the restriction to $\im(\a) \otimes \cO_{\overline{D}}$  of the map
 \begin{equation}
 \label{eq:phi}
 \xymatrix{
H^0(E,\cO_D(D)) \otimes \cO_{\overline{D}}  
\ar[rr] &&H^0(E, \cI_D\cL)^* \otimes \cO_{\overline{D}} (1),
}
  \end{equation} 
  that occurs in the third display on page 53 of  \cite{gvb-hul} (our $\cL$ replaces their $\cO_E(nA)$, our $\cI_D\cL$ replaces their $\cO_E(nA-D)$
  and  our $\phi$ is the restriction of their $d\nu'$).  
  Since  \cite{gvb-hul} proves \cref{eq:phi} is injective at all points in $\tau^{-1}(\overline{D})$ when $d<\frac{n}{2}$, we conclude that 
  $\phi$ is also injective at all such points  when $d<\frac{n}{2}$. 
  Hence $d\nu_y$ is injective when $d<\frac{n}{2}$. 
  
As is remarked in loc. cit.,   the map in \cref{eq:phi} is induced by the composition
 \begin{equation}
 \label{eq:phi.2}
 \xymatrix{
H^0(E,\cO_D(D)) \otimes H^0(E, \cI_D \cL)   
\ar[r] & H^0(E,\cO_D(D)) \otimes H^0(D, \cI_D \cL \otimes \cO_D)  \ar[r] & H^0(D, \cL \otimes \cO_D).
}
  \end{equation} 
  The second arrow in this composition comes from identifying $\cO_D(D)$ with $\cI_D^{-1} \otimes \cO_D$, then 
  using  the multiplication map  
$(\cI_D^{-1} \otimes \cO_D) \otimes_{\cO_D} (\cI_D \cL \otimes \cO_D) \to \cL \otimes \cO_D$ (which is an isomorphism).
The first arrow is induced by the restriction map $H^0(E, \cI_D \cL)  \to H^0(D, \cI_D \cL \otimes \cO_D)$.

 Now assume $d=\frac{n}{2}$ and $x \notin \Omega$. In this case we need an additional step before using the argument in \cite[Prop.~8.15]{gvb-hul}.
 The reason for this is that the proof in loc. cit. for $d<\frac{n}{2}$ uses the fact that the natural map 
 $H^0(E, \cI_D \cL) \to H^0(D,\cI_D \cL \otimes \cO_D)$ is surjective  
 (which is why the first arrow in \cref{eq:phi.2}, equivalently, the vertical arrow in the fourth display on page 53 of  \cite{gvb-hul},  
 is surjective)   when $d<\frac{n}{2}$
  (our $d$ is their $d-1$) but now we need this surjectivity when $d=\frac{n}{2}$ and $x \notin \Omega$. 
  
 To see that this map is surjective replace $\cI_D\cL$ by $\cL(-D)$ and consider the cohomology sequence associated to the sequence 
 $0 \to \cL(-2D) \to \cL(-D) \to \cL(-D) \otimes \cO_D \to 0$. If $d<\frac{n}{2}$, then $\deg \cL(-2D)>0$ so $H^1(E,\cL(-2D))=0$ 
 which implies that the map $H^0(E,\cL(-D)) \to  H^0(D,\cL(-D) \otimes \cO_D)$
  is surjective.  
  When $d=\frac{n}{2}$ and $x \notin \Omega$, $\deg \cL(-2D)=0$ but $ \cL(-2D) \not\cong \cO_E$ because $x \notin \Omega$ so, again, 
  $H^1(E,\cL(-2D))=0$; the map $H^0(E,\cL(-D)) \to  H^0(D,\cL(-D) \otimes \cO_D)$ is therefore surjective
  (in fact, an isomorphism)
and  the argument in \cite[Prop.~8.15]{gvb-hul} completes the proof.
\end{proof}

Given \cref{prop.8.15.GvB-H}, it is natural to ask the following question: when $n=2r$, is $\Sec_{r,\omega}(E)-\Sec_{r-1}(E)$ smooth when 
$\omega \in \Omega$? The answer is no: \Cref{th:nsplitev} shows that $\Sec_{r,\omega}(E)-\Sec_{r-1}(E)$ is the disjoint union of the 
homological leaves $L(\cE_\omega)$ and $L(\cL_\omega \oplus \cL_\omega)$;  
\Cref{th:smth} shows that $L(\cE_\omega)$ and $L(\cL_\omega \oplus \cL_\omega)$ are smooth; 
\Cref{prop.ssnsmth}
 shows that
\begin{equation*}
{\rm Sing} \big( \Sec_{r,\omega}(E)-\Sec_{r-1}(E)\big) \;=\; L(\cL_\omega \oplus \cL_\omega).
\end{equation*}

%%%%%%%%%%%%%%%%%%%%%%%%%%%%%%%%%%%%%%%%%%%%%%%%%%%%%%%%%%%%%%%%
\subsection{Results of T.G. Room}
%%%%%%%%%%%%%%%%%%%%%%%%%%%%%%%%%%%%%%%%%%%%%%%%%%%%%%%%%%%%%%%%

The next result, \cref{prop.room}, is due to Room \cite[9.22.1, 9.26.1]{room38}. 
The notation and terminology in Room's book is that of 1938; it will challenge the modern reader. 
Fortunately there are ``modern'' proofs of the results we need: 
see Fisher's unpublished paper \cite{fisher2006} and his published paper \cite{fis10}; see 
\cite[Lem.~7.1, Prop.~7.2]{fisher2006} and \cite[Lems~2.7 and 2.9]{fis10}.

The following terminology is taken from  \cite{fis10}.

A {\sf divisor pair} $(D_1,D_2)$ is a pair of effective divisors on $E$ such that $D_1+D_2 \sim H$.
Given a divisor pair $(D_1,D_2)$, let $\{e_i\}$ and $\{f_j\}$ be bases for $H^0(E,\cO_E(D_1))$ and $H^0(E,\cO_E(D_2))$, respectively,
and define
\begin{equation*} 
\Phi(D_1,D_2) \; :=\;  \text{the matrix whose $ij^{\rm th}$ entry is $\mu(e_i \otimes f_j)$}
\end{equation*}
 where  
\begin{equation*}
\mu:H^0(E,\cO_E(D_1)) \otimes H^0(E,\cO_E(D_2)) \, \longrightarrow \, H^0(E,\cL) 
\end{equation*}
is the multiplication map. Elements in $H^0(E,\cL)$ are linear forms on $\PP H^0(E,\cL)^* = \PP_\cL$, so the entries in $\Phi(D_1,D_2)$
are linear forms on $\PP_\cL$. 

Divisor pairs $(D_1,D_2)$ and $(D_1',D_2')$ are {\sf equivalent} if either
$D_1 \sim D_1'$ or $D_1 \sim D_2'$.
\reqed

\begin{proposition}
[Room]
\label{prop.room}
Let $(D_1, D_2)$ be a divisor pair and let $x_i:=\s(D_i)$; note that $x_1+x_2=\s(H)$.
\begin{enumerate}
  \item\label{item.room.leq} 
  If $d:=\deg D_1 \le \deg D_2$, then 
 \begin{equation*}
 \Sec_{d,x_1}(E) \;=\;  \bigcup_{D \sim D_1} \overline{D} \; = \; \bigl\{p \in \PP^{n-1} \; | \; \rank \Phi(D_1,D_2)_p < d\bigr\},
 \end{equation*} 
where $\Phi(D_1,D_2)_p$ denotes the evaluation of the matrix $\Phi(D_1,D_2)$ at the point $p$.\footnote{Of course, this ``value'' is 
only defined up to a non-zero scalar multiple but that does not affect its rank.}
  \item\label{item.room.eq}  
  If  $n=2r$ and $\deg D_1=\deg D_2=r$, then 
\begin{equation*}
\Sec_{r,x_1}(E) \;=\; \Sec_{r,x_2}(E)  \;=\;  \bigl\{p \in \PP^{n-1} \; \big\vert \; \det \Phi(D_1,D_2)_p =0 \bigr\}.
\end{equation*}
In particular, $\dim(\Sec_{r,x}(E))=n-2$. 
  \item\label{item.room.ineq}  
   If  $n=2r$ and $(D_1,D_2)$ and  $(D_1',D_2')$ are inequivalent divisor pairs such that all $D_i$'s and $D_i'$'s have degree $r$, then
  \begin{equation*}
\Sec_{r-1}(E) \;=\;  \bigl\{p \in \PP^{n-1} \; \big\vert \; \det\Phi(D_1,D_2)_{p} = \det\Phi(D_1',D_2')_{p} =0\bigr\}. 
\end{equation*} 
\end{enumerate}
\end{proposition}
\begin{proof}
	See \cite[Lem.~2.7]{fis10} or \cite[Lem.~7.1]{fisher2006} for \cref{item.room.leq}. As in the proof of \cite[Prop.~7.2(i)]{fisher2006}, \cref{item.room.eq} is a special case of \cref{item.room.leq}. See \cite[Lem.~2.9]{fis10} or \cite[Prop.~7.2(ii)]{fisher2006} for \cref{item.room.ineq}.
\end{proof}

\begin{corollary}
[Room]
\label{cor.room}
Assume $n=2r$. 
\begin{enumerate}
  \item 
  $\Sec_{r-1}(E)$ has codimension 2 in $\PP^{n-1}$ and is the intersection of two hypersurfaces of degree $r$.
  \item 
The varieties $\Sec_{r,x}(E)$, $x \in E$, form a pencil, parametrized by the projective line $E/\!\sim$ where $x \sim \s(H)-x$, 
of degree-$r$ hypersurfaces that contain $\Sec_{r-1}(E)$.
\end{enumerate} 
\end{corollary}

\begin{proposition}
  If $n=2r$, then $\Sec_{r,x}(E) =\Sec_{r,y}(E)$ if and only if either $x=y$ or $x+y = \s(H)$. 
\end{proposition}
\begin{proof}
($\Leftarrow$)
This is the content of \cref{prop.room}\cref{item.room.eq}.  

($\Rightarrow$)
Suppose the claim is false; i.e., $\Sec_{r,x}(E) =\Sec_{r,y}(E)$ and $x \ne y$ and $x+y\ne \s(H)$. 
 
 Let $p \in  \Sec_{r,x}(E)-\Sec_{r-1}(E)$. Such $p$ exists since $\codim\Sec_{r,x}(E)=1$ and $\codim\Sec_{r-1}(E)=2$.
 Let $D \in E^{[r]}_x$ and $D' \in E^{[r]}_y$  be such that $p \in \overline{D} \cap 
 \overline{D'}$.  

If $\lcm(D,D') \sim H$, then $\deg(\lcm(D,D'))=n$, whence $\lcm(D,D')=D+D'$. But then $D+D'\sim H$ so $x+y=\s(D+D')=\s(H)$ which 
contradicts the assumption that $x+y \ne \s(H)$. We conclude that $\lcm(D,D') \not \sim H$. 
However, $\deg(\lcm(D,D')) \le \deg(D+D')=n$ so, by 
 \cref{le:fisher}\cref{item.le.fisher.cap}, $\overline{D} \cap \overline{D'} = \overline{\gcd(D,D')}$. 
 Thus $p \in  \overline{\gcd(D,D')}$. But $p \notin  \Sec_{r-1}(E)$ so $\deg(\gcd(D,D'))=r$. Hence $D=\gcd(D,D')=D'$, which implies that
 $x=\s(D)=\s(D')=y$; this contradicts the assumption that  $x \ne y$ so we conclude that the proposition is true. 
\end{proof}

%%%%%%%%%%%%%%%%%%%%%%%%%%%%%%%%%%%%%%%%%%%%%%%%%%%%%%%%%%%%%%%%
%%%%%%%%%%%%%%%%%%%%%%%%%%%%%%%%%%%%%%%%%%%%%%%%%%%%%%%%%%%%%%%%
\section{The bundles $\cE \in \Bun(2,\cL)$ for which $L(\cE) \ne \varnothing$ and their trivial subbundles}
\label{sect.rank.2.buns}
%%%%%%%%%%%%%%%%%%%%%%%%%%%%%%%%%%%%%%%%%%%%%%%%%%%%%%%%%%%%%%%%
%%%%%%%%%%%%%%%%%%%%%%%%%%%%%%%%%%%%%%%%%%%%%%%%%%%%%%%%%%%%%%%%

As we said in \cref{ssect.setup}, $\cL=\cO_{\PP^{n-1}}(1)|_E$. Thus, $\cL$ is an invertible $\cO_E$-module of degree $\ge 3$.

Each $\xi\in\PP_\cL=\PP\Ext^1(\cL,\cO_E)$ is represented by a non-split exact sequence\footnote{See \Cref{sect.appx.extns} for the basics on extensions and their isomorphism classes.}
\begin{equation*}
	\xymatrix{
		0 \ar[r] &\cO_E \ar[r] & \cE \ar[r] & \cL \ar[r] & 0.
	}
\end{equation*}
The middle term, $\cE$, is denoted by $m(\xi)$. It is an element in the set
\begin{equation*}
	\Bun(2,\cL) \;:=\; \{\text{rank-two locally free $\cO_E$-modules $\cE$ such that $\det\cE\cong\cL$}\}.
\end{equation*}
For each $\cE\in\Bun(2,\cL)$, we define
\begin{equation*}
	L(\cE)\;:=\;m^{-1}(\cE)\;=\;\{\xi \in \PP\Ext^1(\cL,\cO_E)  \; \,  |  \, \; \text{the middle term of $\xi$ is isomorphic to $\cE$} \}.
\end{equation*}

%%%%%%%%%%%%%%%%%%%%%%%%%%%%%%%%%%%%%%%%%%%%%%%%%%%%%%%%%%%%%%%%
\subsection{An assumption on $\cE$}
\label{ssect.assump}
%%%%%%%%%%%%%%%%%%%%%%%%%%%%%%%%%%%%%%%%%%%%%%%%%%%%%%%%%%%%%%%%

We will say that $\cE$ satisfies the assumptions in \cref{ssect.assump} if
\begin{enumerate}
\item\label{item.asmp.one}
$\cE\in\Bun(2,\cL)$, and
\item\label{item.asmp.two}
no invertible $\cO_{E}$-module of degree $\leq 0$ is a direct summand of $\cE$.
\end{enumerate}
We will show that \cref{item.asmp.one} and \cref{item.asmp.two} hold if and only if $L(\cE)\neq\varnothing$.
It is easy to see that \cref{item.asmp.one} and \cref{item.asmp.two} hold if $L(\cE)\neq\varnothing$  (\cref{thm.good.Es}). 
The converse is proved in \Cref{cor.good.E's}.

%%%%%%%%%%%%%%%%%%%%%%%%%%%%%%%%%%%%%%%%%%%%%%%%%%%%%%%%%%%%%%%%
\subsection{The bundles $\cE_{d,x}$, $\cE_{o}$, and $\cE_{\omega}$ in $\Bun (2,\cL)$}
\label{sect.relevant.cE's}
%%%%%%%%%%%%%%%%%%%%%%%%%%%%%%%%%%%%%%%%%%%%%%%%%%%%%%%%%%%%%%%%

We will show, eventually, that the $\cE$'s in the title of this subsection are only $\cE$'s, up to isomorphism, such that $L(\cE) \ne \varnothing$.

%%%%%%%%
\subsubsection{Indecomposable bundles in $\Bun (2,\cL)$}
\label{ssect.bun.notn}
%%%%%%%%

The degree of a locally free $\cO_E$-module is defined to be the degree of its determinant.

Suppose $n=2r$. If $\omega \in \Omega$, we define 
\begin{align*}
 \cL_\omega  &\; :=\; \cO_E\big((\omega)+ (r-1) \cdot (0)\big),
 \\
 \cE_\omega  &\; :=\; \text{the unique non-split self-extension of } \cL_\omega.
\end{align*}
We have $\cL_{\omega}^{\otimes 2} \cong \cL$ because $\cO_E(2(\omega)+ (n-2) \cdot (0)) \cong  \cO_E((2\omega)+ (n-1) \cdot (0))$.
Up to isomorphism the four $\cL_\omega$'s  are the only $\cN$'s such that $\cN^{\otimes 2} \cong \cL$.  
By Atiyah's classification of indecomposable bundles on an elliptic curve, the four $\cE_\omega$'s  are pairwise non-isomorphic and, up to isomorphism, 
are the only indecomposables in $\Bun(2,\cL)$ (see \cite[Thm.~7]{Atiyah} and \cite[Cor.~V.2.16]{hrt}).

If $n$ is odd and $\cE \in \Bun(2,\cL)$, then the rank and degree of $\cE$ are relatively prime so, by Atiyah's classification, 
 \cite[Cor.~(i) to Thm.~7]{Atiyah}, there is a {\it unique} indecomposable $\cE$ in $\Bun(2,\cL)$;
this also follows from \cite[Thm.~10]{Atiyah} with his $(r,d)$ equal to $(2,n)$. 
Up to isomorphism, that $\cE$ is
\begin{equation*}
 \cE_{o} \; :=\;  \text{the unique non-split extension of $\cO_E((\omega)+\tfrac{n-1}{2} \cdot (0))$ by $\cO_E((\omega)+\tfrac{n-3}{2} \cdot (0))$}
\end{equation*}
where $\omega$ is any element in $\Omega$; the isomorphism class of $\cE_o$ does not depend on the choice of $\omega$ because $2(\omega)+(n-2)\cdot(0) \sim 
(2\omega)+(n-1)\cdot (0) = (\s(H))+(n-1)\cdot (0) \sim H$.

%%%%%%%%
\subsubsection{The decomposable $\cE$'s in $\Bun (2,\cL)$ for which $L(\cE)\ne \varnothing$}
\label{ssect.dec.buns}
%%%%%%%%

If $d \in [1,\frac{n}{2}]$ and $x \in E$, we define 
\begin{equation*}
 \cE_{d,x} \; :=\; \cO_E(D)  \oplus \cL(-D)
\end{equation*}
where $D \in E^{[d]}_x$. The isomorphism class of $\cO_E(D)$, and hence that of $\cE_{d,x}$, does not depend on the choice of $D$ 
in $E^{[d]}_x$. If $n=2r$   and $\omega\in\Omega$, then $\cE_{r,\omega}\cong\cL_{\omega}\oplus\cL_{\omega}$.

The next four results are elementary, but we record them for completeness.

\begin{lemma}
\label{lem.possible.Es}
Let $\cL$ be any invertible $\cO_E$-module, and let $\xi \in \Ext^1(\cL,\cO_E)$ be the  extension 
\begin{equation*}
\xymatrix{0 \ar[r] &\cO_E \ar[r]^-f & \cE \ar[r] & \cL \ar[r] & 0}.
\end{equation*}
 It follows that
\begin{enumerate}
  \item\label{item.lem.possible.Es.lf} 
$\cE$ is locally free of rank two;
  \item\label{item.lem.possible.Es.isom} 
$\det\cE \cong \cL$;  
  \item\label{item.lem.possible.Es.dsum} 
if  $\cE = \cL_1 \oplus \cL_2$, where  $\cL_1$ and $\cL_2$ are invertible $\cO_E$-modules, then 
\begin{enumerate}
  \item\label{item.lem.possible.Es.dsum.a} 
  $\cL_1 \otimes \cL_2 \cong \cL$, and
  \item\label{item.lem.possible.Es.dsum.b} 
 if $\xi \ne 0$ and $\deg\cL \ge 1$, then  both $\cL_1$ and $\cL_2$ have positive degree.
\end{enumerate}
\end{enumerate}
\end{lemma}
\begin{proof}
\cref{item.lem.possible.Es.lf}
Trivial.

\cref{item.lem.possible.Es.isom}
If $0 \to \cF_1 \to \cF \to \cF_2 \to 0$ is an exact sequence of locally free $\cO_E$-modules, then $\det\cF \cong (\det\cF_1) \otimes (\det\cF_2)$ (see \cite[Exer.~II.5.16(d)]{hrt}). Note also that, if $\cN$ is an invertible $\cO_E$-module, then $\det\cN=\cN$.

\cref{item.lem.possible.Es.dsum}\cref{item.lem.possible.Es.dsum.a}
This is a special case of \cref{item.lem.possible.Es.isom} because the determinant of a line bundle is itself.

\cref{item.lem.possible.Es.dsum}\cref{item.lem.possible.Es.dsum.b}
  Suppose $\deg\cL \ge 1$ and the sequence does not split.

  For $j\in \{1,2\}$, let $\pi_j:\cL_1 \oplus \cL_2 \to \cL_j$ be the projection.

Suppose $H^0(E,\cL_1)=0$.  Then (1) $\pi_1f=0$ so $f(\cO_E) \subseteq \ker(\pi_1) = \cL_2$, and (2) $\deg\cL_1 \le 0$, 
so $\deg\cL_2 \ge \deg\cL \ge 1$. Thus $f(\cO_E)$ is a proper submodule of $\cL_2$.   Since $\coker(f)$ contains a copy of $\cL_2/f(\cO_E)$, it is not torsion-free and 
therefore not isomorphic to $\cL$. This is a contradiction so we conclude that $H^0(E,\cL_1) \ne 0$. 

Hence $\deg\cL_1 \ge 0$. We will show that $\deg\cL_1>0$. Since $H^0(E,\cL_1) \ne 0$, it suffices to show that 
$\cL_1 \not \cong \cO_E$. Suppose to the contrary that $\cL_1 \cong \cO_E$. 
 Then $\deg\cL_2=\deg\cL\ge 1$. 
If $f(\cO_E) \subseteq \cL_2$, then $\cL_2/f(\cO_E)$, and hence $\coker(f)$, is not torsion-free so, as in the previous paragraph,
$\coker(f)$ would not be isomorphic to $\cL$. Hence $f(\cO_E) \not\subseteq \cL_2$. Thus  $\pi_1f(\cO_E) \ne 0$,
whence $\pi_1 f$ is an isomorphism $\cO_E \to \cL_1$. Let $g:\cL_1 \oplus \cL_2 \to \cO_E$ be the map $((\pi_1f)^{-1},0)$.
The image of $g \circ f$ is a non-zero map $\cO_E \to \cO_E$ so some scalar multiple of $g$ splits the sequence, contradicting the hypothesis that the 
sequence does not split. 

We therefore conclude that $\deg\cL_1>0$. Similarly,  $\deg\cL_2>0$.
\end{proof}

\begin{lemma}
\label{lem.easy.non-split}
  Fix positive-degree invertible $\cO_E$-modules $\cN_1$ and $\cN_2$. If $\deg(\cN_1 \otimes \cN_2) \ge 3$, 
  then there is a non-split extension of the form
  \begin{equation}
    \label{eq:easy.non-split}
     0 \to \cO_E \to \cN_1 \oplus \cN_2 \to \cN_1 \otimes \cN_2 \to 0.
  \end{equation}
\end{lemma}
\begin{proof}
  Since $\deg\cN_1+\deg\cN_2 \ge 3$, there are effective divisors $D_1$ and $D_2$ such that $\cN_i \cong \cO_E(D_i)$ and $D_1\cap D_2 =
   \varnothing$.  Let $\cI_j$ denote the ideal vanishing on $D_j$.

  Since $D_1 \cap D_2 = \varnothing$, $\cI_1+\cI_2=\cO_E$ and $\cI_1 \cap \cI_2=\cI_1\cI_2 \cong \cO_E(-D_1-D_2) $.

  The kernel of the map $\cI_1 \oplus \cI_2 \to \cI_1+\cI_2=\cO_E$, $(a,b) \mapsto a-b$, is $\cI_1 \cap \cI_2$. Hence there is an exact sequence $0 \to \cI_1\cI_2 \to \cI_1 \oplus \cI_2 \to \cO_E \to 0$, which does not split because $D_1$ and $D_2$ are effective.  If we now apply the functor $\cO_E(D_1+D_2) \otimes -$, which is isomorphic to $\cN_1 \otimes \cN_2 \otimes -$, to the sequence we obtain a non-split sequence of the form \cref{eq:easy.non-split}.
\end{proof}

\begin{proposition}\label{prop.leaves.for.decomp.Es}
  Let $\cE\in\Bun(2,\cL)$. If $\cE$ is decomposable, then the following are equivalent:
  \begin{enumerate}
  \item\label{item.leaves.for.decomp.Es.LE} $L(\cE) \ne \varnothing$.
  \item\label{item.leaves.for.decomp.Es.deg} $\cE = \cN_1 \oplus \cN_2$ where $\cN_1$ and $\cN_2$ are  positive-degree invertible $\cO_{E}$-modules.
  \item\label{item.leaves.for.decomp.Es.Edx} $\cE \cong \cE_{d,x}$ for some $d\in [1,\frac{n}{2}]$ and $x\in E$.
  \end{enumerate}
\end{proposition}
\begin{proof}
	\cref{lem.possible.Es,lem.easy.non-split} show \cref{item.leaves.for.decomp.Es.LE}$\Leftrightarrow$\cref{item.leaves.for.decomp.Es.Edx}. Since \cref{item.leaves.for.decomp.Es.Edx}$\Rightarrow$\cref{item.leaves.for.decomp.Es.deg} is obvious, it remains to show \cref{item.leaves.for.decomp.Es.deg}$\Rightarrow$\cref{item.leaves.for.decomp.Es.Edx}. Since $n=\deg\cE=\deg\cN_1+\deg\cN_2$, we can assume that $d:=\deg\cN_1\in [0,\tfrac{n}{2}]$. If $\sigma(\cN_{1})=x$, then $\cN_{1}\cong\cO_{E}(D)$ for some $D\in E^{[d]}_{x}$. Since $\cN_{1}\otimes\cN_{2}\cong\cL$, we conclude that $\cN_{2}\cong\cL(-D)$ and $\cE\cong\cE_{d,x}$.
\end{proof}

We summarize the content of \cref{ssect.bun.notn,ssect.dec.buns}.

\begin{proposition}
\label{thm.good.Es}
If $L(\cE)\ne \varnothing$, then $\cE$ satisfies the assumptions in \cref{ssect.assump}. 
More particularly:
\begin{enumerate}
  \item\label{item.good.Es.indec.even}
  If $\cE$ is indecomposable and $n$ is even, then $\cE\cong\cE_{\omega}$ for some $\omega\in\Omega$.
  \item\label{item.good.Es.indec.odd}
  If $\cE$ is indecomposable and $n$ is odd, then $\cE\cong\cE_{o}$.
  \item\label{item.good.Es.dec}
  If $\cE$ is decomposable, then $\cE\cong\cE_{d,x}$ for some $d\in[1,\tfrac{n}{2}]$ and $x \in E$. 
  In particular, $\cE$ has no invertible direct summand of degree $\le 0$.
\end{enumerate} 
\end{proposition}
\begin{proof}
In \cref{ssect.bun.notn} we made note of the fact that \cref{item.good.Es.indec.even} and \cref{item.good.Es.indec.odd} 
hold (without assuming that $L(\cE)\ne \varnothing$). 
	
	\cref{item.good.Es.dec}
	Suppose $\cE$ is decomposable. By \cref{prop.leaves.for.decomp.Es}, $\cE\cong\cE_{d,x}$ for some $d\in[1,\tfrac{n}{2}]$ and $x \in E$, 
	Since the Krull-Schmidt theorem holds for locally free $\cO_{E}$-modules of finite rank (\cite[Thm.~3]{AtiyahKS}), 
	$\cE$ does not have an invertible $\cO_{E}$-module of degree $\leq 0$ as a direct summand.
\end{proof}

The converse of \cref{thm.good.Es} is proved in \cref{cor.good.E's}. It shows that 
$L(\cE)\ne \varnothing$ if and only if $\cE$ is one of the $\cE$'s listed in \cref{item.good.Es.indec.even}, \cref{item.good.Es.indec.odd}, \cref{item.good.Es.dec}.
 \cref{prop.leaves.for.decomp.Es} showed that $L(\cE_{d,x})\neq\varnothing$, so to prove  \cref{cor.good.E's} 
  it remains to show that $L(\cE)\neq\varnothing$ when $\cE \in \{\cE_o\} \cup \{\cE_\omega \; | \; \omega \in \Omega\}$, depending on 
  the parity of $n$. 

%%%%%%%%%%%%%%%%%%%%%%%%%%%%%%%%%%%%%%%%%%%%%%%%%%%%%%%%%%%%%%%%
\subsection{Conventions for locally free modules and vector bundles}
%%%%%%%%%%%%%%%%%%%%%%%%%%%%%%%%%%%%%%%%%%%%%%%%%%%%%%%%%%%%%%%%

Sometimes we  consider vector bundles, and their associated projective-space bundles, on  a variety $X$ as schemes over $X$ via the familiar correspondence of \cite[Exer.~II.5.18]{hrt} and \cite[Exer.~II.7.10]{hrt}. Here we use the notation
  \begin{equation}\label{eq:vbcorresp}
    \left(\text{locally free sheaf }\cE\right)
    \leftrightsquigarrow
    \begin{cases}
      \VV \cE\text{ or }\VV(\cE):=\text{the vector bundle }\mathrm{Spec}\left(\mathrm{Sym}(\cE^\vee)\right)\\
      \PP\cE\text{ or }\PP(\cE):=\text{the projective bundle }\mathrm{Proj}\left(\mathrm{Sym}(\cE^\vee)\right),
    \end{cases}    
  \end{equation}
  where $\mathrm{Sym}(\cE^\vee)$ denotes the symmetric algebra on the dual sheaf  $\cE^\vee$, \cite[Exer.~II.5.16]{hrt}, and the $\mathrm{Spec}$ and 
  $\mathrm{Proj}$ constructions are as in \cite[Exer.~II.5.17]{hrt} and \cite[Construction before Ex.~II.7.8.7]{hrt}. 
  One can recover $\cE$ as the sheaf of sections of $\VV \cE$ \cite[Exer.~II.5.18 (c)]{hrt}. 
  We conflate $\cE$ and $\VV \cE$   and refer to $\cE$ itself as a bundle; this is not uncommon: e.g., see
   \cite[Ch.~0]{3264}.\footnote{
 Although the correspondence in \Cref{eq:vbcorresp} is often used: for example
  \begin{itemize}
  \item in \cite[\S B.5.5]{Fulton-2nd-ed-98} the {\it projective bundle of $\cE$} is what \Cref{eq:vbcorresp} refers to as $\PP\cE$, and is denoted there by $P(E)$ for $E:=\VV \cE$;
  \item similarly, the {\it projectivization of $\cE$} in \cite[\S 9.1]{3264} is our $\PP\cE$. As explained there, the points of $\PP\cE$ over $x\in X$ are precisely the lines in the vector space $\cE_x$ (fiber of $\cE$ at $x$);
  \end{itemize}  
  it is {\it not always observed}:  for example, \cite[Note following \S B.5.5]{Fulton-2nd-ed-98} observes that $P(E)$ (our $\PP\cE$) would be the $\PP(\cE^\vee)$ in \cite[\S 8.4]{ega2} and, similarly, the $\mathbf{P}(\cE)$ of \cite[Definition preceding Prop.~II.7.11]{hrt} omits the dualization in \Cref{eq:vbcorresp}. The same goes for the $\mathbf{P}(E)$ in \cite[Appendix A]{lzf1}; if anything, the convention opposite ours (i.e., \Cref{eq:vbcorresp} without dualization) might be more common.
  
  This difference in conventions can cause some confusion (it did for some of us): it explains, for instance, the apparent discrepancy in defining {\it Chern classes} in various sources: compare the simple summation of \cite[Defn.~5.10]{3264} with the {\it alternating} sum of \cite[\S A.3, Defn.]{hrt}. The two convention switches ($\cE$ versus $\cE^\vee$ and the sign difference) cancel out to accord the same significance to the Chern classes $c_k(\cE)$ of a locally free sheaf $\cE$ as introduced in either reference.  
  }

%%%%%%%%%%%%%%%%%%%%%%%%%%%%%%%%%%%%%%%%%%%%%%%%%%%%%%%%%%%%%%%%
\subsection{The $\Aut(\cE)$-action on the quasi-affine variety $X(\cE):= \Hom(\cO_E,\cE) - (Z_\cE \cup \{0\})$}
%%%%%%%%%%%%%%%%%%%%%%%%%%%%%%%%%%%%%%%%%%%%%%%%%%%%%%%%%%%%%%%%

As the next lemma suggests, the following definition plays a central role in all that follows; also see \cref{rmk.defn.psi_E}.

\begin{definition}
Let $\cE$ be a rank-two locally free $\cO_E$-module. Define
\begin{align}
   Z_\cE  &  \; : = \;       \{\text{non-zero maps }  f:\cO_E \to \cE \; | \; \coker(f) \text{ is not invertible}\} 
  \\
  &
  \; \phantom{:} = \; 
   \{\text{non-zero maps }  f:\cO_E \to \cE \; | \; \coker(f) \text{ is not torsion-free}\}   
        \notag
   \\
 &  \;  \phantom{:}\subseteq \; \Hom_E(\cO_E,\cE) \, - \, \{0\}
 \notag
\end{align}
and
\begin{equation*} 
X(\cE) \; :=\; \Hom_E(\cO_E,\cE) \, - \,  (Z_\cE  \cup \{0\}).
\end{equation*} 
\end{definition}

Geometrically, $X(\cE)$ consists of the  embeddings of the trivial bundle $E \times \CC$ as a subbundle of  $\VV(\cE)$.

\begin{lemma}
\label{lem.tfree.coker}
Let $\cE$ be a rank-two locally free $\cO_E$-module. Suppose $\cO_E$ is not a direct summand of $\cE$.
\begin{enumerate}
\item{}
$X(\cE)\,=\, \{\text{non-zero maps }  f:\cO_E \to \cE \; | \; \coker(f) \cong \det\cE\}$.
\item{}
If $f\in X(\cE)$, then there is a unique-up-to-non-zero-scalar-multiple exact sequence
  \begin{equation*}
  \xymatrix{
 \xi: \quad   0 \ar[r] &  \cO_E \ar[r]^f & \cE \ar[r] &  \det\cE \ar[r] &  0.
    }
  \end{equation*}
  \item
If 
$\det\cE \cong \cL$ and $X(\cE) \ne \varnothing$, then $L(\cE) \ne \varnothing$.
\end{enumerate} 
\end{lemma}
\begin{proof}
Suppose $f\in X(\cE)$. Since $\rank\cO_{E}=1$, $f$ is monic. So there is a short exact sequence
  \begin{equation*}
  \xymatrix{
    0 \ar[r] &  \cO_E \ar[r]^f & \cE \ar[r] & \coker(f) \ar[r] &  0
    }
  \end{equation*}
  of locally free $\cO_{E}$-modules, which implies that $\det\cE\cong\det\cO_{E}\otimes\det(\coker(f))\cong\coker(f)$. The ``uniqueness'' follows from \cref{prop.isom.extns}. Since $\cO_E$ is not a direct summand of $\cE$, the sequence does not split.
\end{proof}

The left action of $\End(\cE)$ on $\Hom(\cO_E,\cE)$ via composition makes $\Hom(\cO_E,\cE)$ into a left $\End(\cE)$-module. That action restricts to give a left action of the automorphism group $\Aut(\cE)$ on $\Hom(\cO_E,\cE)$ and hence on $X(\cE)$.

\begin{proposition}\label{le:Aut.cE.action.on.extns}
Let $\cE$ be an arbitrary $\cO_E$-module, let $f_1,f_2: \cO_E \to \cE$ be arbitrary monomorphisms, and let $g_i$ be a cokernel of $f_i$.  
The extensions
  \begin{equation*}
  \xymatrix{
 \xi_{i}: \quad   0 \ar[r] &  \cO_E \ar[r]^-{f_{i}} & \cE \ar[r]^-{g_{i}} &  \coker(f_i) \ar[r] &  0
    }
  \end{equation*}
are isomorphic if and only if $f_2 \in \Aut(\cE) f_1$.
\end{proposition}
\begin{proof}
  ($\Rightarrow$) Suppose $\xi_1 \cong \xi_2$. There is a commutative diagram
\begin{equation*}
  \xymatrix{
    0 \ar[r] &  \cO_E \ar[r]^{ f_1}  \ar[d]_{\mu} & \cE \ar[r]^<<<<<{g_1}  \ar[d]_{\nu} &  \coker(f_1)  \ar[r] \ar[d]_{\tau} & 0 
    \\
    0 \ar[r] &  \cO_E \ar[r]_{ f_2}  & \cE \ar[r]_>>>>>{g_2}& \coker(f_2)  \ar[r]  & 0 
  }
\end{equation*}
in which $\mu$, $\nu$, and $\tau$ are isomorphisms. Since $\End(\cO_E)=\CC$, $\mu$ is a scalar multiple of the identity. We can therefore treat $\mu$ as a scalar and rewrite the equality $\nu f_1=f_2\mu$ as $f_2= (\mu^{-1} \nu)f_1$. Hence $f_2 \in \Aut(\cE) f_1$.

($\Leftarrow$) 
If $f_2= \nu f_1$ for some $\nu  \in \Aut(\cE)$, then there is a commutative diagram as above with $\mu$ being the identity
map $\cO_E \to \cO_E$, and $\tau$ being some isomorphism, so $\xi_1 \cong \xi_2$.
\end{proof}

\begin{corollary}
\label{cor.defn.Psi.cE}
Let $\cE$ be a locally free $\cO_E$-module. If $\rank\cE=2$ and $\deg\cE \ge 1$,
then the map $f \mapsto \xi$ in  \Cref{lem.tfree.coker} descends to a  surjective set map
\begin{equation}
\label{eq:quo.map}
\Psi_\cE: X(\cE)   \, \longrightarrow \, L(\cE)
\end{equation}
whose fibers are the $\Aut(\cE)$-orbits in $X(\cE)$. 
\end{corollary}

\begin{remark}
\label{rmk.defn.psi_E}
\Cref{le:ismor} shows $\Psi_\cE$ is a morphism. 
We observe in \Cref{le:dim.Aut.cE} that $\Aut(\cE)$ is an affine algebraic group. 
\Cref{th:smth} shows $L(\cE)$ is the geometric quotient of $X(\cE)$ modulo the action of $\Aut(\cE)$.
\Cref{pr:z-codim1} shows $Z_\cE \cup \{0\}$ is a hypersurface in $\Hom(\cO_E,\cE)$, so $X(\cE)$ is quasi-affine.
\reqed
 \end{remark}

\begin{lemma}\label{le:autfree}
Suppose $\cE$ satisfies the assumptions in \cref{ssect.assump}.
\begin{enumerate}
\item\label{item.le.autfree.gen} If $f \in X(\cE)$, then the left $\End(\cE)$-module generated by $f$ is free.
\item\label{item.le.autfree.free} $\Aut(\cE)$ acts freely on $X(\cE)$.
  \end{enumerate}
\end{lemma}
\begin{proof}
\cref{item.le.autfree.gen}
Suppose to the contrary that the cokernel of $0\neq f \in \Hom_E(\cO_E,\cE)$ is torsion-free and 
$\gamma f=0$ for some non-zero $\gamma \in \End(\cE)$. The hypotheses on $\cE$ imply $\coker(f) \cong \cL$.
Let $g:\cE \to \cL$ be a cokernel of $f$. By the universal property of the cokernel there is a unique $h:\cL \to \cE$ 
such that $\gamma = hg$. Since $\gamma \ne 0$, $h \ne 0$. Since $\deg\cL>0$, the image of $h$ is not contained in the image of $f$.
Hence $gh$ is a non-zero endomorphism of $\cL$ which is invertible. It follows that the sequence 
splits. 

\cref{item.le.autfree.free}
This is an immediate consequence of \cref{item.le.autfree.gen}. 
\end{proof}

Suppose $\cE$ satisfies the assumption in \cref{ssect.assump}. Our next task, which we complete in \cref{cor.good.E's}, 
is to show $L(\cE)\neq\varnothing$; i.e., there is a non-split extension of the form $0 \to \cO_E \to \cE \to \cL \to 0$.
By \cref{lem.easy.non-split}, $L(\cE_{d,x}) \ne \varnothing$ for all $x \in E$ and all $d \in [1,\frac{n}{2}]$ so,  
by \cref{lem.tfree.coker}, it remains to show that if $\cE\cong\cE_o$ or $\cE\cong\cE_\omega$ ($\omega \in \Omega$), 
depending on the parity of $n$, then there is a non-zero map $f:\cO_E \to \cE$ whose cokernel is torsion-free; i.e., 
it remains to show that $X(\cE)\neq\varnothing$.

%%%%%%%%%%%%%%%%%%%%%%%%%%%%%%%%%%%%%%%%%%%%%%%%%%%%%%%%%%%%%%%%%%%%%%%%%%%%%%%%%%%%%%%%
\subsection{The isomorphism $X(\cE) \cong {\rm Epi}(\cE,\det \cE)$}
\label{sect.props.of.cE}
%%%%%%%%%%%%%%%%%%%%%%%%%%%%%%%%%%%%%%%%%%%%%%%%%%%%%%%%%%%%%%%%%%%%%%%%%%%%%%%%%%%%%%%%

The next lemma  is part of \cite[Exer.~II.5.16(b)]{hrt}.

\begin{lemma}\label{le:edete}
Let $\cE$ be a rank-two locally free $\cO_E$-module. Under the adjunction $\left(-\otimes\cE\right) \dashv \left(-\otimes\cE^\vee\right)$ the 
epimorphism $\cE \otimes \cE \to \wedge^2\cE=\det\cE$ corresponds to an $\Aut(\cE)$-equivariant isomorphism 
   \begin{equation*}
 \nu: \cE \to  (\det \cE) \otimes \cE^\vee .
  \end{equation*}
   \end{lemma}
   \begin{proof} 
     By definition, $\det\cE = \wedge^2\cE$. The adjunction isomorphism $\Hom(\cE \otimes\cE,\wedge^2\cE) \cong 
     \Hom(\cE, \wedge^2\cE \otimes \cE^\vee)$ sends the epimorphism $\cE \otimes \cE \to \wedge^2\cE$ to a homomorphism 
     $\nu:\cE \to \wedge^2 \cE \otimes \cE^\vee$. 
     The fact that the map $\cE \otimes \cE \to \wedge^2\cE$ gives a non-degenerate pairing on stalks implies that $\nu$ is an isomorphism on stalks, 
 and hence an isomorphism. The fact that $\nu$ is $\Aut(\cE)$-equivariant is a general fact about rigid symmetric monoidal abelian
 categories $(\cC,\otimes,\mathbf{1})$; see \cref{rem.equiv.ex}.
    \end{proof}
    
Let  $\nu$ be as in  \cref{le:edete} and let
\begin{equation*}
\nu_*: \Hom(\cO_E,\cE)  \, \longrightarrow \,    \Hom(\cO_E, (\det \cE) \otimes \cE^\vee)
\end{equation*}
be the map $\nu_*(a)=\nu a$; clearly, $\nu_*$ is $\Aut(\cE)$-equivariant because $\nu$ is. Let
\begin{equation*}
\theta: \Hom(\cO_E,( \det \cE) \otimes \cE^\vee)  \, \longrightarrow \,    \Hom(\cE,\det \cE)
\end{equation*}
be the isomorphism associated to the adjunction $\left(-\otimes\cE\right) \dashv \left(-\otimes\cE^\vee\right)$; it is $\Aut(\cE)$-equivariant by \cref{prop.adj.equiv}.

\begin{lemma}\label{le.mor.corresp}
	Let $\cE$ be a rank-two locally free $\cO_E$-module. The $\Aut(\cE)$-equivariant isomorphism
	\begin{equation}\label{eq.theta.nu}
		\theta\circ\nu_{*}: \Hom(\cO_E,\cE)\, \longrightarrow \,\Hom(\cE,\det \cE)
	\end{equation}
	has the following properties:
	\begin{enumerate}
		\item\label{item.mor.corresp.zero} 
		If $s\in\Hom(\cO_E,\cE)$, then $(\theta \nu_{*})(s) \circ s=0$.
		\item\label{item.mor.corresp.ex} 
		If $s_1,s_2\in\Hom(\cO_E,\cE)$, then  $(\theta \nu_{*})(s_1)\circ s_{2}=-(\theta \nu_{*})(s_2)\circ s_{1}$.
	\end{enumerate}
\end{lemma}
\begin{proof}
	Since both statements can be verified locally, it suffices to prove the corresponding claims for a rank-two free module $F:=\cE_{x}$ over the stalk $R:=\cO_{E,x}$ at each point $x\in E$. Fix a basis $(e_{1},e_{2})$ for $F$ and let $(e_{1}^{\vee},e_{2}^{\vee})$ be its dual basis.
	
	The isomorphism $\nu$ induces the  $R$-module isomorphism
	\begin{equation*}
		F\,\longrightarrow\,(\det F)\otimes F^{\vee},\qquad v\mapsto\sum_{i}(v\wedge e_{i})\otimes e_{i}^{\vee},
	\end{equation*}
	and $\theta$ induces the $R$-module isomorphism
	\begin{equation*}
		(\det F)\otimes F^{\vee}\,\cong\,\Hom_R(R,(\det F)\otimes F^{\vee})\,\longrightarrow\,\Hom_R(F,\det F),\qquad\omega\otimes g \mapsto \bigl(u\mapsto g(u)\omega\bigr),
	\end{equation*}
	so $\theta\circ\nu_{*}$ induces the $R$-module isomorphism
	\begin{equation}
	\label{third.isom}
		F\,= \,\Hom_R(R,F)\,\longrightarrow\,\Hom_R(F,\det F),\qquad v\mapsto\bigl(u\mapsto\sum_{i}e_{i}^{\vee}(u)(v\wedge e_{i})=v\wedge u\bigr).
	\end{equation}
	
	\cref{item.mor.corresp.zero}
	 If the isomorphism \cref{third.isom} sends $s\in\Hom_R(R,F)$ to $\varphi\in\Hom_R(F,\det F)$, 
	then $\varphi(u)=s(1)\wedge u$ for all $u\in F$, so $(\varphi\circ s)(1)=s(1)\wedge s(1)=0$. Therefore $\varphi\circ s=0$.
	
	\cref{item.mor.corresp.ex}
	If the isomorphism \cref{third.isom}  sends $s_{1}$  to $\varphi_{1}$ and $s_{2}$ to $\varphi_{2}$, then
	\begin{equation*}
		(\varphi_{1}\circ s_{2})(1) \, = \, s_{1}(1)\wedge s_{2}(1) \,=\, -s_{2}(1)\wedge s_{1}(1)\,=\, -(\varphi_{2}\circ s_{1})(1).
	\end{equation*}
	Therefore $\varphi_{1}\circ s_{2}=-\varphi_{2}\circ s_{1}$.
\end{proof}

    \subsubsection{Notation for epimorphisms}
    Given quasi-coherent $\cO_E$-modules $\cF$ and $\cG$, we define
    \begin{equation}
    \label{eq:epi.notn}
    \pushQED{\qed}
    \renewcommand\qedsymbol{$\lozenge$}
    {\rm Epi}(\cF,\cG) \; :=\; \{\text{epimorphisms }\pi:\cF \to \cG\} \;\subseteq\; \Hom(\cF,\cG). \qedhere
	\popQED
    \end{equation}

    \begin{lemma}
      \label{re:kercoker}
  Let $\cE$ be a rank-two  locally free $\cO_E$-module. 
 The composition of the linear isomorphisms in the top row of the following diagram restricts to an $\Aut(\cE)$-equivariant bijection 
 between the sets in the bottom row:
 \begin{equation*}
  \label{eq:imp.bijection.Z}
 \xymatrix{
  \Hom(\cO_E,\cE) \ar[r]^-{\nu_*} & \Hom(\cO_E,  (\det\cE) \otimes \cE^\vee) \ar[r]^-{\theta}& \Hom(\cE,\det \cE)
  \\
\ar@{^{(}->}[u]     \phantom{\big\vert}  X(\cE)  \ar[rr] &&   {\rm Epi}(\cE,\det\cE).   \ar@{^{(}->}[u]  \phantom{\big\vert}
 }
 \end{equation*}
 
 If  $f\in X(\cE)$ and $\pi=\theta\nu_*(f)$, then the sequence
 $
 \xymatrix{ 0 \ar[r] & \cO_{E} \ar[r]^-{f} & \cE \ar[r]^-{\pi} & \det\cE \ar[r] &  0}
$ 
is exact.
 \end{lemma}
 \begin{proof}
Let $f:\cO_E\to \cE$ be any non-zero homomorphism and define $\pi:=\theta\nu_*(f)=\theta(\nu f)$.
We will show that $\coker(f)$ is torsion-free if and only if $\pi$ is epic.

Since $\pi f=0$ by \cref{le.mor.corresp}\cref{item.mor.corresp.zero}, the morphism $\pi:\cE\to\det\cE$ is the composition of the cokernel morphism $\cE\to\coker(f)$ and a non-zero morphism $g:\coker(f)\to\det\cE$. If $\coker(f)$ is torsion-free, then $\coker(f)\cong\det\cE$ by \cref{lem.tfree.coker}, so $g$ is an isomorphism, and hence $\pi$ is an epimorphism.

Conversely, if $\pi:\cE \to \det\cE$ is epic, then the sequence 
$\xymatrix{
		0 \ar[r] & \ker(\pi) \ar[r] & \cE \ar[r]^-{\pi} & \det\cE \ar[r] &  0
	}
$	
is exact so $\ker(\pi)\cong\cO_{E}$. Since $\pi f=0$ a similar argument shows that $f$ is the composition of an isomorphism $\cO_{E}\to\ker(\pi)$ and the kernel morphism $\ker(\pi)\to\cE$. Hence $\coker(f)\cong\det\cE$, which is torsion-free.
\end{proof}

\begin{remark}
  The bijection $f \longleftrightarrow \pi =\theta(\nu f)$ in \cref{eq:imp.bijection.Z} puts the two middle arrows in exact sequences of the form
$ 0 \to \cO_E \to \cE \to \det\cE \to 0$
in bijective correspondence with each other. 
  We will freely switch perspective between the two points of view, parametrizing the non-split extensions of $\det\cE$ by $\cO_E$ 
  either by torsion-free-cokernel monomorphisms $\cO_E\to \cE$ or by epimorphisms $\cE\to \det\cE$.
  \reqed
\end{remark}

\begin{lemma}
\label{lem.varphi.f}
If $\xymatrix{0 \ar[r] &  \cO_E \ar[r]^f & \cE  \ar[r]^\pi \ar[r] & \cL  \ar[r] &  0}$ is exact,
then $\Hom(\cE, \cL) f = \pi \Hom(\cO_E,\cE)$.
\end{lemma}
\begin{proof}
	By \cref{re:kercoker} and the universality of cokernels, we can assume that $\cL=\det\cE$, and  that $f$ and $\pi$ correspond to each other via 
	$\theta\circ\nu_{*}$ in \cref{eq.theta.nu}; i.e., $\pi=(\theta\nu_{*})(f)$. 
	
	To prove ``$\subseteq$'', let $\pi'\in\Hom(\cE,\cL)$; then $\pi'=(\theta \nu_*)(f')$ for a 	unique $f'\in\Hom(\cO_{E},\cE)$;  
	by \cref{le.mor.corresp}\cref{item.mor.corresp.ex}, $\pi' f = (\theta \nu_*)(f') \circ f  = -  (\theta \nu_*)(f) \circ f' \in \pi \Hom(\cO_E,\cE)$.
	
	To prove ``$\supseteq$'', let $f'\in\Hom(\cO_{E},\cE)$; then 
	$\pi f'=(\theta\nu_{*})(f) \circ f' = -(\theta\nu_{*})(f') \circ f \in  \Hom(\cE, \cL) f$.
\end{proof}

%%%%%%%%%%%%%%%%%%%%%%%%%%%%%%%%%%%%%%%%%%%%%%%%%%%%%%%%%%%%%%%%
\subsection{Serre duality and the notation $\xi^\perp$}
\label{sect.SD}
%%%%%%%%%%%%%%%%%%%%%%%%%%%%%%%%%%%%%%%%%%%%%%%%%%%%%%%%%%%%%%%%

Fix an isomorphism $t: \Ext^1(\cO_E,\cO_E)  \to \CC$ of vector spaces and define the map
\begin{equation} 
\label{eq:SD}
\Ext^1(\cL,\cO_E) \times \Hom(\cO_E,\cL) \, \longrightarrow \, \CC,    \qquad (\eta,s) \mapsto t(\eta\cdot s),
\end{equation}
where $\eta\cdot s \in \Ext^1(\cO_E,\cO_E)$ is the pullback of $\eta$ along the homomorphism $s:\cO_E \to \cL$ 
(we remind the reader of the definition of this pullback in the proof of  \cref{lem.xi.perp}).
The fact that the map in \cref{eq:SD} is non-degenerate is the essence of Serre duality in this situation. 

We use the pairing in \cref{eq:SD} to make the identifications
\begin{align} 
\label{eq:SD.identification}
\Ext^1(\cL,\cO_E)^*    & \; = \; \Hom(\cO_E,\cL),  \qquad    t(- \cdot s) = s,
\\
\label{eq:SD.identification.2}
\Ext^1(\cL,\cO_E)\phantom{i}   & \; = \; \Hom(\cO_E,\cL)^*,  \qquad    \eta = t(\eta \cdot -).
\end{align}
If $\eta \in \Ext^1(\cL,\cO_E)$, we define
\begin{equation}
  \label{eq:xi.perp1}
  \eta^\perp \; := \; \{s \in \Hom(\cO_E,\cL) \; | \;  \eta \cdot s=0\}.
\end{equation}
Although the identifications in \cref{eq:SD.identification} and \cref{eq:SD.identification.2} depend on the choice of $t$, $\eta^\perp$ does not.  Given a point $\xi =\CC \eta \in \PP_\cL$ we define
\begin{equation}
\label{defn.xi.perp2}
\xi^\perp \; := \; \eta^\perp.
\end{equation}
The next lemma gives a description of  $\xi^\perp$ that will be used repeatedly in what follows.
 
The symbol $\perp$  is also used in the following way:
given a subspace $W \subseteq  \Hom(\cO_E,\cL)$, we define
\begin{equation}
\label{defn.W.perp}
W^\perp \; := \; \{ \lambda \in \Hom(\cO_E,\cL)^* \; | \; \lambda(s)=0 \text{ for all } s \in W \}.
\end{equation}
When $\codim W=1$ we often consider $W^\perp$ as a point in $\PP \Hom(\cO_E,\cL)^* = \PP\Ext^1(\cL,\cO_E)=\PP_\cL$;
because $W^\perp \in \PP_\cL$ we can use \cref{defn.xi.perp2} to define $(W^\perp)^\perp$; we have $(W^\perp)^\perp=W$. Thus
the definitions in \cref{defn.xi.perp2,defn.W.perp} are compatible with each other.

Given a homomorphism $\pi:\cE \to \cL$, we define the map
\begin{equation}
\label{eq:pi_*}
\pi_*:     \Hom(\cO_E,\cE) \, \longrightarrow \, \Hom(\cO_E,\cL), \qquad  \pi_*(b):=\pi b.
  \end{equation}
  
\begin{lemma}
\label{lem.xi.perp}
If $\xi\in \Ext^1(\cL,\cO_E)$ is the non-split extension $\xymatrix{ 0 \ar[r] & \cO_E \ar[r]^f & \cE \ar[r]^>>>>>\pi & \cL \ar[r] & 0,}$ then
 \begin{equation}
\label{eq:xi.perp=}
\im(\pi_*) \;=\;  \pi \Hom(\cO_E,\cE) \;=\; \Hom(\cE,\cL)f \;=\; \xi^\perp.
\end{equation}
Under the Serre duality identifications in \cref{eq:SD.identification}, $\xi^\perp=\im(\pi_*)$ and $\CC\xi=\im(\pi_*)^\perp$. 
\end{lemma}
\begin{proof}
The first equality in \Cref{eq:xi.perp=} is the definition of $\im(\pi_*)$. The second is the content of \Cref{lem.varphi.f}.

Applying $\Hom(\cO_E,-)$ to  $\xi$ yields an exact sequence
\begin{equation*}
\xymatrix{ 
\Hom(\cO_E,\cE) \ar[r]^-{\pi_*}  & \Hom(\cO_E,\cL) \ar[r]^-{\d} & \Ext^1(\cO_E,\cO_E)  \ar[r] & 0=  \Ext^1(\cO_E,\cE).
}
\end{equation*}
Since the image of $\pi_*$ is the kernel of $\d$, to complete the proof we must show that an element $s$ in  $\Hom(\cO_E,\cL)$ is in the image of $\pi_*$
if and only if $\xi \cdot s=0$, i.e., if and only if the pull-back of $\xi$ along $s$ is a split extension. The pullback  $\xi \cdot s$ is, by definition, 
  the top row in the pullback diagram
\begin{equation}
    \xymatrix{
      \xi\cdot s : \quad 0 \ar[r] & \cO_E \ar[r]  \ar@{=}[d]&  \cE \times_\cL \cO_E \ar[d] \ar[r]^>>>>{\pi'}  & \cO_E \ar[r]  \ar[d]^s & 0 
      \\
      \xi: \quad 0 \ar[r] & \cO_E \ar[r]  &  \cE  \ar[r]_\pi  &  \cL \ar[r]  & 0.
    }
  \end{equation}
The top row splits (i.e., $\xi \cdot s=0$) if and only if there is a map $g:\cO_E \to \cE$ such that 
$\pi g = s$; i.e., if and only if $s \in \im(\pi_*)$. Hence $\xi^\perp=\im(\pi_*)$. The equality $\CC \xi=\im(\pi_*)^\perp$ now follows from the remarks
about $(W^\perp)^\perp$ just before this lemma.
\end{proof}

%%%%%%%%%%%%%%%%%%%%%%%%%%%%%%%%%%%%%%%%%%%%%%%%%%%%%%%%%%%%%%%%
\subsection{The subvarieties $Z_x \subseteq Z_\cE$ and  properties of $\PP Z_\cE \subseteq \PP \Hom(\cO_E,\cE)$}
%%%%%%%%%%%%%%%%%%%%%%%%%%%%%%%%%%%%%%%%%%%%%%%%%%%%%%%%%%%%%%%%

In this section we show that  if $\cE$ satisfies the assumptions in \cref{ssect.assump}, then $\PP Z_\cE$ is a hypersurface in
$\PP\Hom(\cO_E,\cE)$. (At this stage we don't even know that $Z_\cE \cup \{0\}$ is Zariski-closed in $\Hom(\cO_E,\cE)-\{0\}$.)
It will follow that $L(\cE)$ is non-empty for all $\cE$ in \cref{ssect.assump}.

 For each $x \in E$, let
\begin{equation}
  Z_x \; :=\; \{\text{non-zero maps } f:\cO_E \to \cE  \; | \; \coker(f) \text{ contains a copy of $\cO_x$}\} \; \subseteq \;  \Hom_E(\cO_E,\cE)-\{0\}.
\end{equation} 
Clearly, 
\begin{equation}
  Z_{\cE}     \;=\;  \bigcup_{x \in E} Z_x. 
\end{equation}

\begin{lemma}
\label{lem.h0.E-x}
 Let $\cE \in \Bun(2,\cL)$.
   Fix a point $x \in E$.
  \begin{enumerate}
  \item\label{item.hzero.one.dec} 
  If $\cE = \cN_{1}\oplus\cN_{2}$ for some invertible $\cO_{E}$-modules $\cN_{1},\cN_{2}$ of degree $\geq 1$, then
    \begin{equation*}
      h^0(\cE(-x)) 
      \;=\; 
      \begin{cases}
        h^0(\cE)-2 & \text{if $\cN_{1}\not\cong\cO_{E}(x)$ and $\cN_{2}\not\cong\cO_{E}(x)$,}
        \\
        h^0(\cE)-1 & \text{if $\cN_{1}\cong\cO_{E}(x)$ or $\cN_{2}\cong\cO_{E}(x)$.}
      \end{cases}
    \end{equation*}
  \item\label{item.hzero.one.indec} 
    If $\cE$ is indecomposable, then $h^0(\cE(-x)) = h^0(\cE)-2=\deg\cE-2 $.
  \end{enumerate}
  More succinctly, if $\cE \in \Bun(2,\cL)$ and $x \in E$, then
  \begin{equation}
  \label{h0.E(-x).alternatives}
      h^0(\cE(-x)) 
      \;=\; 
      \begin{cases}
      h^0(\cE)-1 & \text{if $\cO_E(x)$ is a direct summand of $\cE$,}
      \\
        h^0(\cE)-2  & \text{otherwise.} 
      \end{cases}
    \end{equation}
\end{lemma}
\begin{proof}
  \cref{item.hzero.one.dec}
If both $\cN_{1}$ and $\cN_{2}$ have degree $\ge 2$, then both $\cN_{1}(-x)$ and $\cN_{2}(-x)$ have degree  
  $\ge 1$ so $h^0(\cE(-x))=\deg\cN_{1}(-x) + \deg\cN_{2}(-x) =\deg\cN_{1}+\deg\cN_{2}-2=h^0(\cE)-2$.

  If $\deg\cN_{1}=1$ and $\cN_{1}\not\cong\cO_{E}(x)$, then $h^0(\cN_{1})=0$ and, since $\deg\cL \ge 3$, $h^0(\cN_{2}(-x))=h^0(\cN_{2})-1$, whence $h^0(\cE(-x))=h^0(\cE)-2$.

  If $\cN_{1}\cong\cO_{E}(x)$, then $h^0(\cN_{1}(-x))=1=h^0(\cN_{1})$ and $h^0(\cN_{2}(-x))=h^0(\cN_{1})-1$, whence $h^0(\cE(-x))=h^0(\cE)-1$.

  \cref{item.hzero.one.indec} 
  Suppose $\cE$ is indecomposable.  Every indecomposable locally free $\cO_E$-module is semistable 
  (a proof can be found in \cite[Appendix A]{tu}) so $h^0(\cE)=\deg\cE$ by \cite[Lem.~17]{tu}.  
  Since $\deg\cE=\deg\cL \ge 3$, $\cE(-(x))$ is also semistable of positive degree so $h^0(\cE(-x))=\deg\cE(-x)=\deg\cE-2=h^0(\cE)-2$.
\end{proof}

\begin{lemma}\label{lem.hzero.two}
	Let $\cE \in \Bun(2,\cL)$. Fix points $x,y\in E$.
	\begin{enumerate}
		\item\label{item.hzero.two.dec}
		If $\cE = \cN_{1}\oplus\cN_{2}$ for some invertible $\cO_{E}$-modules $\cN_{1},\cN_{2}$ of degree $\geq 2$, then
		\begin{equation*}
			h^{0}(\cE(-x-y))\;=\;
			\begin{cases}
				h^{0}(\cE)-4 & \text{if neither $\cN_{1}$ nor $\cN_{2}$ is isomorphic to $\cO_{E}(x+y)$,}\\
				h^{0}(\cE)-3 & \text{if exactly one of $\cN_{1},\cN_{2}$ is isomorphic to $\cO_{E}(x+y)$,}\\
				h^{0}(\cE)-2 & \text{if both $\cN_{1}$ and $\cN_{2}$ are isomorphic to $\cO_{E}(x+y)$.} \\
			\end{cases}
		\end{equation*}
		The third case occurs if and only if $n=4$ and $\cE \cong \cO_E(\omega +(0))^{\oplus 2}$  for some $\omega \in \Omega$ and $x+y =\omega$.\footnote{When $n=4$,  $\cO_E(\omega + (0))$ is the sheaf $\cL_\omega$ defined in \cref{ssect.bun.notn}. Hence the third case occurs if and only if $n=4$
		and $\cE \cong \cL_\omega \oplus \cL_\omega$.}
		\item\label{item.hzero.two.indec}
		If $\cE$ is indecomposable, then
		\begin{equation*}
			h^{0}(\cE(-x-y))\;=\;
			\begin{cases}
				h^{0}(\cE)-3=0 & \text{if $n=3$,}\\
				h^{0}(\cE)-3=1 & \text{if $n=4$, $x+y\in\Omega$, and $\cE\cong\cE_{x+y}$,}\\
				h^{0}(\cE)-4 & \text{otherwise.}\\
			\end{cases}
		\end{equation*}
	\end{enumerate}
\end{lemma}
\begin{proof}
	\cref{item.hzero.two.dec}
	This can be proved in a similar way to \cref{lem.h0.E-x}\cref{item.hzero.one.dec}.
	The only point that might need additional explanation is the last sentence: if $\cE \cong \cO_E(x+y) \oplus \cO_E(x+y)$, then 
	$\cL \cong \det\cE \cong \cO_E(x+y)^{\otimes 2}$ so $n=\deg\cL=4$, and $x+y+x+y \sim H$ so $x+y \in \Omega$; in particular,
	$x+y = \omega$ for some 	$\omega \in \Omega$, so $\cO_E(x+y) \cong \cO_E((\omega) + (0))=\cL_\omega$ ($\cL_\omega$ is defined in     
	\cref{ssect.bun.notn}).
	
	\cref{item.hzero.two.indec}
	If $n=3$, then $\cE(-x-y)$ is semistable of negative degree, so $h^{0}(\cE(-x-y))=0$.
	
	Suppose $n=4$. As in \cref{ssect.bun.notn}, $\cE\cong\cE_{\omega}$ for some $\omega\in\Omega$, and $\cE_{\omega}$ is a non-split self-extension of $\cL_{\omega}=\cO_{E}((\omega)+(0))$. If $x+y\neq\omega$, then $h^{0}(\cL_{\omega}(-x-y))=0$ so $h^{0}(\cE(-x-y))=0$. If $x+y=\omega$, then $\cE(-x-y)$ is a non-split self-extension of $\cO_{E}$, so applying $\Hom(\cO_{E},-)$ to that extension produces an exact sequence
	\begin{equation*}
		\xymatrix{
			0\ar[r] & \Hom(\cO_{E},\cO_{E})\ar[r] & \Hom(\cO_{E},\cE(-x-y))\ar[r] & \Hom(\cO_{E},\cO_{E})\ar[r] & \Ext^{1}(\cO_{E},\cO_{E})
		}
	\end{equation*}
	in which the right-most morphism is non-zero. Hence $H^{0}(\cE(-x-y))=\Hom(\cO_{E},\cE(-x-y))$ has dimension $1=h^{0}(\cE)-3$.
	
	 If $n\geq 5$ this can be proved in a similar way to \cref{lem.h0.E-x}\cref{item.hzero.one.indec}.
\end{proof}

\subsubsection{Remark on notation}
We will use the symbol ``$\PP$'' to denote the projectivizations of spaces, varieties, etc. 
In all cases, it will be clear from the context which action of $\CC^{\times}$ we are quotienting out. 
For example, $\PP Z_x$ is the quotient of $Z_x$ by the  action that is the restriction of the  natural
action of $\CC^\times$ on $\Hom(\cO_E,\cE)-\{0\}$. Similarly, $\PP Z_{\cE}= Z_{\cE}/\CC^\times$.

\begin{proposition}\label{le:zx-codim2}
Suppose $\cE$ satisfies the assumptions in \cref{ssect.assump}.
  If $x\in E$, then 
  \begin{equation}\label{eq:pzx}
    \PP Z_x \; \subseteq \; \PP \Hom(\cO_E,\cE)
  \end{equation}
   is a linear subspace of $\PP \Hom(\cO_E,\cE)$  and 
  \begin{equation*}
  \codim\PP Z_x \;=\; 
  \begin{cases}
  1 & \text{if $\cO_E(x)$ is a direct summand of $\cE$,}
  \\ 
 2 & \text{otherwise.}
\end{cases}
\end{equation*}
\end{proposition}
\begin{proof}
  The cokernel of a morphism $f:\cO_E \to \cE$ contains a copy of the skyscraper sheaf $\cO_x$ if and only if $f$ extends to the central term in the extension
$0\to \cO_E \to \fm_x^{-1} \to \cO_x\to 0$.
Hence $Z_x \cup\{0\}$ is the image of the natural map $\Hom(\fm_x^{-1}, \cE) \to \Hom(\cO_E,\cE)$. This map is injective because $\Hom(\cO_x,\cE)=0$.
Thus, $\dim Z_x=\dim \Hom(\fm_x^{-1}, \cE) = h^0(\cE(-x))$.
The result now follows from \cref{lem.h0.E-x}.
\end{proof}

Piecing together the various projective spaces $\PP Z_x$ of \Cref{le:zx-codim2}, we have the following result.

\begin{proposition}\label{pr:z-codim1}
If $\cE$ satisfies the assumptions in \cref{ssect.assump}, then $\PP Z_\cE$ is Zariski-closed in $\PP \Hom(\cO_E,\cE)\,{\rm :}$
 \begin{equation*}
 \PP Z_\cE \;=\; 
  \begin{cases}
  \text{a union of two hyperplanes if $\cO_E(x)$ is a direct summand of $\cE$ for some $x \in E$,}
  \\ 
 \text{an irreducible hypersurface otherwise.}
\end{cases}
\end{equation*}
Hence $X(\cE)=\Hom(\cO_E,\cE)-\text{(a hypersurface)}$.
\end{proposition}
\begin{proof}
Since $\cE$ is fixed, we will write $Z$ in place of $Z_\cE$. 

\begin{enumerate}[inline]
\item
Suppose $\cO_E(x)$ is a direct summand of $\cE$. 
Then $\cE\cong \cO_E(x)\oplus \cL(-x)$. We fix such an isomorphism, and 
regard $\cO_E(x)$ and $\cL(-x)$ as submodules of $\cE$, and 
regard $\Hom(\cO_{E},\cO_E(x))$ and $\Hom(\cO_{E},\cL(-x))$ as subspaces of $\Hom(\cO_{E},\cE)$. We first show that
\begin{equation}\label{eq.union.zy}
	\bigcup_{y\in E-\{x\}}(Z_{y}\cup\{0\})\;\subseteq\;\Hom(\cO_{E},\cL(-x))\;\subseteq\;Z\cup\{0\}.
\end{equation}

Let $y \in E-\{x\}$. 
As in the proof of \cref{le:zx-codim2},  $Z_y\cup\{0\}$ is the image of the natural injection $\Hom(\fm_y^{-1}, \cE) \to \Hom(\cO_E,\cE)$.
Since $\Hom(\fm_y^{-1},\cO_E(x))\cong\Hom(\cO_E(y),\cO_E(x))=0$, that image is equal to the image of $\Hom(\fm_y^{-1}, \cL(-x))\to\Hom(\cO_{E}, \cL(-x))$. Hence the first inclusion in \cref{eq.union.zy} holds.

Let $0\neq f\in\Hom(\cO_{E},\cL(-x))$. Since $\deg\cL(-x)\geq 2$, the cokernel of $f$ is not torsion-free, so, as an element of $\Hom(\cO_{E},\cE)$, $f$ belongs to $Z$. Hence the second inclusion in \cref{eq.union.zy} holds.

On the other hand, since $Z_x\cup\{0\}$ is the image of the natural injection $\Hom(\fm_x^{-1}, \cE) \to \Hom(\cO_E,\cE)$, it is equal to the direct sum of
\begin{align*}
	V_{1}\;&:=\;\text{the image of $\Hom(\fm_x^{-1},\cO_{E}(x))\;\longrightarrow\;\Hom(\cO_E,\cO_{E}(x))$},\quad\text{and}\\
	V_{2}\;&:=\;\text{the image of $\Hom(\fm_x^{-1},\cL(-x))\;\longrightarrow\;\Hom(\cO_E,\cL(-x))$}.
\end{align*}
Since $\fm_x^{-1} \cong \cO_{E}(x)$,
$\dim V_{1}=1$. Since $\Hom(\fm_x^{-1},\cL(-x))\cong\Hom(\cO_{E},\cL(-2x))$,
$\dim V_{2}=n-2$.
It follows from this, and  \cref{eq.union.zy}, that
\begin{equation*}
	Z\cup\{0\}\;=\;\Hom(\cO_{E},\cL(-x))\cup(V_{1}\oplus V_{2}).
\end{equation*}
Both $\Hom(\cO_{E},\cL(-x))$ and $V_{1}\oplus V_{2}$ have codimension one in $\Hom(\cO_{E},\cE)$. Since these subspaces are different, 
$\PP Z$ is the union of
two hyperplanes.

\item\label{item.proof.pr.z-codim1.2}
Suppose no $\cO_E(x)$ is a direct summand of $\cE$. 

The natural map $H^0(E,\cE) \otimes \cO_E \to \cE$ is epic: if $\cE$ is decomposable, then it is a direct sum of two invertible sheaves of degree $\ge 2$ 
(because no $\cO_E(x)$ is a direct summand of $\cE$) so is generated by their global sections;
if $\cE$ is indecomposable it is also generated by its global sections (by \cite[Lem.~4.8(4)]{CKS3}, for example)  
because it is then semistable and, by hypothesis, its slope is $\frac{n}{2}>1$. Thus,
there is an exact sequence
\begin{equation}\label{eq.gen.sec}
	\xymatrix{
		0\ar[r] & \cZ\ar[r] & H^0(E,\cE) \otimes \cO_E\ar[r] & \cE\ar[r] & 0.
	}
\end{equation}

Since $\cZ$ is a submodule of a locally free $\cO_E$-module it is locally free; and $\rank\cZ = n-\rank\cE=n-2$. 
Let $x \in E$.
Since ${\mathcal{T}\!}{or}_1(\cO_E/\fm_x, \cE)  =0$, applying $\cO_E/\fm_x \otimes -$ to \cref{eq.gen.sec} produces an exact sequence 
\begin{equation}
\label{eq:fibers.of.cZ}
\xymatrix{
0 \ar[r] & \cZ/\fm_x \cZ    \ar[r] & H^0(E,\cE)  \ar[r] & \cE/\fm_x \cE \ar[r] & 0
}
\end{equation}
of vector spaces.\footnote{If $x$ is a closed point on a scheme $X$ and $\cF$ is an $\cO_X$-module we often speak of, and think of, $\cF/\fm_x\cF$ as 
a vector space over the residue field $\kappa(x)$. The justification for this is to think of the point $x$ as a morphism $i_x:\Spec(\kappa(x)) \to X$
and to identify $\cF/\fm_x\cF$ with $i_x^*\cF$.}
The kernel of the map $H^0(E,\cE) \to \cE/\fm_x\cE$ is the image of the map $H^0(E,\cE(-x))\cong \Hom(\fm_x^{-1},\cE)\to H^0(E,\cE)$
that appeared in the proof of \cref{le:zx-codim2},  so the fiber over $x$ of the vector bundle associated to $\cZ$ is $Z_x \cup\{0\}$. 

The map $\cZ \to H^0(E,\cE) \otimes \cO_E$ gives rise to a morphism $\phi:\PP(\cZ)\to E\times \PP H^0(E,\cE)$
between the corresponding projective-space bundles. The composition of $\phi$ with the projection to $\PP H^0(E,\cE)$ is a morphism 
$\PP(\cZ) \to \PP H^0(E,\cE)$ whose image is the union of all the $\PP Z_x$'s; i.e., its image is $\PP Z \subseteq \PP H^0(E,\cE)$.
Since  $\phi$ is a projective morphism,  
$\PP Z$ is a closed subvariety of $\PP H^0(E,\cE)$, and is irreducible because  $\PP(\cZ) $ is irreducible.
Since $\rank\cZ=n-2$, $\dim \PP(\cZ)=n-2$. 

Finally, the morphism $\PP(\cZ) \to \PP H^0(E,\cE)$ has finite fibers, because the cokernel of a morphism $\cO_E\to \cE$ can only contain finitely many 
$\cO_x$'s, and is therefore {\it finite}  because it is a projective morphism (\cite[Exer.~III.11.2]{hrt}). Hence $\dim\PP Z =\dim \PP(\cZ)=n-2$;
i.e., $\PP Z$ is a hypersurface, as claimed.\qedhere
\end{enumerate}
\end{proof}

\begin{remark}\label{re:grauert}
The argument used to prove \cref{pr:z-codim1}\cref{item.proof.pr.z-codim1.2}
is a surreptitious application of Grauert's Theorem \cite[Cor.~III.12.9]{hrt}. To see this, let $\pi_1,\pi_2:E \times E \to E$ be the projections to the first and second factors, respectively, and let $\Delta \subseteq E^2$ denote the diagonal.

From the exact sequence $0 \to  \cO_{E^2}(-\Delta) \to \cO_{E^2} \to \cO_\Delta \to 0$ we obtain exact sequences
\begin{equation}
\label{eq:ses.on.E2}
\xymatrix{
0 \ar[r] & \cO_{E^2}(-\Delta)\otimes\pi_1^*\cE  \ar[r]  & \pi_1^*\cE \ar[r]  &  \cO_\Delta \otimes \pi_1^*\cE \ar[r] & 0
}
\end{equation}
and
\begin{equation}
\label{eq:ses.bundles}
\xymatrix{
0 \ar[r] & \pi_{2*}(\cO_{E^2}(-\Delta)\otimes\pi_1^*\cE) \ar[r] & \pi_{2*}\pi_1^*\cE \ar[r] & \pi_{2*} (\cO_\Delta \otimes \pi_1^*\cE) \cong  \cE
}
\end{equation}
where the isomorphism at the right follows from the fact that the restriction of $\pi_2$ to $\D$ is an isomorphism onto $E$.  We also note that 
$ \pi_{2*}\pi_1^*\cE \cong H^0(E,\cE) \otimes \cO_E$ and that the right-most map in \cref{eq:ses.bundles} is the natural map $H^0(E,\cE) \otimes \cO_E \to \cE$,
which is an epimorphism as shown in the proof of \cref{pr:z-codim1}. 
Hence the kernel, $\pi_{2*}(\cO_{E^2}(-\Delta)\otimes\pi_1^*\cE)$, is isomorphic to $\cZ$ in \cref{eq.gen.sec}. Each term in \cref{eq:ses.bundles} is the sheaf of sections of a bundle on $E$, and the bundle associated to the middle term 
$ \pi_{2*}\pi_1^*\cE$ is the trivial bundle $E \times H^0(E,\cE)$. 

The module $\cE(-x)$ is  the $\pi_2$-fiber over $x\in E$ of the bundle corresponding to $ \cO_{E^2}(-\Delta)\otimes\pi_1^*\cE$.
  Since all $H^0(E,\cE(-x))$'s have the same dimension, $n-2$ (by \cref{lem.h0.E-x} because $\cO_E(x)$ is not a direct summand of $\cE$),   
  \cite[Cor.~III.12.9]{hrt} tells us that
  \begin{equation*}
    H^0(E,\cE(-x)) \; \cong \; \text{the fiber over $x$ of the bundle corresponding to } \pi_{2*}(\cO_{E^2}(-\Delta)\otimes\pi_1^*\cE).
  \end{equation*}
  Similarly, since $\pi_{2*}\pi_1^*\cE \cong H^0(E,\cE)\otimes \cO_E$, 
  \begin{align*}
    H^0(E,\cE) & \; \cong \; \text{the fiber over $x$ of the bundle corresponding to } \pi_{2*}\pi_1^*\cE.
  \end{align*}    
The map $\pi_{2*}(\cO_{E^2}(-\Delta)\otimes\pi_1^*\cE)\to\pi_{2*}\pi_1^*\cE$ is by construction a fiber-wise linear embedding, and hence induces a map
  \begin{equation*}
    \PP\bigl( \pi_{2*}(\cO_{E^2}(-\Delta)\otimes\pi_1^*\cE)\bigr)
    \, \longrightarrow \,
    \PP (\pi_{2*}\pi_1^*\cE) \; \cong \; E\times \PP H^0(E,\cE)
  \end{equation*}
  of projective-space bundles over $E$. Projecting  onto the second component of the last term gives a map
  \begin{equation}\label{eq:allx}
    \PP\bigl( \pi_{2*}(\cO_{E^2}(-\Delta)\otimes\pi_1^*\cE)\bigr)
 \, \longrightarrow \,
    \PP \Hom(\cO_E,\cE)
  \end{equation}
  of projective varieties which maps the fiber $\PP H^0(E,\cE(-x))$ above $x$ isomorphically to the subspace $\PP Z_x \subseteq \PP \Hom(\cO_E,\cE)$.
  Because $Z$ is the union of the $Z_x$'s, it follows that $\PP Z$ is the image of the morphism \Cref{eq:allx} of projective varieties, and hence Zariski-closed
  and irreducible. \reqed
\end{remark}

The next result completes the proof of \cref{thm.good.Es.intro}.

\begin{corollary}\label{cor.good.E's}
If $\cE \in \Bun(2,\cL)$, then $L(\cE) \ne \varnothing$ if and only if 
\begin{equation*} 
\cE \; \cong \; \text{\rm something in } \, 
\begin{cases}
\,  \{ \cE_{d,x} \; | \; 1 \le d \le r, \; x \in E\} \cup \{\cE_o\} & \text{if $n=2r+1$,}
 \\
\,   \{ \cE_{d,x} \; | \; 1 \le d \le r, \; x \in E\}\cup \{\cE_\omega \; | \; \omega \in \Omega\}   & \text{if $n=2r$.}
\end{cases}
\end{equation*}
\end{corollary}
\begin{proof}
($\Rightarrow$)
Suppose $L(\cE) \ne \varnothing$. Then $\cE$ satisfies the assumptions in \cref{ssect.assump}. Thus, if $\cE$ is decomposable, then it is isomorphic 
to some $\cE_{d,x}$ by \Cref{prop.leaves.for.decomp.Es}. If $\cE$ is indecomposable, then, by the discussion in \cref{ssect.bun.notn}, 
it is isomorphic to $\cE_o$ if $n$ is odd and is isomorphic to $\cE_\omega$ for some $\omega \in \Omega$.

($\Leftarrow$)
By \Cref{prop.leaves.for.decomp.Es}, $L(\cE_{d,x}) \ne \varnothing$.  
Both $\cE_o$ and $\cE_\omega$ satisfy the assumptions in \cref{ssect.assump}, so both $X(\cE_o)$ and $X(\cE_\omega)$ are non-empty by
\Cref{le:zx-codim2}, whence $L(\cE_o)$ and $L(\cE_\omega)$ are non-empty by  \cref{lem.tfree.coker}.
\end{proof}

\Cref{cor.good.E's} also shows that $\cE$ satisfies the assumptions in \cref{ssect.assump} if and only if $L(\cE) \ne \varnothing$. 

\begin{corollary} 
\label{cor.mapE.to.N}
If $L(\cE) \ne \varnothing$, then $\Hom(\cE,\cO_E(D))=0$ for all divisors $D$ of degree $\le 0$.
\end{corollary}
\begin{proof}
It follows from \cref{cor.good.E's} and \Cref{ssect.dec.buns,ssect.bun.notn} that there is an exact sequence 
$0 \to \cN_1 \to \cE \to \cN_2 \to 0$ in which $\cN_1$ and $\cN_2$ are invertible $\cO_E$-modules of positive degree.  
The result follows.
\end{proof}

\begin{corollary}\label{cor:isaff}
If $L(\cE) \ne \varnothing$,
  then $\PP X(\cE)$ is an affine variety. 
\end{corollary}
\begin{proof}
By \Cref{pr:z-codim1}, $\PP Z_\cE\subseteq \PP H^0(E,\cE)$ is the zero locus of a homogeneous polynomial so  its open
complement is affine.\footnote{That the complement is affine is a special case of a more general fact that we will use later: if $D$ is an ample effective divisor 
on a projective scheme $X$, then the complement $X-D$ is affine. To see this, first note that if $Y$ is a closed subvariety of $X$, then $X-Y=X-Y_{\rm red}$, so
there is no loss of generality in replacing $D$ by $mD$ where $m$ is a positive integer. Thus, we can assume that $D$ is very ample. 
After identifying $X$ with its image under the morphism $X \to \PP^r:=\PP H^0(X,\cO_X(D))^*$ associated to the complete linear system $|D|$, there is a hyperplane
$H \subseteq \PP^r$ such that $X-D=X-(X \cap H) = X \cap (\PP^r-H)$; hence $X-D$ is a closed subscheme of the affine scheme $\PP^r-H$ and is therefore an
affine scheme also.}
\end{proof}

The bundles $\cL_\omega$ for $\omega \in \Omega$ were defined in \cref{ssect.bun.notn}.
When $n=4$,  $\cL_\omega = \cO_E((\omega)+(0))$. 

\begin{proposition}\label{pr:deghyp}
If $L(\cE) \ne \varnothing$, then
 \begin{enumerate}
  \item\label{item.deghyp.one}
$\deg\PP Z_\cE=  2$ if $\cE \cong \cE_{1,x}$ for some $x \in E$, i.e., if some $ \cO_E(x)$ is a direct summand of $\cE$, 
  \item\label{item.deghyp.two}
  $\deg\PP Z_\cE= \tfrac{n}{2}=2$ if $n=4$ and $\cE \cong \cL_\omega \oplus  \cL_\omega$ for some $\omega \in \Omega$,
  and
  \item\label{item.deghyp.ow}
  $\deg\PP Z_\cE = n$ otherwise.
\end{enumerate}
\end{proposition}
\begin{proof}
We will write $Z$ in place of $Z_\cE$.

\cref{item.deghyp.one}
 If $\cE \cong  \cE_{1,x}$,  then $\PP Z$ is a union of two hyperplanes by \Cref{pr:z-codim1},
so $\deg\PP Z=2$.

\cref{item.deghyp.two}\cref{item.deghyp.ow}
  We now assume that no $\cO_E(x)$ is a direct summand of $\cE$.
 Hence, by \cref{h0.E(-x).alternatives}, $h^{0}(E,\cE(-x))=h^{0}(E,\cE)-2$ for all $x \in E$.
  
  By definition, 
  \begin{equation*}
  \PP Z \; =\; \bigcup_{x \in E} \PP Z_x  \;=\; \bigcup_{x \in E} \PP \Hom(\cO_E,\cE(-x)) \; \subseteq \; \PP \Hom(\cO_E,\cE).
   \end{equation*}
Since $\PP Z_x$ is a codimension-two linear subspace of $\PP \Hom(\cO_E,\cE)$,
there is a morphism
  \begin{equation*}
    E  \, \longrightarrow \, \GG(n-3,n-1), \qquad x\mapsto \PP Z_x,
  \end{equation*}
  where $\GG(n-3,n-1)$ is the Grassmannian of $(n-3)$-planes in $\PP \Hom(\cO_E,\cE)$: as noted in \Cref{re:grauert}, the inclusions
$    H^0(E,\cE(-x))  \subseteq H^0(E,\cE)$ 
  glue to an embedding of a rank-$(n-2)$ vector bundle over $E$ into the trivial vector bundle $H^0(E,\cE)\times E$, and the fact that this gives a morphism into the Grassmannian then follows from the latter's universal property \cite[Thm.~3.4]{3264}.

  We are now in the setting of \cite[\S 10.2]{3264}, and wish to apply \cite[Prop.~10.4]{3264}. 
  To do this we must determine the number $d$ in \cite[Prop.~10.4]{3264} which is, 
 by definition, the unique number having the following property:  a general point in $\PP Z$ belongs to $\PP Z_x$ for 
 exactly $d$ different $x$'s;   i.e., $d$ is the unique positive integer such that there is a non-empty open set $U \subseteq \PP Z$ 
  with the property that every point in $U$ belongs to $\PP Z_x$ for exactly $d$ different $x$'s.
  We will show that $d=2$ in the case of \cref{item.deghyp.two} and that $d=1$ in the case of  \cref{item.deghyp.ow}.

 If $x \ne y$, then a section of $\cE$ vanishes at both $x$ and $y$ if and only if it is a section of $\cE(-x-y)$ so,
 as subspaces of $H^0(E, \cE)$,
  \begin{equation*}
   H^0(E,\cE(-x)) \cap H^0(E,\cE(-y)) \;=\; 
  H^0(E,\cE(-x-y))
    \end{equation*} 
  whence 
  \begin{equation*}
\PP Z_x \cap \PP Z_y  \;=\;  \PP H^0(E,\cE(-x-y)) .
  \end{equation*} 
Hence
\begin{equation} 
\label{eq:Zx.cap.Zy.nonempty}
 \{p \in \PP Z \; | \; p \text{ belongs to $\PP Z_x$ for at least two different $x$'s}  \}
 \;=\; 
 \bigcup_{D \in E^{[2]} -\Delta}\PP H^0(E,\cE(-D))
 \end{equation} 
 where $\Delta=\{ (\!(x,x)\!) \; | \; x \in E\} \subseteq E^{[2]}$.

Proof of \cref{item.deghyp.two}. 
 Suppose that $n=4$ and $\cE \cong \cL_\omega \oplus  \cL_\omega$ for some $\omega \in \Omega$. 
 
\underline{Claim:}
 If $p$ in $\PP Z_x$, then $p$ belongs to $\PP Z_y$ if and only if $y \in \{x, \omega -x\}$.

\underline{Proof:}
Suppose $x \ne y$. It follows from  \cref{lem.h0.E-x}\cref{item.hzero.one.dec} that $h^0(\cE(-x))=h^0(\cE(-y))=2$,
and follows from \cref{lem.hzero.two}\cref{item.hzero.one.dec} that 
   \begin{equation*}
h^0(\cE(-x-y))  \;=\;  
\begin{cases}
2 & \text{if $x+y=\omega$}
\\
0 & \text{if $x+y\ne \omega$.}
\end{cases}
  \end{equation*} 
Hence $\PP Z_x \cap \PP Z_y = \varnothing$ if  $x+y\ne \omega$, and $\PP Z_x = \PP Z_y$ if $x+y=\omega$.
$\lozenge$

The claim shows that 
\begin{equation*}
	\{p \in \PP Z \; | \; p \text{ belongs to $\PP Z_x$ for a unique $x\in E$}\} \;=\; \bigcup_{\{x \; | \; 2x=\omega\}} \PP Z_x
\end{equation*}
The set of such $x$'s has four elements because it is an $E[2]$-coset. Hence $d=2$.

Proof of \cref{item.deghyp.ow}. 
We are still assuming that no $\cO_E(x)$ is a direct summand of $\cE$. 
And, since we are not in case \cref{item.deghyp.two}, if $n=4$, then $\cE$ is not isomorphic to $\cL_\omega \oplus \cL_\omega$ for any $\omega$ in $\Omega$.
Hence, by  \cref{lem.h0.E-x}, $h^0(\cE(-x))=h^0(\cE)-2$ for all $x \in E$.

Let $D \in E^{[2]}$ be a degree-two effective divisor. 
If $n=3$, then $h^0(\cE(-D))=0$ by  \cref{lem.hzero.two} so $\PP Z_x \cap \PP Z_y=\varnothing$ for all 
$(\!(x,y)\!) \in E^{[2]}-\Delta$, whence $d=1$. 

Now assume $n\geq 4$ and $D \in E^{[2]}$.  We are not in the third case of \cref{lem.hzero.two}\cref{item.hzero.two.dec}, so
\begin{equation*}
h^0(\cE(-D)) \;=\; 
        \begin{cases} 
        h^{0}(\cE)-3   & \text{if $\cO_E(D)$ is a (multiplicity-one) direct summand of $\cE$,}
        \\
        h^{0}(\cE)-3   & \text{if $n=4$, $\cE\cong\cE_{\omega}$ ($\omega\in\Omega$), and $D\sim \omega+(0)$,}
        \\ 
        h^{0}(\cE)-4 & \text{otherwise.} 
	\end{cases}
\end{equation*} 
The first two cases occur for only finitely many $D$, say $\{D_{i}=z_{i}+(0)\}_{i\in I}$, up to linear equivalence. Hence the subset, $U$ say, of $E^{[2]}-\Delta$ for which the third case occurs 
is a non-empty and open subset of $E^{[2]}$. Arguing as in  \cite[discussion preceding Prop.~10.4]{3264} (where their $m$ denotes the dimension of their $B$), the dimension of the variety
 \begin{equation}\label{eq.union.u}
 \bigcup_{D \in U}\PP H^0(E,\cE(-D))
  \end{equation}
  is $\le 2+(n-4)-1=n-3$. But $\dim \PP Z=n-2$, so a general point in $\PP Z$ does not belong to \cref{eq.union.u}.

The complement $U^{c}:=(E^{[2]}-\Delta)-U$ is the union of the sets
\begin{equation*}
	\{(\!(x,y)\!)\in E^{[2]}-\Delta\mid x+y=z_{i}\}\quad\text{($i\in I$)},
\end{equation*}
each of which is a fiber of the restriction to $E^{[2]}-\Delta$  of the summation map $\sigma:E^{[2]}\to E$, thus is isomorphic to an open subscheme of $\PP^{1}$. Hence, by \cite[discussion preceding Prop.~10.4]{3264}, the  dimension of the variety
\begin{equation}\label{eq.union.uc}
	\bigcup_{D \in U^{c}}\PP H^0(E,\cE(-D))
\end{equation}
is $\le 1+(n-3)-1=n-3$. Since $I$ is  finite, a general point in $\PP Z$ does not belong to \cref{eq.union.uc}, either.

Consequently, a general point in $\PP Z$ belongs to exactly one $\PP Z_x$, whence $d=1$. This concludes the proof that $d=1$ in  the case of \cref{item.deghyp.ow}.

 In order to apply \cite[Prop.~10.4]{3264}  we now translate it to the present setting:
  \begin{itemize}
  \item Let $\cZ$ be the locally free $\cO_{E}$-module in part \cref{item.proof.pr.z-codim1.2} of the proof of \cref{pr:z-codim1}, whose $x$-fiber is $H^0(E,\cE(-x))$.
  \item Taking into account the relation between {\it Chern} and {\it Segre classes} in the {\it Chow ring} of $E$ 
  (\cite[\S 10.1, especially Defn.~10.1 and Prop.~10.3]{3264}) and the fact that in the present situation 
  the number $m$ in \cite[Prop.~10.4]{3264} is 1,  that result says that
    \begin{equation}\label{eq:minusdeg}
      \text{the degree of }\PP Z \text{ as a subvariety of }  \PP H^0(E,\cE) \; = \; -\, \frac{\deg\cZ}{d}.
    \end{equation}
  \end{itemize}  
By \cref{eq.gen.sec} and additivity  of degree \cite[Exer.~II.6.12]{hrt}, we have $\deg\cZ=-\deg\cE=-n$, which concludes the proof.
\end{proof}

%%%%%%%%%%%%%%%%%%%%%%%%%%%%%%%%%%%%%%%%%%%%%%%%%%%%%%%%%%%%%%%%
\subsection{The morphism $\Psi_\cE:X(\cE)  \to L(\cE)$ and the dimension of $L(\cE)$}
%%%%%%%%%%%%%%%%%%%%%%%%%%%%%%%%%%%%%%%%%%%%%%%%%%%%%%%%%%%%%%%%

\begin{lemma}\label{le:ismor}
  If $L(\cE) \ne \varnothing$, then
the map
  \begin{equation}
    \label{eq:defn.psi_E}
    \Psi_\cE:   X(\cE) \, \longrightarrow \,  \PP_\cL 
  \end{equation} 
  that  sends $f \in X(\cE)$, or the corresponding $\pi \in {\rm Epi}(\cE,\cL)$, to the isomorphism class of the extension 
  \begin{equation}
    \label{eq:temp.xi}
    \xi \; := \quad
    \xymatrix{
      0 \ar[r] & \cO_E \ar[r]^f & \cE \ar[r]^\pi & \cL  \ar[r] & 0
    }
  \end{equation}
  is a morphism whose image is $L(\cE)$. 
\end{lemma}
\begin{proof}
  Let $\GG:=\GG(n-1,\Hom(\cO_E,\cL))$.  

Let $f \in X(\cE)$, let $\pi \in {\rm Epi}(\cE,\cL)$ be the corresponding epimorphism in \Cref{re:kercoker}, and let 
$\xi=\Psi_\cE(f)$ be the corresponding extension in \cref{eq:temp.xi}. We give ${\rm Epi}(\cE,\cL)$ its usual algebraic structure as a Zariski-open 
subset of $\Hom(\cE,\cL)$. 
The linear  bijection $X(\cE) \to  {\rm Epi}(\cE,\cL)$  in \cref{re:kercoker} is an isomorphism of quasi-affine varieties so it suffices to 
show that the map $ {\rm Epi}(\cE,\cL) \to \PP_\cL$, $\pi \mapsto \xi$, is a morphism. 
Since the map  $\xi \mapsto \im(\pi_*)=\xi^\perp$ is an isomorphism $\PP_\cL \to \GG$,  it suffices to show that that the map
 ${\rm Epi}(\cE,\cL) \to  \GG$,  $\pi \mapsto \mathrm{im}(\pi_*)$,   is a morphism. That is what we will do.

 The map that sends $\pi: \cE\to \cL$ to the linear map $\pi_*$ in \cref{eq:pi_*} is itself a linear map, hence a morphism.  Restricting that map to {\it epimorphisms} produces linear maps $\Hom(\cO_E,\cE) \to \Hom(\cO_E,\cL)$ of {\it maximal} rank, i.e., of rank $n-1 = \deg\cL-1 = \deg\cE-1$.  It is a standard fact that under these conditions the map ${\rm Epi}(\cE,\cL) \to \GG$, $\pi \mapsto \mathrm{im}(\pi_*)$, is a morphism: see, for example, \cite[Prop.~13.4]{saltman} or \cite[Prop.~3.17(1)]{CKS1}.
\end{proof}

\begin{remark}
We will often use the projectivized version $\PP X(\cE) \to \PP_\cL$ of \Cref{le:ismor}.
\reqed
\end{remark}

We will eventually see that the homological leaves $L(\cE)$ are locally closed subsets of $\PP_\cL$ (\Cref{th:splitall,th:nsplitodd,th:nsplitev}). For now, the next lemma suffices. Recall that a set is {\it constructible} if it is a finite disjoint union of locally closed subsets (see \cite[\S AG.1.3, p.~2]{borel-LAG-se} or \cite[Exer.~II.3.18]{hrt}, for example).

\begin{lemma}
\label{le:rightdim}
Every $L(\cE)$ is an irreducible constructible subset of $\PP_\cL$ and 
  \begin{equation}
  \dim L(\cE) \; =\; n-\dim \Aut(\cE).
  \end{equation}
\end{lemma}
\begin{proof}
Let $\Psi:X(\cE)  \to \overline{L(\cE)}$ be
the morphism $\Psi_\cE$ in  \Cref{cor.defn.Psi.cE}.
Since $L(\cE)$ is the image of $\Psi$, it is constructible by Chevalley's Theorem \cite[Exer.~II.3.19]{hrt}. 
But $L(\cE)$ is dense in $\overline{L(\cE)}$, so $\Psi$ is dominant.\footnote{See \cite[\S AG.8.2, p.~19]{borel-LAG-se} or \cite[Exer.~II.3.7, p.~91]{hrt}.}
Since the domain of $\Psi$ is irreducible (by \Cref{pr:z-codim1}) its image is also irreducible; hence $\overline{L(\cE)}$ is   irreducible.

The non-empty fibers of $\Psi$ are the images of the orbit maps $ \Aut(\cE)\to X(\cE)$, $g \mapsto gx$,
and hence codomains of bijective morphisms from $\Aut(\cE)$ onto (closed) subschemes of $X(\cE)$. 
They all have dimension $\dim \Aut(\cE)$ (given explicitly in \Cref{le:dim.Aut.cE}), as follows, for instance, from \cite[Exer.~II.3.22]{hrt}.
The same exercise (specifically, \cite[Exer.~II.3.22(b),(c)]{hrt}) also shows  there is a dense open subset of $\overline{L(\cE)}$ (hence of $L(\cE)$) over which the fibers have irreducible components of dimension exactly
  \begin{equation*}
    e_{\Psi} \; := \; \dim(\text{domain of }\Psi)-\dim(\text{codomain of }\Psi).
  \end{equation*} 
  It follows   that $e_{\Psi}=\dim \Aut(\cE)$, which is the desired conclusion rephrased.
\end{proof}

\begin{remark}
The statement about $\dim L(\cE)$ in \cref{le:rightdim} is compatible with \cite[Prop.~2.3]{HP3} which says that 
the rank of the Poisson bracket $\Pi$ on $\PP_\cL$ at a point $\xi \in \PP_\cL$ equals 
$n-\dim\End(\cE)$ where $\cE$ is the middle term of the extension $\xi$. In short,  $\rank\Pi_\xi = n-\dim\End(\cE)$.

The next result, \cref{le:dim.Aut.cE}, computes the dimension of $\Aut(\cE)$; together with \cref{le:rightdim}, it allows us to determine  $\dim L(\cE)$
and \Cref{prop.dim.LE} then follows.
\reqed
\end{remark}

We follow customary practice (e.g., \cite[\S 7.1]{hmph-lag}) by writing
\begin{itemize}
\item $\GG_a$ for the {\it additive} 1-dimensional algebraic group, which is $(\CC,+)$ in this paper, and
\item  $\GG_m$  for the {\it multiplicative} group $(\CC^{\times},\cdot)$.
\end{itemize}

\begin{lemma}\label{le:dim.Aut.cE}
If $L(\cE) \ne \varnothing$,  then $\Aut(\cE)$ is a linear algebraic group,
  and is isomorphic to
  \begin{enumerate}
  \item\label{item.le.dim.Aut.cE.indec.odd} $
  \GG_m$ if $\cE$ is indecomposable and $n$ is odd;
  \item\label{item.le.dim.Aut.cE.indec.even} 
  $\GG_a\times \GG_m$ if $\cE$ is indecomposable and $n$ is even;
  \item\label{item:mostcompl}
   $\big(\GG_a^{\left|\deg\cN_1-\deg\cN_2 \right|}\rtimes \GG_m\big)\times \GG_m$ if $ \cE\cong \cN_1\oplus \cN_2$ and $\cN_1\not\cong \cN_2$, 
    where $\GG_m$ acts on $\GG_a^{\bullet}$ by scaling;
  \item\label{item:gl2case}  ${\rm GL}(2)$ if $\cE \cong \cN \oplus \cN$. (In this case, $\cN \cong \cL_\omega$ for a unique $\omega \in \Omega$.)
  \end{enumerate}
 The dimension of $\Aut(\cE)$ in the four cases is, respectively,
  \begin{equation*}
    1,\quad 2,\quad 2+\left|\deg\cN_1-\deg\cN_2\right|, \quad 4.
  \end{equation*}
\end{lemma}
\begin{proof}
  In the first case, $\cE \cong \cA' \otimes \cN$ for some invertible $\cO_E$-module $\cN$, where $\cA'$ is the unique non-split extension of $\cO_{E}((0))$ by $\cO_{E}$, so $\End( \cE) \cong \End( \cA') =\CC$.  In the second case, $\cE \cong \cA \otimes \cN$ for some invertible $\cO_E$-module $\cN$, where $\cA$ is the unique non-split self-extension of $\cO_{E}$, so $\End( \cE) \cong \End( \cA) \cong \CC[x]/(x^2)$.  The last two cases follow from the fact that
\begin{equation*}
  \End( \cN_1 \oplus \cN_2) \; \cong \;
  \begin{pmatrix}
    \CC     &  \Hom(\cN_2,\cN_1)   \\
    \Hom(\cN_1,\cN_2)   &  \CC
  \end{pmatrix}. 
\end{equation*}

Finally, one sees by inspection that all the groups in \cref{item.le.dim.Aut.cE.indec.odd,item.le.dim.Aut.cE.indec.even,item:mostcompl,item:gl2case} are linear algebraic groups.
\end{proof}

We define 
\begin{equation}
\label{eq:projaut}
   \PP \Aut(\cE) \; := \; \frac{\Aut(\cE)} {\text{ the central  copy of $\GG_m$}}.
\end{equation}

%%%%%%%%%%%%%%%%%%%%%%%%%%%%%%%%%%%%%%%%%%%%%%%%%%%%%%%%%%%%%%%%
%%%%%%%%%%%%%%%%%%%%%%%%%%%%%%%%%%%%%%%%%%%%%%%%%%%%%%%%%%%%%%%%
\section{The relation between the $L(\cE)$'s  and the partial secant varieties $\Sec_{d,x}(E)$}
\label{sec.xi.D.bar.sec}
%%%%%%%%%%%%%%%%%%%%%%%%%%%%%%%%%%%%%%%%%%%%%%%%%%%%%%%%%%%%%%%%
%%%%%%%%%%%%%%%%%%%%%%%%%%%%%%%%%%%%%%%%%%%%%%%%%%%%%%%%%%%%%%%%

Fix a point $\xi \in \PP_\cL$ and a representative for it, say
\begin{equation}
 \label{eq:repr.for.xi}
 \xymatrix{
 0 \ar[r] & \cO_E \ar[r]^f & \cE \ar[r]^{\pi} & \cL \ar[r] & 0.
 }
 \end{equation}
In \Cref{sec.secants.containing.xi} we provide several characterizations of the effective divisors $D$ such that $\xi \in \overline{D}$. 
Those results are  used in 
 \Cref{sec.leaves.and.secants} to describe the relationship between the $L(\cE)$'s and the  varieties
 $\Sec_{d,x}(E)-\Sec_{d-1}(E)$.

%%%%%%%%%%%%%%%%%%%%%%%%%%%%%%%%%%%%%%%%%%%%%%%%%%%%%%%%%%%%%%%%
\subsection{Various criteria for $\xi \in \PP_\cL$  to belong to $\Sec_{d,x}(E)$}
\label{sec.secants.containing.xi}
%%%%%%%%%%%%%%%%%%%%%%%%%%%%%%%%%%%%%%%%%%%%%%%%%%%%%%%%%%%%%%%%

If $D$ is an effective divisor on $E$ we write
\begin{itemize}
\item[-]
$\cI_D$ for the ideal in $\cO_E$ that vanishes on $D$;
\item[-]
$i_D:\cI_D \cL \to \cL$ for the inclusion map;
\item[-]
$|D|= \PP H^0(E,\cO_E(D)) = \{\text{effective divisors  } D' \; | \; D' \sim D\}$, where we identify $D'$ with $\CC s$, the line spanned by
a section $s$ such that $(s)_0=D'$.
\end{itemize}

\begin{proposition}
\label{prop.bertram}
\label{new.version.prop.4.1}
\label{prop.D.min.epi}
\label{prop.D.min}
Let $\xi \in \PP_\cL$ be the  isomorphism class of the sequence \Cref{eq:repr.for.xi}, and let $D\in \Div(E)$ be an effective divisor of degree $<n$. 
The following statements are equivalent:
\begin{enumerate}
  \item\label{item.D.sec} 
   $\xi \in \overline{D}$;
  \item\label{item.D.comp} 
  $\xi \cdot i_D=0$; 
  \item\label{item.D.im} 
  $i_D \in \pi \Hom(\cI_D\cL,\cE)$.
\end{enumerate}
If $D' \in \Div(E)$ is effective and either $\deg D'<n$ or $D'\sim H$, then the following statements are equivalent:
\begin{enumerate}
  \setcounter{enumi}{3}
  \item\label{item.D.some} 
   statements \cref{item.D.sec}, \cref{item.D.comp}, \cref{item.D.im} hold for some $D \in |D'|$.
  \item\label{item.D.hom} 
  $\Hom(\cE,\cO_E(D')) \ne 0$; 
  \item\label{item.D.ideal} 
  $\Hom(\cI_{D'}\cL,\cE) \ne 0$.
\end{enumerate}
If $D_{\rm min}$ is an effective divisor of minimal degree such that $\xi \in \overline{D_{\rm min}}$, then there is an epimorphism $\cE \to \cO_E(D)$
for all $D \in |D_{\rm min}|$.
\end{proposition}
\begin{proof}
\cref{item.D.sec} $\Leftrightarrow$ \cref{item.D.comp} $\Leftrightarrow$ \cref{item.D.im}.
The condition $\xi\in\overline{D}$ says that $\xi$ belongs to every hyperplane in $\PP \Ext^1(\cL,\cO) = \PP H^0(E,\cL)^*$
that contains $D$. Hyperplanes in $ \PP H^0(E,\cL)^*$ 
correspond to non-zero sections $s\in H^0(E,\cL)$, and containing 
$D$ simply means that the corresponding section vanishes on $D$. Thus
\begin{equation*}
	\overline{D} \;=\; \PP\bigl(\text{the kernel of the map $H^{0}(E,\cL)^{*}\to H^{0}(E,\cI_{D}\cL)^{*}$}\bigr),
\end{equation*}
where the map $H^{0}(E,\cL)^{*}\to H^{0}(E,\cI_{D}\cL)^{*}$  is induced by  $i_D$.
The condition $\xi\in\overline{D}$, then, is equivalent to the condition that $\xi$, regarded as a functional on $H^0(E,\cL)$, vanishes on $H^0(E, \cI_D\cL)$. 
Serre duality provides a commutative diagram
\begin{equation}
\label{eq:prop.bertram.SD}
\xymatrix{
	H^0(E,\cL)^*   \ar[d]  \ar[r]  & H^0(E,\cI_D\cL)^*    \ar[d]
	\\
	\Ext^1_E(\cL,\cO_E)  \ar[r]_-{\iota_D}  & \Ext^1_E(\cI_D\cL,\cO_E),
}
\end{equation}    
in which $\iota_D$ sends an extension $\xi$ to the extension $\xi \cdot i_D$ which is the top row of the pullback diagram
\begin{equation}
\label{eq:prop.bertram.SD.2}
\xymatrix{
	\xi\cdot i_D : \quad 0 \ar[r] & \cO_E \ar[r]^-{f'}  \ar@{=}[d]& \pi^{-1}(\cI_{D}\cL) \ar[d]^{h}  \ar[r]^-{\pi'}  &  \cI_D\cL \ar[r]  \ar[d]^-{i_D} & 0 
\\
	\xi: \quad 0 \ar[r] & \cO_E \ar[r]_f  &  \cE  \ar[r]_\pi  &  \cL \ar[r]  & 0.
}
\end{equation}
It follows that $\xi \in \overline{D}$ if and only if $\xi \cdot i_D=0$ in $\Ext^1_E(\cI_D\cL,\cO_E)$, i.e., 
if and only if the top row of \cref{eq:prop.bertram.SD.2} splits.
Hence $\xi \in \overline{D}$ if and only if there is a map $g:\cI_D\cL \to \cE$ such that $\pi g=i_D$.\footnote{In the diagram, $h$ denotes the inclusion map.}

This completes the proof that \cref{item.D.sec} $\Leftrightarrow$ \cref{item.D.comp} $\Leftrightarrow$ \cref{item.D.im}.

Now let $D'$ be an effective divisor such that either $\deg D'<n$ or $D' \sim H$.

Suppose $D'\sim H$. Since $\xi$ lies on a hyperplane it lies on $\overline{D}$ for some $D \in |D'|$, so \cref{item.D.some} holds. Statements 
\cref{item.D.hom} and \cref{item.D.ideal} also hold because $\cO_E(D') \cong \cL$  and 
 $\cI_{D'} \cL \cong \cO_E$. So, for the rest of the proof we assume that $\deg D'<n$.

\cref{item.D.some} $\Rightarrow$ \cref{item.D.ideal}.
Suppose \cref{item.D.some} holds. Let $D \in |D'|$ be such that $i_D \in \pi \Hom(\cI_D\cL,\cE) \ne 0$. Then $0 \ne \Hom(\cI_D\cL,\cE) \cong  \Hom(\cI_{D'}\cL,\cE)$, 
so  \cref{item.D.ideal} holds.  

\cref{item.D.ideal} $\Rightarrow$ \cref{item.D.some}.
Suppose $\Hom(\cI_{D'}\cL,\cE) \ne 0$. We will show that
$i_D \in \pi \Hom(\cI_D\cL,\cE)$ for some $D \in |D'|$.

First, $\Hom(\cI_{D'}\cL,\cO_E) =0$ because $\deg(\cI_{D'}\cL) =\deg(\cL)-\deg(D')>0$. 
The map  $\Hom(\cI_{D'}\cL,\cE)\to\Hom(\cI_{D'}\cL,\cL)$, $g' \mapsto \pi g'$,  is therefore injective.
Let $g': \cI_{D'}\cL \to \cE$ be any non-zero map. Then $\pi g' \ne 0$ so $\im(\pi g')$ is a non-zero submodule of $\cL$ and therefore equal to
$\cI_{D} \cL$ for a unique effective divisor $D$. 
Hence  $\cI_{D'}\cL \cong \cI_{D}\cL$. It follows that $\cI_{D'} \cong \cI_{D}$, which implies $D' \sim D$, whence $D \in |D'|$.

Now consider the diagram
\begin{equation*}
	\xymatrix{
		\cI_{D'}\cL     \ar@/^1.5pc/ [rr]^{\pi g'}   \ar@{..>}[r] &  \cI_{D}\cL  \ar[r]_{i_D} & \cL .
	}
\end{equation*}
Since $\pi g'(\cI_{D'}\cL)=i_D(\cI_{D}\cL)$, $\pi g'$ factors through $\cI_D\cL$; i.e., $\pi g'= i_D \circ \a$ for some  $\a: \cI_{D'}\cL \to \cI_{D}\cL$. But
$\pi g' \ne 0$, so $\a$ is non-zero and therefore an isomorphism. Hence $i_D = \pi g'\a^{-1} \in \pi \Hom(\cI_D\cL,\cE)$.

\cref{item.D.hom} $\Leftrightarrow$ \cref{item.D.ideal}.  
This is a consequence of the better result that  $\Hom(\cE,\cO_E(D'))  \cong \Hom(\cI_{D'}  \cL , \cE)$, which follows from the isomorphism 
$\cE \cong \cE^\vee \otimes \cL$ in  \Cref{le:edete}: 
\begin{align*}
\Hom(\cE,\cO_E(D')) \; \cong \; \Hom(\cO_E(-D'), \cE^\vee)   \; \cong \; \Hom(\cO_E(-D')\otimes \cL , \cE^\vee\otimes \cL)  
  \; \cong \; \Hom(\cI_{D'}  \cL , \cE). 
\end{align*}

Finally, let $D_{\rm min}$ be an effective divisor of minimal degree such that $\xi \in \overline{D_{\rm min}}$, and let $D \in |D_{\rm min}|$.
Since $D_{\rm min} \in |D|$, \cref{item.D.sec} holds for some divisor in $|D|$ so, because \cref{item.D.some} implies \cref{item.D.hom}, there is a non-zero homomorphism $f:\cE \to \cO_E(D)$. 
Suppose $f$ is not an epimorphism. Then the image of $f$ is isomorphic to $\cO_E(D')$ for some $D'$ whose degree is strictly less than the degree of $D$. 
By \Cref{cor.mapE.to.N}, $\deg D'>0$ so we may assume that $D'$ is effective.
But \cref{item.D.hom} holds for $D'$ so \cref{item.D.some} holds for $D'$, which implies that $\xi \in \overline{D''}$ for some $D'' \in |D'|$. 
However, $\deg D''=\deg D'<\deg D = \deg D_{\rm min}$, which contradicts the definition of $D_{\rm min}$. 
We conclude that $f$ must be an epimorphism. 
\end{proof}

\begin{corollary}
\label{cor.D.min}
Let $D \in E^{[d]}_x$ 
and assume that either $d<n$ or $D\sim H$. Then 
\begin{enumerate}
\item{}
$\xi \in \Sec_{d,x}(E)$ if and only if $\Hom(\cE,\cO_E(D))\ne 0$;
\item{}
if $\xi \in \Sec_{d,x}(E)-\Sec_{d-1}(E)$, then $d$ is minimal such that $\xi \in \Sec_d(E)$ and there is an epimorphism $\cE \to \cO_E(D)$.
\end{enumerate} 
\end{corollary}
\begin{proof}
This follows from the fact that
$ \xi \in \Sec_{d,x}(E) \, \Longleftrightarrow \, \xi \in \overline{D} \text{ for some } D \in E^{[d]}_x$.
\end{proof}

\begin{theorem}
\label{thm.bertram}
Let $\xi \in \PP_\cL$ and write $\cE:=m(\xi)$. If 
\begin{align*}
d_1   & := \min \big\{  d  \; | \; \xi \in \Sec_d(E)  \big \},
\\
d_2   & := \min \big \{  d\; | \;  \xi \in \overline{D} \text{ for some } D \in E^{[d]}  \big \},
\\
d_3   & := \min\big \{  d  \; | \;  \text{there is a non-zero morphism } \cE \to \cO_E(D) \text{ for some } D \in E^{[d]} \big \},
\\
d_4   & := \min\big \{  d  \; | \;  \text{there is an epimorphism } \cE \to \cO_E(D) \text{ for some } D \in E^{[d]} \big \},
\end{align*}
then $d_1=d_2=d_3=d_{4}$. Let $d$ denote this common integer.  
\begin{enumerate}
  \item\label{item.bertram.epi} 
 If $D \in E^{[d]}$, then every non-zero map  $\cE \to \cO_E(D)$ is an epimorphism.
  \item\label{item.bertram.sec} 
$\xi \in \Sec_{d,x}(E)$ if and only if  there is an epimorphism $ \cE \to \cO_E(D)$  for some (equivalently, all) $D \in E^{[d]}_x$.
\end{enumerate}
\end{theorem}
\begin{proof}
The equality $d_1=d_2$ follows from the fact that $\Sec_d(E)= \bigcup \overline{D}$ as $D$ varies over $E^{[d]}$. 

Let $D_2$, $D_3$, and $D_4$ be effective divisors of degrees $d_2$, $d_3$, and $d_4$ such that $\xi \in \overline{D_2}$, $\Hom(\cE,\cO_E(D_3)) \ne 0$,
and there is an epimorphism $\cE \to \cO_E(D_4)$. 
Because $\Hom(\cE,\cO_E(D_4)) \ne 0$, $d_3 \le d_4$. 
By \Cref{prop.D.min}, $d_4 \le d_2$. 
Because \cref{item.D.hom} $\Rightarrow$ \cref{item.D.some} in \cref{prop.D.min},  $\xi \in \overline{D'}$ for some $D'\in |D_3|$; hence $d_2 \le d_3$. 

This completes the proof that $d_1=d_2=d_3=d_{4}$.

\cref{item.bertram.epi} This follows from \Cref{prop.D.min}.

\cref{item.bertram.sec} 
We know $\xi$ is in $\Sec_d(E)$. It is in $\Sec_{d,x}(E)$ if and only if it is in $\overline{D}$ for some $D \in E^{[d]}_x$; however, if $\xi \in \overline{D}$, then $\Hom(\cE, \cO_E(D)) \ne 0$ so, by \cref{item.bertram.epi}, there is an epimorphism $\cE \to \cO_E(D)$ and hence an epimorphism $\cE \to \cO_E(D')$ for all $D' \in |D|$; i.e., for all $D' \in E^{[d]}_x$.
\end{proof}

\begin{remark}\label{re:bertr}
 In contrast to \Cref{prop.bertram}, the discussion at \cite[p.~430, preceding Thm.~1]{Ber92} suggests that $\xi\in \overline{D}$ if and only 
  if there is an epimorphism $\cE \to \cO_E(D)$. 
  This cannot be: the condition $\xi\in \overline{D}$  depends on $D$ as a subscheme of $E$, but the existence of an epimorphism 
  $\cE \to \cO_E(D)$ depends only on the linear equivalence class of $D$. 
\reqed
\end{remark}

The next result provides an alternative proof, and a more complete understanding,
of  the equivalence \cref{item.D.some} $\Leftrightarrow$ \cref{item.D.ideal} in  \Cref{new.version.prop.4.1}; given $\xi \in \PP_\cL$, the set of 
hyperplanes through $\xi$ is a linear system on $\PP_\cL$ -- it is the hyperplane $\xi^\perp$ defined in \cref{sect.SD}; given $D' \in E^{[d]}_x$, 
$|D'|=E_x^{[d]}$ is a linear system on $E$; the next result shows that the subset of $|D'|=E_x^{[d]}$ consisting of those 
$D$ such that $\overline{D}$ passes through $\xi$ is a linear system on $E$ and gives an explicit description of it. First, some notation.

If $s:\cI_{D'}\cL \to \cL$ is a non-zero homomorphism, then its image equals $\cI_D \cL$ for a unique $D \in |D'|$. Conversely, if $D \in |D'|$, there is a 
unique point $\CC s \in \PP \Hom(\cI_{D'}\cL,\cL)$ such that  $\im(s)=\cI_D\cL$. Hence we can, and do, for the purposes of the next result, 
make the identification 
\begin{equation*}
|D'| \;= \; \PP \Hom(\cI_{D'}\cL,\cL)
\end{equation*}
 where $D \in |D'|$ is identified with the point $\CC s$ where $ s: \cI_{D'}\cL \to \cL$ is the unique-up-to-non-zero-scalar-multiple 
 map such that $\im(s)=\cI_D\cL$, i.e., such that  $\im(s)= \im(i_D)$.  

\begin{proposition}
\label{prop.D.bar.bijection}
Let $\xi \in \PP_\cL$ be the  isomorphism class of the sequence \Cref{eq:repr.for.xi}. Let $D' \in E^{[d]}$ with $1 \le d<n$.
The map  $ \Psi(\CC g') :=  \CC \pi g'$ is an isomorphism
\begin{equation*}
\Psi: \, \PP \Hom(\cI_{D'}\cL,\cE)  \; \longrightarrow \; \{D \in |D'| \; | \; \xi \in \overline{D}\}.
\end{equation*} 
\end{proposition}
\begin{proof}
To see that the definition of $\Psi$ makes sense we must show that $\{D \; | \; \im(\pi g')=\cI_D \cL\}$ is contained in $\{D \in |D'| \; | \; \xi \in \overline{D} \, \}$.
With this goal in mind, suppose $\im(\pi g')=\cI_D \cL$. It follows that $\im(\pi g')=\im(i_D) = \im (i_D\a)$ for all isomorphisms $\a:\cI_{D'}\cL \to \cI_D\cL$. 
Since $\pi g'$ and $i_D \a$ are non-zero maps $\cI_{D'}\cL \to \cL$ having the same image, namely  $\cI_D\cL$, $\CC \pi g' = \CC i_D\a$ 
(because $\dim \Hom(\cI_{D'}\cL, \cI_D\cL)=1$).
It follows that $i_D \in \CC \pi g'\a^{-1}$. Hence $i_D \in \pi \Hom(\cI_D\cL,\cE)$, which implies 
$\xi \in \overline{D}$ by \Cref{new.version.prop.4.1}.

 Thus $\Psi$ is a well-defined morphism. 

If $\Psi(\CC g')=\Psi(\CC g'')$, then $\CC \pi g'= \CC \pi g''$.
As we observed in the proof of \cref{item.D.ideal} $\Rightarrow$ \cref{item.D.some}, the map $g' \mapsto \pi g'$ is injective;  hence $\CC g'=\CC g''$. 
Thus, $\Psi$ is injective.

We now show that $\Psi$ is surjective. Let $D \in |D'|$ and suppose $\xi \in \overline{D}$. 
Then the top row in \cref{eq:prop.bertram.SD.2} splits; i.e., there is a map $\phi:\cI_D \cL \to \pi^{-1}(\cI_D\cL)$ such that $\pi' \phi=\id_{\cI_D\cL}$. 
 Let $h:\pi^{-1}(\cI_D\cL) \to \cE$ be the inclusion map and define $g=h\phi:\cI_D\cL \to \cE$.  
 Then $\pi g = \pi h \phi = i_D \pi'\phi = i_D$.  Let $\a:\cI_{D'}\cL \to \cI_D\cL$ be any isomorphism. Then $g\a :\cI_{D'}\cL \to \cE$
 and $\Psi(\CC g \a)= \CC \pi g \a = \CC i_D\a$; since $\im(i_D \a)=\cI_D\cL$, $D=\Psi(\CC g\a)$. Hence $\Psi$ is surjective and therefore an isomorphism.
\end{proof}

\begin{remark}
\label{rem.bertram}
Using the ``usual'' identification $|D'|=\PP H^0(E,\cO_E(D'))$, there is an isomorphism  
\begin{equation*}
\Psi:\PP \Hom(\cI_{D'}\cL,\cE)  \; \longrightarrow \;  \{D \in |D'| \; | \; \xi \in \overline{D}\} \, \subseteq \,\PP H^0(E,\cO_E(D')),
\end{equation*} 
$ \Psi(\CC g') =\big( (\cI_{D'}\cL)^{-1} \otimes  \pi g' \big)_0 =$ the zero locus of the section in $H^0(E,\cO_E(D'))$ that ``is'' the map
\begin{equation*}
(\cI_{D'}\cL)^{-1} \otimes \pi g' : \cO_E  \cong (\cI_{D'}\cL)^{-1} \otimes (\cI_{D'}\cL)   \, \longrightarrow \, (\cI_{D'}\cL)^{-1} \otimes \cL \cong \cO_E(D').
\end{equation*}
\end{remark}

%%%%%%%%%%%%%%%%%%%%%%%%%%%%%%%%%%%%%%%%%%%%%%%%%%%%%%%%%%%%%%%%
\subsection{Description of the leaves $L(\cE)$ in terms of the spaces $\Sec_{d,x}(E) \, - \, \Sec_{d-1}(E)$}
\label{sec.leaves.and.secants}
%%%%%%%%%%%%%%%%%%%%%%%%%%%%%%%%%%%%%%%%%%%%%%%%%%%%%%%%%%%%%%%%

\begin{theorem}\label{th:splitall}
Let $x \in E$.
If either $d <\frac{n}{2}$  or $d=\frac{n}{2}$ and $x \notin \Omega$, then
\begin{equation}\label{eq:splitall}
    L(\cE_{d,x}) \;= \; \Sec_{d,x}(E) \, - \, \Sec_{d-1}(E).
  \end{equation} 
  In particular, $L(\cE_{d,x})$  is locally closed in $\PP_\cL$ and $\dim L(\cE_{d,x})=2d-2$.
\end{theorem}
\begin{proof}
By definition, $\cE_{d,x} \cong \cO_E(D) \oplus \cL(-D)$ for every $D \in E^{[d]}_x$. 
Since $\Hom(\cE_{d,x}, \cO_E(D)) \ne 0$ but  $\Hom(\cE_{d,x},\cO_E(D'))=0$ for all divisors $D'$ of degree $<d$, ``$\subseteq$'' in \Cref{eq:splitall} follows from \Cref{prop.bertram}.

To show ``$\supseteq$'', let $\xi=[0 \to \cO_E \to \cE \to \cL \to 0]$ be a point in the right-hand side of \Cref{eq:splitall}. We  must show that $\cE\cong\cE_{d,x}$.
Since $\xi \in \Sec_d(E)-\Sec_{d-1}(E)$,
$d$ is the smallest integer such that $\xi$ is contained in the linear span of an effective divisor of degree $d$. Hence, 
by \Cref{prop.bertram}, there is an epimorphism $\cE\to\cO_{E}(D)$ for some $D\in E^{[d]}_{x}$.

Suppose $\cE\cong\cE_{e,y}=\cO_{E}(D')\oplus\cL(-D')$ for some $e\in[1,\tfrac{n}{2}]$ and $y\in E$, where $D'\in E^{[e]}_{y}$. If $d<\tfrac{n}{2}$, then 
$\deg \cO_E(D)=d<\frac{n}{2} \le n-e = \deg \cL(-D')$ so $\Hom(\cL(-D'),\cO_{E}(D))=0$, and there is an epimorphism $\cO_{E}(D')\to\cO_{E}(D)$ if and only if $D'\sim D$, so $\cE\cong\cE_{d,x}$. If $d=\tfrac{n}{2}$, then both summands $\cO_{E}(D')$ and $\cL(-D')$ have degree $\tfrac{n}{2}$, so the existence of an epimorphism (or a non-zero morphism) $\cE\to\cO_{E}(D)$ implies that either $\cO_{E}(D')$ or 
$\cL(-D')$  is isomorphic to $\cO_{E}(D)$. Hence $\cE\cong\cE_{d,x}$.

Suppose $n$ is odd. The case $\cE\cong\cE_{o}$ does not occur: $\cE_{o}$ is  indecomposable (hence semistable)  of slope $\tfrac{n}{2}$, so there is no epimorphism  $\cE_o \to \cO_{E}(D)$ if $\deg D<\frac{n}{2}$ (because $n$ is odd).

Suppose $n$ is even.  The case $\cE\cong\cE_{\omega}$ ($\omega\in\Omega$) does not occur: $\cE_{\omega}$ is a non-split self-extension of 
$\cL_{\omega}$; if there is an epimorphism $\cE_\omega \to\cO_{E}(D)$, then there is a non-zero morphism $\cL_{\omega}\to\cO_{E}(D)$, 
which must be an isomorphism because $\deg\cL_{\omega}=\tfrac{n}{2}\geq d=\deg\cO_{E}(D)$. But this implies $x=\omega$, which contradicts the assumption of the theorem that $x\notin\Omega$.

We have covered all possibilities for $\cE$, so ``$\supseteq$'' in \Cref{eq:splitall} follows.

The  local closure of $L(\cE_{d,x})$  follows from \Cref{eq:splitall}, and its dimension is given by \Cref{prop.dim.Sec.dz}.
\end{proof}

\begin{theorem}\label{th:nsplitodd}
If $n=2r+1$, then $ L(\cE_o) =\PP_\cL- \Sec_{r}(E)$.
\end{theorem}
\begin{proof} 
By \cref{th:splitall}, $\Sec_r(E)$ is the union of all the $L(\cE_{d,x})$'s ($1 \le d \le r$, $x \in E$). 
The result now follows from the fact that $\PP_\cL$ is the disjoint union of all the $L(\cE)$'s.
\end{proof}

The next result disposes of the still-missing even-$n$ cases.

\begin{theorem}\label{th:nsplitev}
Assume $n=2r$. Fix $\omega \in \Omega$.  
 As in \cref{ssect.bun.notn}, let $\cL_\omega$ be the unique-up-to-isomorphism invertible $\cO_E$-module of degree $r$
  such that $\s(\cL_\omega)=\omega$, and  let $\cE_\omega$ be the unique-up-to-isomorphism non-split extension of $\cL_\omega$ by $\cL_\omega$.
  \begin{enumerate}
  \item\label{item:disju}  
     $ L(\cE_\omega)\sqcup L(\cL_\omega \oplus \cL_\omega) \;=\; \Sec_{r,\omega}(E) \, - \,  \Sec_{r-1}(E) $.  
   \item\label{item:e0} $L(\cE_\omega)$ is open dense in $\Sec_{r,\omega}(E)$, has dimension $n-2$, and consists of those points in $\Sec_{r,\omega}(E) - \Sec_{r-1}(E)$ that lie on a {\it unique} $r$-secant.
  \item\label{item:e1} $L(\cL_\omega \oplus \cL_\omega)$ has dimension $n-4$ and consists of those points in $\Sec_{r,\omega}(E) - \Sec_{r-1}(E)$ that lie on infinitely many $r$-secants.
  \item\label{item:pencil}
 A point in $\Sec_{r,\omega}(E) - \Sec_{r-1}(E)$ lies on infinitely many $r$-secants if and only if it lies on 
 at least two distinct $r$-secants.
  \end{enumerate}
\end{theorem}
\begin{proof}
Let $\xi \in \PP_\cL$ and  suppose $0 \to \cO_E \to \cE \to \cL \to 0$ represents $\xi$.

\cref{item:disju} Suppose $\xi \in L(\cE_\omega) \sqcup L(\cL_\omega \oplus \cL_\omega)$; equivalently, $\cE$ is isomorphic to either 
$\cE_\omega$ or $\cL_\omega \oplus \cL_\omega$. In both cases, there is an exact sequence $0 \to \cL_\omega \to \cE \to \cL_\omega \to 0$ so, since 
$\deg \cL_\omega=r$, $\Hom(\cE_\omega,\cO_E(D))=0$ for all $D$ of degree $<r$. Hence the integer $d$ in \Cref{thm.bertram} is $r$.
Hence $\xi \in \Sec_r(E)-\Sec_{r-1}(E)$. However, since $\s(\cL_\omega)=\omega$, i.e., since $\cL_\omega \cong \cO_E(D)$ for some, hence all, 
$D$ in $E^{[r]}_\omega$, $\xi \in \Sec_{r,\omega}(E)-\Sec_{r-1}(E)$. Thus
\begin{equation*}
 L(\cE_\omega)\sqcup L(\cL_\omega \oplus \cL_\omega) \; \subseteq \; \Sec_{r,\omega}(E) \, - \,  \Sec_{r-1}(E).
\end{equation*}

To prove the reverse inclusion, let $\xi \in \Sec_{r,\omega}(E)-\Sec_{r-1}(E)$.
By \Cref{cor.D.min}, or  \Cref{thm.bertram}, $r$ is the smallest integer $d$ such that $\xi \in \overline D$ 
 for some $D \in E^{[d]}$. Since  $\xi \in \Sec_{r,\omega}(E)$, that $D$ belongs to $E^{[r]}_\omega$. 
 It follows that 
 $\Hom(\cE,\cO_E(D)) \ne 0$ for all $D \in  E^{[r]}_\omega$ and all non-zero maps $\cE \to \cO_E(D)$ are epimorphisms.
 In particular, there is an epimorphism $\cE \to \cL_\omega$. The kernel of such an epimorphism is isomorphic to $\cL \otimes \cL_\omega^{-1} \cong \cL_\omega$,
 so $\cE$ is isomorphic to either $\cE_\omega$ or $\cL_\omega \oplus \cL_\omega$.
 Hence 
 \begin{equation*}
 \Sec_{r,\omega}(E) \, - \,  \Sec_{r-1}(E)   \; \subseteq \;  L(\cE_\omega)\sqcup L(\cL_\omega \oplus \cL_\omega).
\end{equation*} 
Hence the equality in \cref{item:disju}.
 
\cref{item:e0}, \cref{item:e1}, \cref{item:pencil}.
We will now prove  that points in $L(\cE_\omega)$ lie on a
unique $r$-secant and that points in $ L(\cL_\omega \oplus \cL_\omega)$ lie on infinitely many $r$-secants.
This will prove \cref{item:pencil} and  the parts of \cref{item:e0} and \cref{item:e1} that do not concern dimension or ``open denseness''.

Suppose $\xi \in L(\cE_\omega) \sqcup  L(\cL_\omega \oplus \cL_\omega)$. 

Fix $D' \in E^{[r]}_\omega$. Then $|D'|=E^{[r]}_\omega$ and 
$\cL_\omega \cong \cI_{D'}\cL$. Hence, by \Cref{prop.D.bar.bijection}, there is a bijection
\begin{equation}
\label{eq:bij.lin.system.assoc.to.xi}
\{D \in E^{[r]}_\omega \; | \; \xi \in \overline{D} \} \; \longleftrightarrow \; \PP \Hom(\cL_\omega,\cE).
\end{equation}

Since $\dim_\CC \Hom(\cL_\omega,\cL_\omega \oplus \cL_\omega) = 2$, points in $L(\cL_\omega \oplus \cL_\omega)$ lie on infinitely many $r$-secants. 

To show that a point $\xi \in L(\cE_\omega)$ lies on a unique $r$-secant we will show that $\dim_\CC \Hom(\cL_\omega,\cE_\omega)=1$ (and apply  \Cref{prop.D.bar.bijection}).

By the definition of $\cE_\omega$, there is a non-split extension
$
0 \to \cL_\omega \stackrel{\a} {\longrightarrow}  \cE_\omega  \stackrel{\b}{\longrightarrow}   \cL_\omega  \to 0 
$.
Suppose to the contrary that $\dim_\CC \Hom(\cL_\omega,\cE_\omega)\ge 2$.
Then there is a non-zero map $\gamma:\cL_\omega \to \cE_\omega$ such that $\a(\cL_\omega)\neq \g(\cL_\omega) \cong \cL_\omega$. 
Hence $\b\g \ne 0$. Since $\End(\cL_\omega)=\CC$, $\b\g$ is a non-zero scalar multiple of the identity map $\id_{\cL_\omega}$. 
Hence there is a scalar multiple $\g'$ of $\g$ such that $\b\g'=\id_{\cL_\omega}$. This contradicts the fact that $\cE_\omega$ is indecomposable so
we conclude that no such $\g$ exists. Hence $\PP\Hom(\cL_\omega,\cE_\omega) \cong \PP^0$.  

By \cref{le:rightdim,le:dim.Aut.cE}, $\dim L(\cL_\omega \oplus \cL_\omega)=n-4$.
This completes the proof of \cref{item:e1}.

By \cref{prop.dim.Sec.dz.2}, the dimension of $ \Sec_{r,\omega}(E)$ is $n-2$ so $\dim  L(\cE_\omega)=n-2$.

To complete the proof of \cref{item:e0}, and hence the proof of this theorem, we will now prove the openness claim in \Cref{item:e0}; the density claim in
\Cref{item:e0} follows from openness and irreducibility (\cref{prop.S.dx.irred}). 

Applying the functor $\Hom(-,\cL_\omega)$  to {\it any} extension $\xi=[0 \to \cO_E \to \cE \to \cL \to 0]$ in $\PP_\cL$
  gives rise, via the long exact cohomology sequence attached to it, to a linear map
  \begin{equation}\label{eq:odod}
    H^0(E,\cL_\omega)\longrightarrow \Ext^1(\cL,\cL_\omega) \; \cong \; \Ext^1(\cL_\omega,\cO_E)\; \cong\;  H^0(E, \cL_\omega)^*
  \end{equation}
  with cokernel $\Ext^1(\cE,\cL_\omega)$ (the first isomorphism exists  because  $\cL \cong \cL_\omega \otimes \cL_\omega$).
  Since $\dim \Ext^1(\cL,\cL_\omega) =r$, the rank  of this map is  $r- \dim \Ext^1(\cE,\cL_\omega)$.
  
  In particular, if $\xi \in L(\cE_\omega)$ the rank of the map is $r-1$ because  $\dim \Ext^1(\cE_\omega,\cL_\omega)=1$ and
  if $\xi \in L(\cL_\omega \oplus \cL_\omega)$  the rank of the map is $r-2$ because  $\dim \Ext^1(\cL_\omega \oplus \cL_\omega, \cL_\omega )=2$. 

  Since the condition $\rank \le r-2$ is given by the vanishing of various minors, $L(\cL_\omega \oplus \cL_\omega)$ 
   is a Zariski-closed subset of $L(\cE_\omega)\sqcup L(\cL_\omega \oplus \cL_\omega)$.     Hence  $L(\cE_\omega)$ 
     is a Zariski-open subset of $L(\cE_\omega)\sqcup L(\cL_\omega \oplus \cL_\omega)$.  
\end{proof}

%%%%%%%%%%%%%%%%%%%%%%%%%%%%%%%%%%%%%%%%%%%%%%%%%%%%%%%%%%%%%%%%
%%%%%%%%%%%%%%%%%%%%%%%%%%%%%%%%%%%%%%%%%%%%%%%%%%%%%%%%%%%%%%%%
\section{$L(\cE)$ is a smooth variety and is the geometric quotient $X(\cE)/ \!\Aut(\cE)$}
\label{subse:smth} 
%%%%%%%%%%%%%%%%%%%%%%%%%%%%%%%%%%%%%%%%%%%%%%%%%%%%%%%%%%%%%%%%
%%%%%%%%%%%%%%%%%%%%%%%%%%%%%%%%%%%%%%%%%%%%%%%%%%%%%%%%%%%%%%%%

%%%%%%%%%%%%%%%%%%%%%%%%%%%%%%%%%%%%%%%%%%%%%%%%%%%%%%%%%%%%%%%%%%%%%%%%%%%%% 
\subsection{Properties of $L(\cE)$ and the morphism $X(\cE) \to L(\cE)$}    
\label{subse:xaut}
%%%%%%%%%%%%%%%%%%%%%%%%%%%%%%%%%%%%%%%%%%%%%%%%%%%%%%%%%%%%%%%%%%%%%%%%%%%%%  

We need some terminology (e.g., \cite[\S14.1]{tDieck08}, \cite[\S 1]{df_loctriv}).

\begin{definition}\label{def:loctrivact}
  Let $\a:G\times X \to X$ be an action of a linear algebraic group on a variety. 
  \begin{enumerate}

  \item 
  A {\sf trivialization} of $\alpha$ is a $G$-equivariant isomorphism $m: G\times Y \to X$, with $G$ acting on $G\times Y$ by $g\cdot(h,y)=(gh,y)$. 
 We then call $m(1,Y)$ a {\it slice} for the action (cf. \cite[p.~82]{luna-slice}). 
    
  \item 
  We say the action is {\sf trivial} if it admits a trivialization. 

  \item\label{item.def.loctrivact.loctriv} 
  The action is {\sf Zariski-locally trivial} if $X$ admits a cover by $G$-invariant open subschemes on which the restricted $G$-actions are trivial.
  \end{enumerate}
\end{definition}

\begin{remark}\label{re:loctrivinherit}
  If $\a:G\times X\to  X$ is locally trivial, so is its restriction $\alpha'$ to any $G$-invariant open subscheme $X'\subseteq X$: 
  if $Y\subseteq X$ is a slice for $\alpha$ then $Y':=Y\cap X'$ is a slice for $\alpha'$.
 We will use this remark repeatedly but (mostly) tacitly  from now on.
\end{remark}

\begin{definition}\label{def:geomquot}
  A  {\sf geometric quotient} for an action $\a:G \times X \to X$ of a linear algebraic group $G$ on a variety $X$ is a pair $(Y,q)$ consisting of a scheme 
  $Y$ and a morphism $q:X\to Y$ such that 
  \begin{enumerate}
  \item $q$ is open and surjective,
  \item $(q_*\cO_X)^G = \cO_Y$, and
  \item the geometric fibers of $q$ are the geometric $G$-orbits for the action.
  \end{enumerate}
\end{definition}

As noted in \cite[discussion following Defn.~1.1]{gp1}, this is equivalent to \cite[Defn.~0.6]{mumf-git} and \cite[\S II.6.3, p.~95]{borel-LAG-se} (since our ``variety'' is always of finite type over $\bbC$). As Borel remarks in \cite[\S II.6.16, p.~103]{borel-LAG-se}, what we are calling a geometric quotient is the same as what he simply calls a ``quotient''.

\Cref{def:loctrivact} applies to {\it actions}, but the terms ``(locally) trivial'' are also applicable to {\it quotients} (e.g., \cite[Defn.~2.7]{derk_quot}, \cite[\S XI.4.7]{sga1book}, \cite[\S 2.2]{SCC_1958__3__A1_0}, etc.).

\begin{definition}\label{def:loctriv}
  A geometric quotient $q:X\to X/G$ for an action of a linear algebraic group is {\sf locally trivial} if there is an open cover
  \begin{equation*}
    X/G \; = \; \bigcup V_i
  \end{equation*}
  such that for each $i$ there is a $G$-equivariant isomorphism $G \times V_i \to q^{-1}(V_i)$, where $G$ 
  acts on  $G \times V_i $ via $g \cdot (h,v)=(gh,v)$. Thus, if we identify $q^{-1}(V_i)$ with $G \times V_i$, then $q$ is the projection onto the 
  second coordinate.
\end{definition}

``Freeness'' is another measure of how well-behaved an action can be. For the action of $\Aut(\cE)$ on $X(\cE)$ the notion of freeness 
in \Cref{le:autfree} is the naive set-theoretic one: the identity is the only element in $\Aut(\cE)$ having a fixed point.  In the present context of algebraic-group actions the more appropriate concept of freeness is that in \cite[Defn.~0.8]{mumf-git}.  The terminology used below comes from \cite[\S I.1.8, p.~53]{borel-LAG-se} (and other sources in other contexts, e.g., \cite[\S 1]{ell-princ}).

\begin{definition}\label{def:princ}
  An action $\a:G \times X \to X$
   of an algebraic group on an algebraic variety is {\sf principal} if
  \begin{equation}\label{eq:can}
\a_+=(\a,\pi_2)   : G \times X \to X \times X
  \end{equation}
  is a closed immersion, where $\pi_2$ denotes the projection $G \times X \to X$.
  \reqed
\end{definition}
We might occasionally refer to principal actions as {\it algebro-geometrically-free} or {\it ag-free} for short.

There is a distinction:  \cite[Ex.~0.4]{mumf-git} is an otherwise fairly well-behaved set-theoretically free action that is not principal in the sense of \Cref{def:princ}.

\begin{theorem}\label{th:actloctriv}  \label{cor:xony}
Suppose $L(\cE) \ne \varnothing$.
  \begin{enumerate}
  \item\label{item:th:actloctriv.1} The action of $\Aut(\cE)$ on $X(\cE)$ is locally trivial.

  \item\label{item:th:actloctriv.2} The cover in \cref{def:loctrivact}\cref{item.def.loctrivact.loctriv} can be chosen so that each member of it admits a trivialization whose slice is smooth.

  \item\label{item:th:actloctriv.3} The action of $\Aut(\cE)$ on $X(\cE)$ is principal.
  
  \item\label{item:th:actloctriv.4} The geometric quotient $X(\cE)/\!\Aut(\cE)$ exists, is smooth, and is locally trivial in the sense of \Cref{def:loctriv}.

  \end{enumerate}
\end{theorem}
\begin{proof}
  The proof proceeds by showing
  \begin{enumerate}[(a)]
  \item\label{item.proof.actloctriv.cv} there is an open cover $\{U_i\}$ of $X(\cE)$ such that each $U_i$ is stable under the action of $\Aut(\cE)$, and
  \item\label{item.proof.actloctriv.eq} there are $\Aut(\cE)$-equivariant isomorphisms $\Aut(\cE) \times Y_i \to U_i$ where $\Aut(\cE)$ acts on $\Aut(\cE) \times Y_i$ by $g\cdot (h,y) = (gh,y)$, and
  \item\label{item.proof.actloctriv.sm} each $Y_i$ is smooth.
  \end{enumerate}

  {\bf \Cref{item:th:actloctriv.1} and \Cref{item:th:actloctriv.2}.} We will prove these simultaneously.

    There are four cases to treat, one for each of the possible $\Aut(\cE)$ as classified in \Cref{le:dim.Aut.cE}.  We will recall the corresponding 
    structure of $\cE$ in the course of addressing the cases in turn.
    
    Throughout we will use the notation $V:=\Hom(\cO_E,\cE)$.
    \begin{enumerate}[(I), inline]
      
    \item {\bf $\cE$ indecomposable and $n$ is odd; i.e., $\cE=\cE_o$.}  In this case $\Aut(\cE) \cong \GG_m$ and it acts on 
    $X(\cE)\subseteq \Hom(\cO_E,\cE)$ by scaling.  By \cref{re:loctrivinherit}, it suffices to show that the scaling action of $\GG_m$ on $V^{\times}=V-\{0\}$ is locally trivial. This happens because $V^{\times}$ is covered by the $\GG_m$-invariant open subsets
    \begin{equation}\label{eq:vf}
      V_f:=\{v\in V \; | \;  f(v)\ne 0\},\quad 0\ne f\in V^*,
    \end{equation}
    and the $\GG_m$-equivariant isomorphisms
    \begin{equation*}
      \GG_m \times \{v\in V \; | \;  f(v)=1\} \to   V_f, \qquad (\lambda,v) \mapsto \lambda v,
    \end{equation*} 
    provide locally trivializations: the set $\{v\in V \; | \; f(v)=1\}$ is a smooth slice.
    
  \item\label{item:deceven} {\bf $\cE$ indecomposable and $n$ is even.} By \Cref{ssect.bun.notn}, $\cE \cong \cE_\omega$ for a unique $\omega \in \Omega$ and there is  a non-split extension
    \begin{equation}\label{eq:nen}
      0\to \cL_\omega \to \cE\to \cL_\omega \to 0.
    \end{equation}
    Since $\deg\cL_\omega>0$, $H^{1}(E,\cL_\omega)=0$. Hence there is an exact sequence
    \begin{equation*}
      0\to \Hom(\cO_E,\cL_\omega) \to V  \to \Hom(\cO_E,\cL_\omega) \to 0.
    \end{equation*}
    Let $V_1$ denote the image of $\Hom(\cO_E,\cL_\omega)$ in $V$ and fix a  (non-canonical!)  complementary subspace $V_0$ for the duration of this proof. 
    Thus $V= V_0\oplus V_1$.  

    A point 
    $ (s,a)\in  \Aut(\cE) \cong \Hom(\cL_\omega,\cL_\omega) \times \Aut(\cL_\omega) \cong  \GG_a\times \GG_m $
    acts on a column vector
    $    v= (v_0,v_1)^t\in V_0\oplus V_1 =  V$
    as left multiplication by the matrix
    \begin{equation*}
      \begin{pmatrix}
        a&0\\
        s&a
      \end{pmatrix}.
    \end{equation*}
    Since $V_0^{\times}\times V_1$  is an $\Aut(\cE)$-invariant open subset  of $\Hom(\cO_E,\cE)$ containing
    $X(\cE)$, it suffices, by \Cref{re:loctrivinherit},   to show that the action of $\Aut(\cE)$ on $V_0^{\times}\times V_1$ is locally trivial.
    For each non-zero $f\in V_0^* \cong V_1^*$ the set
    \begin{equation*}
      \{(v_0,v_1)\in V_0^{\times}\times V_1\;|\; f(v_0)=1,\ f(v_1)=0\}
    \end{equation*}
    is a smooth slice for the restricted action on $V_{0,f}\times V_1\subseteq V_0^{\times}\times V_1$ (with $V_{0,f}$ as in \Cref{eq:vf}). 
        
  \item\label{item:n12} {\bf $\cE\cong \cN_1\oplus \cN_2$ with $\cN_1 \not\cong \cN_2$.} 
  Thus $V=V_1 \oplus V_2$ where $V_i:=\Hom(\cO_E,\cN_i)$. Since $X(\cE) \subseteq V_1\oplus V_2$ consists of those $f \in \Hom(\cO_E,\cE)$ 
  that do    {\it not} factor through any $\cO_E(x)\to \cE$, $x\in E$, $f = (v_1,v_2) \in V_1\oplus V_2$  belongs to $X(\cE)$ if and only if the two 
  $v_i$'s have no common zero     as sections of $\cE$. 
    
    We assume without loss of generality that  $\deg\cN_1\le \deg \cN_2$.
    The group $\Aut(\cE)$ described in 
    \Cref{le:dim.Aut.cE}\Cref{item:mostcompl} operates on column vectors $(v_1,v_2)^t$
    by ``multiplication'' via matrices
    \begin{equation*}
      \begin{pmatrix}
        a&0\\
        s&b
      \end{pmatrix}
      ,\quad
      a,b\in \GG_m,\quad
      s\in  \Hom(\cN_1,\cN_2)
    \end{equation*}
    where multiplication by $s$ is  the composition,
    $\Hom(\cO_E,\cN_1)\ni v_1\xmapsto{\quad} sv_1\in \Hom(\cO_E,\cN_2)$.
    The set 
    \begin{equation}\label{eq:natdomain}
     X' \; :=\;  \{(v_1,v_2)\; |\; v_1 \ne 0, \;  v_2\not\in \Hom(\cN_1, \cN_2) v_1\}
    \end{equation}
    contains $X(\cE)$ and     is stable under the action of $\Aut(\cE)$ on $\Hom(\cO_E,\cE)$.
    
    We will now show that the action of $\Aut(\cE)$ on $X'$   is locally free. 

With this goal in mind, let $(v_1,v_2) \in X'$.  Choose a subspace $W_2 \subseteq V_2$ that  contains $v_2$ and 
    is a complement to $\Hom(\cN_1,\cN_2)v_1$ in $V_2$. 
    Being a complement to $\Hom(\cN_1,\cN_2)v_1$ is an open  condition on the set of pairs
$
  \left(
    \CC v_1,
    \text{subspace }W_2\subseteq V_2
  \right),
$  
 so we now choose a $\GG_m$-invariant open set $U_1 \subseteq V_1^\times$ that contains $v_1$ and is such that $W_2$ is a  complement in $V_2$ 
 to $\Hom(\cN_1,\cN_2) u$ for all $u \in U_1$. Now choose $f_1 \in V_1^*$ and $f_2 \in W_2^*$ such that $f_1(v_1)=1=f_2(v_2)$. 
The set $\{(u,w)\in U_1\times W_2\;|\; f_1(u) = 1 = f_2(w)\}$ is a smooth slice for a free action of $\Aut(\cE)$ on its  $\Aut(\cE)$-saturation. 
Since $(v_1,v_2)$ belongs to this slice, 
the action of $\Aut(\cE)$ on $X'$ is locally free.   By \Cref{re:loctrivinherit},  the action of $\Aut(\cE)$ on $X(\cE)$ is locally free.
    
  \item\label{item:en2} {\bf $n$ is even and $\cE = \cL_\omega \oplus \cL_\omega$ for some $\omega \in \Omega$.}  
  In this case, $\Aut(\cE) \cong {\rm GL}(2)$ by \Cref{le:dim.Aut.cE}. 
  Similar to \Cref{item:n12}, $X(\cE)$ 
 consists of those pairs of sections $(s_1,s_2) \in \Hom(\cO_E,\cL_\omega) \oplus \Hom(\cO_E,\cL_\omega)$  that have no common zeros.  
 Thus $X(\cE)$ is an open subset of the larger $\Aut(\cE)$-invariant set $X'\subseteq \Hom(\cO_E,\cL_\omega)^{\oplus 2}$ of 
 linearly-independent     pairs of vectors in $\Hom(\cO_E,\cL_\omega)$. The geometric quotient $X'/{\rm GL}(2)$ exists:
it is the Grassmannian $\GG(2,\Hom(\cO_E,\cL_\omega))$ of 2-planes in 
$\Hom(\cO_E,\cL_\omega)$, which is smooth, and $X'\to \GG(2,\Hom(\cO_E,\cL_\omega))$ is the 
{\it 2-frame bundle} \cite[\S C.6.1]{3264} associated to the rank-two {\it universal subbundle} \cite[\S 3.2.3]{3264} on 
    $\GG(2,\Hom(\cO_E,\cL_\omega))$. 
    Since a bundle is locally trivial, so is its frame bundle $X'$ (in the sense of \Cref{def:loctrivact}). 
    Hence the $\Aut(\cE)$ action on $X(\cE)$ is     locally trivial.
    \end{enumerate}

    This completes the proof of \Cref{item:th:actloctriv.1} and \Cref{item:th:actloctriv.2}, and also produces the $U_i$'s and $Y_i$'s satisfying the statements \cref{item.proof.actloctriv.cv}, \cref{item.proof.actloctriv.eq}, \cref{item.proof.actloctriv.sm} at the beginning of this proof.

{\bf \Cref{item:th:actloctriv.3}} 
    The $\Aut(\cE)$ action on each $U_i$ is certainly principal (a trivial action is always principal). It follows that the 
    $\Aut(\cE)$ action on $X(\cE)$ is principal.

{\bf \Cref{item:th:actloctriv.4}} When $\Aut(\cE)$ acts only on the left-hand factor of $\Aut(\cE) \times Y_i$, the projection $\Aut(\cE) \times Y_i \to Y_i$ is a geometric quotient.  This follows, say, from \cite[discussion in \S 0.2, item (4) following Prop.~0.1]{mumf-git} whenever the algebraic group is universally open over the base scheme; for us the base scheme is a field, so that hypothesis holds. By the universality of the geometric quotient (i.e., the fact that it is also a {\it categorical} quotient \cite[Defn.~0.5, Prop.~0.1]{mumf-git}), the quotients $Y_i$ glue along
    \begin{equation*}
      Y_{ij}:=\text{geometric quotient of }\Aut(\cE)\times U_{ij}\to U_{ij},\quad U_{ij}:=U_i\cap U_j 
    \end{equation*}
    to produce a scheme $Y$ that is obviously a geometric quotient.  The smoothness of $Y$ follows from the smoothness of each $Y_{i}$.

    The quotient $X(\cE) \to Y$ is locally trivial because the copies of $Y_i$ in $Y=X(\cE)/\!\Aut(\cE)$ provide the open cover of $X(\cE)/\!\Aut(\cE)$ demanded by \Cref{def:loctriv}.
\end{proof}

\begin{theorem}\label{th:smth}\leavevmode
  \begin{enumerate}
  \item\label{item:smth} The homological leaf $L(\cE)$ is smooth.
  \item\label{item:git} The corestriction $X(\cE) \to L(\cE)$ of the map $X(\cE) \to \PP_\cL$ in \Cref{le:ismor} realizes $L(\cE)$ as the geometric quotient
    \begin{equation}\label{eq:geoquot}
      L(\cE) \; \cong \; \frac{ X(\cE) }{\Aut(\cE)} \, .
    \end{equation}
  \end{enumerate}
\end{theorem}
\begin{proof}
{\bf \Cref{item:git}} 
If $L(\cE)$ is smooth,  then \Cref{item:git} follows from \cite[Prop.~II.6.6, p.~97]{borel-LAG-se} which says the following: 
let $G$ act on $X$ and assume the irreducible components of $X$ are open; 
if $Y$ is normal and $q:X \to Y$ is a separable \cite[\S AG.8.2, p.~20]{borel-LAG-se} orbit map in the sense of \cite[\S II.6.3, p.~95]{borel-LAG-se}, 
i.e., it is surjective and its fibers are the $G$-orbits, then $(Y,q)$ is the geometric quotient\footnote{See the sentence just after \cref{def:geomquot}.} of $X$ by $G$.

    The hypotheses of  \cite[Prop.~II.6.6, p.~97]{borel-LAG-se}  hold  for the morphism $X(\cE) \to L(\cE)$:
    first, it is  surjective and its fibers are the $\Aut(\cE)$-orbits
    (\Cref{le:Aut.cE.action.on.extns}) so it is an {\it orbit map};
    second, its domain, $X(\cE)$, is irreducible; third, $L(\cE)$ is normal because it is smooth; fourth, $X(\cE) \to L(\cE)$ is {\it separable}  
    because the base field is $\CC$.

{\bf \Cref{item:smth}}  
After \Cref{th:splitall,th:nsplitodd,th:nsplitev}, or their summary in \Cref{thm.main},  to prove \Cref{item:smth}  we must show the following
varieties are smooth:
\begin{enumerate}[(a)]
\item
$L(\cE_{d,x})$
when $d<\frac{n}{2}$ and $x \in E$, and when $d=\frac{n}{2}$ and $x \in E-\Omega$; 
\item
$L(\cE_o)$ when $n$ is odd;  
\item
$L(\cE_\omega)$  and $L(\cL_\omega \oplus \cL_\omega)$ when $n$ is even and $\omega \in \Omega$.
\end{enumerate}
We have already proved (a) and (b). By \Cref{th:splitall}, $L(\cE_{d,x})=\Sec_{d,x}(E) - \Sec_{d-1}(E)$ and this is smooth by 
\Cref{prop.8.15.GvB-H}. By \Cref{th:nsplitodd}, $L(\cE_o) \subseteq \PP_\cL$ is dense and open, hence smooth.

(c${}_1$)
 {\bf Proof that $L(\cE_\omega)$ is smooth when $n=2r$ and $\omega \in \Omega$.}
 By \Cref{th:nsplitev}, 
 \begin{equation*}
 L(\cE_\omega) \; =\; \{\xi \in \Sec_{r,\omega}(E) - \Sec_{r-1}(E) \; | \; \xi \text{ lies on a {\it unique} $r$-secant}\}.
 \end{equation*}
 
 We define $Y:=\nu^{-1}(L(\cE_\omega))$ where $\nu:\EE \to \PP_\cL$ is the projection in \Cref{eq:incid.var.EE}.
 By \Cref{th:nsplitev}, $L(\cE_\omega)$ is a dense open subset of $\Sec_{r,\omega}(E)-\Sec_{r-1}(E)$, so $Y$ is a dense open
 subset of  $\EE$ and therefore smooth. 
 The proof builds on that of \Cref{prop.8.15.GvB-H}.
 We will prove $L(\cE_\omega)$ is smooth by showing that $\nu|_Y$ is an isomorphism; i.e., by showing
  that $\nu|_Y$ is injective and  $d\nu_y$ is injective for all $y \in Y$  \cite[Cor.~14.10]{H92}.
 
First, $\nu|_Y$ is injective: if $\nu(D_1,\xi_1)=\nu(D_2,\xi_2)$, then $\xi_1=\xi_2$ and belongs to 
 $\overline{D_1} \cap \overline{D_2}$ but, by \Cref{th:nsplitev},  each point in $L(\cE_\omega)$ lies on a unique $r$-secant, so 
 $\overline{D_1}= \overline{D_2}$ whence $D_1=D_2$ by \Cref{prop.sec.low.deg}.

 We will now show that $d\nu_y$ is injective at  $y=(D,\xi) \in Y$.  
 
 By \Cref{th:nsplitev}, $D$ is the unique divisor in $E^{[r]}_\omega$ such that $\xi \in \overline{D}$. 
 We define $\cL_\omega :=\cI_D\cL$. Since $D \in E^{[r]}_\omega$ and $\omega \in \Omega$, $\cO_E(2D) \cong \cL$. 
 Hence $\cL_\omega \cong \cO_E(D)$ and $\cL \cong \cL_\omega \otimes \cL_\omega$.\footnote{There is a non-split extension   $0 \to \cL_\omega \to \cE_\omega \to \cL_\omega \to 0$,  but we don't use this fact here.}
 We fix a section $s_1$ of $\cL_\omega$ such that $(s_1)_0=D$ and consider the homomorphisms $\d$ and $\a$ in the exact sequence 
  \begin{equation*}
 \xymatrix{
 0  \ar[r] &  \cO_E \ar[r]^<<<<{\d} &  \cO_E(D)  \ar[r]^<<<<{\a} &  \cO_D(D) \cong \cL_\omega\big\vert_D = \cL_\omega \otimes \cO_D \ar[r] & 0
 }
 \end{equation*} 
  given by $\d(1)=s_1$ and $\a(s)=s|_D$. We use the same labels for the maps in the cohomology sequence 
 \begin{equation*}
 \xymatrix{
 0  \ar[r] &  H^0(E,\cO_E) \ar[r]^<<<<{\d} &  H^0(E,\cL_\omega)  \ar[r]^<<<<{\a} &  H^0(E,\cL_\omega \otimes \cO_D)  \ar[r]^<<<<\b &  H^1(E,\cO_E) \ar[r] &0.
 }
 \end{equation*} 
 
 As in \Cref{prop.8.15.GvB-H}, to show that $d\nu$ is injective at $y$ it suffices to show that $\phi$ is injective at $\xi$ where, as in  \Cref{diag.replaces.gvb.diag.p53},  
$\phi$ is the restriction to $T_D E_\omega^{[r]} \otimes  \cO_{\overline{D}} =  \im(\a) \otimes \cO_{\overline{D}}$  of the map 
 \begin{equation*}
 \tilde{\phi}: T_D E^{[r]} \otimes  \cO_{\overline{D}} = H^0(E,\cO_D(D)) \otimes \cO_{\overline{D}}  \, \longrightarrow \, 
H^0(E, \cI_D\cL)^* \otimes \cO_{\overline{D}} (1)  \;=\;   \cN_{\overline{D}/\PP_\cL}
  \end{equation*} 
 that is induced by the composition
 \begin{equation*}
 \xymatrix{
H^0(E,\cO_D(D)) \otimes H^0(E, \cI_D \cL)   
\ar[r] & H^0(E,\cO_D(D)) \otimes H^0(D, \cI_D \cL \otimes \cO_D)  \ar[r] & H^0(D, \cL \otimes \cO_D).
}
  \end{equation*} 
  The second arrow in this composition comes from identifying $\cO_D(D)$ with $\cI_D^{-1} \otimes \cO_D$, then 
  using  the multiplication map  
$(\cI_D^{-1} \otimes \cO_D) \otimes_{\cO_D} (\cI_D \cL \otimes \cO_D) \to \cL \otimes \cO_D$ (which is an isomorphism).
The first arrow is induced by the restriction map $H^0(E, \cI_D \cL)  \to H^0(D, \cI_D \cL \otimes \cO_D)$. 
  The right-most term $H^0(D, \cL \otimes \cO_D)$ is naturally isomorphic to $H^0(\PP_\cL,  \cO_{\overline{D}}(1))$ which is 
  the space of linear forms on $\overline{D}$.

To show that $\phi$ is injective at $\xi$ we must show that the kernel (at $\xi$) of the composition
  \begin{equation*}
 \xymatrix{
   H^0(E,\cO_E(D))    \otimes \cO_{\overline{D}} 
   \ar@{=}[d]
    \ar[rr]^{\a\otimes  \cO_{\overline{D}}} &&  H^0(E,\cO_D(D) )   \otimes \cO_{\overline{D}} 
    \ar@{=}[d]
   \ar[r]  & H^0(E, \cI_D\cL)^* \otimes \cO_{\overline{D}}  (1)
   \ar@{=}[d]
   \\
   H^0(E,\cL_\omega)    \otimes \cO_{\overline{D}} &&  H^0(E,\cL_\omega \otimes \cO_{D})    \otimes  \cO_{\overline{D}} 
   &  H^0(E,\cL_\omega)^*    \otimes  \cO_{\overline{D}} (1)
 }
 \end{equation*}  
 equals the image (at $\xi$) of the map
  \begin{equation*}
 \xymatrix{
 H^0(E,\cO_E)    \otimes \cO_{\overline{D}} \ar[rr]^<<<<<<<<<{\d \otimes  \cO_{\overline{D}}}&&  H^0(E,\cL_\omega)  \otimes  \cO_{\overline{D}}.
  }
 \end{equation*} 
 That image is 
 $\CC s_1 \otimes \cO_{\overline{D}}/\fm_\xi$.   Suppose the kernel of $\tilde{\phi}$ at $\xi$ is larger than this; then it
 contains  $\CC s_2 \otimes \cO_{\overline{D}}/\fm_\xi$ for some  $s_2 \in H^0(E,\cO_D(D)) - \CC s_1$.
 Since $s_2 \otimes (\cO_{\overline{D}}/\fm_\xi)$ is in the kernel, $\xi \in \overline{(s_2)_0}$. 
 But $s_2 \notin \CC s_1$, so  $(s_2)_0 \ne D$; hence $\xi$ is contained in two different $r$-secants, 
 $\overline{D}$ and $ \overline{(s_2)_0}$,  which contradicts  \Cref{th:nsplitev}. 
 We conclude that $\phi$ is injective at $\xi$ and hence that $d\nu_y$ is injective.

(c${}_2$)       {\bf Proof that  $L(\cL_\omega \oplus \cL_\omega)$ is smooth when $\omega \in \Omega$ and $n=2r$.} 
    
    Let $\cE=\cL_\omega \oplus \cL_\omega$, and fix an isomorphism $\cL_\omega \otimes \cL_\omega \to \cL$. 

By the argument in part \Cref{item:en2} of the proof of \Cref{th:actloctriv} we can, and in this proof we do, 
identify the geometric quotient $X(\cE)/\! \Aut(\cE)$ with an open subset of the Grassmannian $\GG(2,\Hom(\cO_E,\cL_\omega))$.
The  argument is this: $X(\cE)$ consists of the pairs of sections of $\cL_\omega$ that have no common zeros; 
this is a subset of the  $\Aut(\cE)$-invariant set $X'$ of linearly independent pairs of sections; 
$X'/\Aut(\cE) = X'/{GL}(2) = \GG(2,H^0(E,\cL_\omega))$.

      We will make use of the incidence variety
      \begin{equation}\label{eq:defk}
        \KK \; := \;  \left\{(p,\xi)\in \GG(2,\Hom(\cO_E,\cL_\omega))\times \PP_\cL \;  |\; p\cdot H^0(E,\cL_\omega)\subseteq \xi^\perp \right\},
      \end{equation}
      where  $p \cdot H^0(E,\cL_\omega)$ denotes the image of $p \otimes  H^0(E,\cL_\omega)$ under the multiplication map
          $H^0(E,\cL_\omega)\otimes H^0(E, \cL_\omega)\to H^0(E,\cL)$.

There are two cases.  Let $p\in \GG(2,\Hom(\cO_E,\cL_\omega))$ be a 2-plane spanned by linearly independent sections 
$s_1,s_2 \in H^0(E,\cL_\omega)$ whose zero loci are  $D_1$ and $D_2$. 

(i)   
 Suppose $D_1 \cap D_2 \ne \varnothing$; i.e.,  $p$ is not in the subset of $\GG(2,\Hom(\cO_E,\cL_\omega))$ that identifies with
 $X(\cE)/\!\Aut(\cE)$. 
We claim that
      \begin{equation}\label{eq.h.contain}
	\{\xi \in\PP_\cL \; | \;  p\cdot H^{0}(E,\cL_\omega)\subseteq \xi^\perp \}\;=\;\overline{\gcd(D_{1},D_{2})}\;\subseteq\;\Sec_{r-2}(E).
      \end{equation}
      By definition, $p\cdot H^{0}(E,\cL_\omega)=s_{1}\cdot H^{0}(E, \cL_\omega)+s_{2}\cdot H^{0}(E,\cL_\omega)$. Since 
      $s_{j}\cdot H^{0}(E,\cL_\omega)\subseteq H^{0}(E,\cL(-D_{j}))$,
      \begin{align*}
	s_{1}\cdot H^{0}(E,\cL_\omega)+s_{2}\cdot H^{0}(E,\cL_\omega)& \; \subseteq  \;H^{0}(E,\cL(-\gcd(D_{1},D_{2}))),\\
	s_{1}\cdot H^{0}(E,\cL_\omega)\cap s_{2}\cdot H^{0}(E,\cL_\omega)& \;\subseteq \; H^{0}(E,\cL(-\lcm(D_{1},D_{2}))). 
      \end{align*}
      Hence 
      \begin{equation}\label{eq.dim.sH}
	\begin{split}
          \dim(s_{1}\cdot H^{0}(E,\cL_\omega)+s_{2}\cdot H^{0}(E,\cL_\omega))& \; \leq \; \dim H^{0}(E,\cL(-\gcd(D_{1},D_{2})))
          \\
          & \phantom{xxxxxx} 
           \;=\; n-\deg(\gcd(D_{1},D_{2})),\\
          \dim(s_{1}\cdot H^{0}(E,\cL_\omega)\cap s_{2}\cdot H^{0}(E,\cL_\omega))& \;\leq\; \dim H^{0}(E,\cL(-\lcm(D_{1},D_{2})))
           \\
          & \phantom{xxxxxx} 
           \;=\;  n-\deg(\lcm(D_{1},D_{2})),
	\end{split}
      \end{equation}
      where the last equality uses the fact that $\deg(\lcm(D_{1},D_{2}))<r+r=n$ because $D_1 \cap D_2 \ne \varnothing$. 
      Since $\dim (s_{i}\cdot H^{0}(E,\cL_\omega))=r$ and $\gcd(D_{1},D_{2})+\lcm(D_{1},D_{2})=D_{1}+D_{2}$, 
      the inequalities in \cref{eq.dim.sH} are equalities. Therefore
      \begin{align*}
	p\cdot H^{0}(E,\cL_\omega) & \;=\; s_{1}\cdot H^{0}(E,\cL_\omega)+s_{2}\cdot H^{0}(E, \cL_\omega)
	\\
	& \;=\;H^{0}(E, \cL(-\gcd(D_{1},D_{2}))),
      \end{align*}
      so, by \cref{prop.bertram},  $p\cdot H^{0}(E,\cL_\omega)\subseteq \xi^\perp$ if and only if 
      $\xi \in\overline{\gcd(D_{1},D_{2})}$.\footnote{In order to apply \cref{prop.bertram} one needs to take account of an
      abuse of notation in this paragraph: 
      by definition, $s_{j}\cdot H^{0}(E, \cL_\omega)$ is the image of $s_j \otimes H^{0}(E, \cL_\omega)$ in $H^0(E,\cL)$ so, in reality, 
      \begin{equation*}
      s_{j}\cdot H^{0}(E, \cL_\omega) \;=\; i_{D_j} \Hom(\cO_E,\cI_j\cL)
      \end{equation*}
      where       $\cI_j$ is the ideal vanishing on $D_j$ and $i_{D_j}:\cI_{j} \cL \to \cL$ is the inclusion map.
      Likewise, $H^{0}(E,\cL(-\gcd(D_{1},D_{2})))=i_3 H^{0}(E,(\cI_1 + \cI_2) \cL)$  and 
      $H^{0}(E,\cL(-\lcm(D_{1},D_{2})))=i_4 H^{0}(E,(\cI_1 \cap \cI_2) \cL)$ where $i_3:(\cI_1 + \cI_2) \cL \to \cL$ and 
       $i_4:(\cI_1 \cap \cI_2) \cL \to \cL$ are the inclusions.}
     Since $s_{1}$ and $s_{2}$ are linearly independent, $D_1 \ne D_2$; hence $\deg(\gcd(D_{1},D_{2})) \le r-1$;
     if  $\deg(\gcd(D_{1},D_{2}))=r-1$, then the fact that  $\sigma(D_{1})=\sigma(D_{2})$  implies that $D_1=D_2$;
     we conclude that $\deg(\gcd(D_{1},D_{2})) \le r-2$, which proves \cref{eq.h.contain}.

(ii)      
Suppose $D_1 \cap D_2 = \varnothing$. 
Then 
\begin{equation*}
f \, :=\, (s_{1},s_{2}) \, \in \,  H^{0}(E,\cL_\omega)\oplus H^{0}(E,\cL_\omega) \; = \; \Hom (\cO_E,\cE)
\end{equation*}
 belongs to $X(\cE)$.
 Let $\pi:\cE \to \cL$ be the epimorphism that corresponds to $f$ under the bijection $X(\cE) \to {\rm Epi}(\cE,\cL)$ in 
 \Cref{re:kercoker},  
 and let $\xi :=\Psi_{\cE}(f) \in \PP_\cL$  be the corresponding extension where $\Psi_{\cE}:X(\cE) \to \PP_\cL$ is the map 
 defined in \cref{le:ismor}.
 Let $\wedge$ denote the composition 
      \begin{equation*}
      \wedge: H^{0}(E, \cE)\otimes H^{0}(E,\cE)  \, \longrightarrow \,   H^{0}(E,\cE\otimes\cE) \, \longrightarrow \,  H^{0}(E, \det\cE) \; \cong \; 
      \Hom(\cO_E, \cL)
      \end{equation*}
      where the second map is induced by the natural map $\cE\otimes\cE\to\det\cE$ (used in \cref{le:edete}).
Hence
      \begin{align*}
	p\cdot H^{0}(E,\cL_\omega)
	& \;=\; s_{1}\cdot H^{0}(E,\cL_\omega)+s_{2}\cdot H^{0}(E,\cL_\omega)\\
	& \; = \; f \wedge \Hom(\cO_E,\cE)  \\
	& \; = \; \pi \Hom(\cO_E,\cE)
	\\
	& \; = \; \xi^\perp,
      \end{align*}
      where
      the third equality holds because $\pi g=f \wedge g$ for every $g\in H^{0}(E,\cE)$, which can be verified locally using \cref{third.isom},
      and the last equality comes from  \cref{lem.xi.perp}.
   
 We now consider the projection maps
       \begin{equation*}
       \xymatrix{
       \GG(2,H^0(E, \cL_\omega)) && \ar[ll]_<<<<<<<<<{\tau} \KK \ar[rr]^{\nu} && \PP_\cL,
       }
       \end{equation*}     
 and  define $Y:= \nu^{-1}(\PP_\cL-\Sec_{r-2}(E))$.
 
 We make the following observations:
\begin{itemize}
  \item[-]
 $\tau(Y)=$ the set of 2-planes that are spanned by pairs of sections of 
 $\cL_\omega$ that have no common zeros, so is an open subset of $\GG(2,\Hom(\cO_E,\cL_\omega))$, and therefore smooth;
   \item[-] 
$\tau(Y) = X(\cE)/\!\Aut(\cE)$ by  part \cref{item:en2} of the proof of \Cref{th:actloctriv};
  \item[-] 
  $Y$ is the graph of the morphism
        \begin{equation*}
          \tau(Y) = X(\cE)/\Aut(\cE)
          \xrightarrow{\quad \widetilde{\Psi}\quad}
          \PP_\cL-\Sec_{r-2}(E)
        \end{equation*}
        induced by the morphism $\Psi_{\cE}\colon X(\cE)\to\PP_{\cL}$ defined in \cref{cor.defn.Psi.cE}; we showed there that 
        $\Psi_{\cE}$ is constant on $\Aut(\cE)$-orbits so it induces such a morphism;
  \item[-] 
 $\tau |_Y: Y \to \tau(Y)$ is the  natural  identification of the graph of 
        $\widetilde{\Psi}$ with its domain. In particular, $\tau |_Y$ is an isomorphism.  
\end{itemize}

      By \cref{rmk.defn.psi_E}, \Cref{cor:tgorb}, and  \cref{change.psi.to.phi}, $\widetilde{\Psi}$ is one-to-one with injective differentials. 
      Hence
  \begin{equation*}
   \nu |_Y: Y \, \longrightarrow \,    \PP_\cL-\Sec_{r-2}(E),
\end{equation*}     
      which is the composition of $\tau|_{Y}$ and $\widetilde{\Psi}$, is also one-to-one with injective differentials, and hence 
      an isomorphism onto its image by \cite[Cor.~14.10]{H92}. 
      Since $Y$ is smooth, so is $\nu(Y)=L(\cE)$. 
\end{proof}

\begin{corollary}\label{cor:mor-smff}
If $L(\cE) \ne \varnothing$,
then the quotient morphism $X(\cE) \to L(\cE)$ is smooth and faithfully flat.
\end{corollary}
\begin{proof}
(This is a consequence of \Cref{th:smth} and its proof.)
Both $X(\cE)$ and $L(\cE)$ are smooth,
and in the course of the proof we showed that the differential of the 
quotient morphism induces surjections of tangent spaces. That the quotient morphism is smooth then follows from the criterion of \cite[Th\'eor\`eme 17.11.1]{ega44}.  
  Smoothness then in turn implies flatness \cite[Th\'eor\`eme 17.5.1]{ega44}, whence faithful flatness, the latter being nothing but flatness plus surjectivity \cite[Chapitre 0, \S 6.7.8]{ega1}.
\end{proof}

%%%%%%%%%%%%%%%%%%%%%%%%%%%%%%%%%%%%%%%%%%%%%%%%%%%%%%%%%%%%%%%%%%%%%%%%\
\subsection{The singular locus of $\Sec_{r,\omega}(E)-\Sec_{r-1}(E)$ when $n=2r$ and $\omega \in \Omega$}
%%%%%%%%%%%%%%%%%%%%%%%%%%%%%%%%%%%%%%%%%%%%%%%%%%%%%%%%%%%%%%%%%%%%%%%%\

We will show that the singular locus is  $  L(\cL_{\omega}\oplus \cL_{\omega})$. 

With that goal, fix $\xi \in L(\cL_\omega \oplus \cL_\omega)$ and a representative $\xymatrix{0 \ar[r] & \cO_E \ar[r]^-{f} & \cL_{\omega}\oplus \cL_{\omega} \ar[r]^-{\pi} & \cL \ar[r] & 0}$ for $\xi$. As observed in part \Cref{item:en2} of the proof of \Cref{th:actloctriv}, 
$f=(s_{1},s_{2})$ where $s_{1},s_{2}\in H^{0}(E,\cL_{\omega})$ are sections of $\cL_\omega$  such that $(s_1)_0 \cap (s_2)_0=\varnothing$.  
We fix such a pair $(s_1,s_2)$. 

\cref{prop.D.bar.bijection} gave a description of the linear system 
\begin{equation*}
\fd(\xi) \; :=\; \{D \in E^{[r]}_\omega \; | \; \xi \in \overline{D}\}
\end{equation*}
that depended on a choice of $D' \in E^{[r]}_\omega$. 
The next lemma describes $\fd(\xi)$ in terms of the map $f:\cO_E \to \cL_\omega \oplus \cL_\omega$ that appears in $\xi$.
The first step is to note that $\fd(\xi)$ is a pencil.

\begin{lemma}
\label{lem.fd.xi}
Suppose $n=2r$ and let $\omega \in \Omega$. 
If $\xi \in L(\cL_\omega \oplus \cL_\omega)$,
 then 
\begin{enumerate}
  \item\label{item.fd.xi.pencil} 
$\fd(\xi)$ is a pencil, i.e., $\{D \in E^{[r]}_\omega \; | \; \xi \in \overline{D}\} \cong \PP^1$,  
\item\label{item.fd.xi.p}
$\fd(\xi) \cong \PP(\CC s_1+\CC s_2) \subseteq \PP H^0(E,\cL_\omega)$;
  \item\label{item.fd.xi.disj} 
if $D_1,D_2 \in \fd(\xi)$ and $D_1 \ne D_2$, then $D_1 \cap D_2=\varnothing$.
\end{enumerate}
\end{lemma}
\begin{proof}
\cref{item.fd.xi.pencil}
Fix a divisor $D'\in E_{\omega}^{[r]}=|D'|$. 
 By \cref{prop.D.bar.bijection}, 
\begin{equation*}
	\{D \in |D'| \; | \; \xi \in \overline{D}\} \; \cong \; \PP \Hom(\cI_{D'}\cL,\cL_{\omega}\oplus \cL_{\omega}) \; \cong \; \PP\Hom(\cL_{\omega},\cL_{\omega}\oplus \cL_{\omega})\;\cong\;\PP^{1}.
\end{equation*}

\cref{item.fd.xi.p}
We will show that $\fd(\xi)$ ``is'' the projective line $\PP(\CC s_1+\CC s_2)$;
i.e.,  $\fd(\xi) = \{(s)_0 \; | \; 0 \ne s \in \CC s_1+\CC s_2\}$.

Under the isomorphisms in \cref{item.fd.xi.pencil}, a point $\mu=(\mu_{1},\mu_{2})\in\PP^{1}$, corresponds to the unique effective divisor $D$  with the property that 
\begin{equation*}
	\xymatrix@C=12mm{
	\im\bigl(	\cL_{\omega}\ar[r]^-{(\mu_{1},\mu_{2})} & \cL_{\omega}\oplus \cL_{\omega}\ar[r]^-{\pi} & \cL\bigr) \; =\;  \cI_{D}\cL
	}
\end{equation*}
or equivalently, the cokernel of the composition is isomorphic to $\cO_{D}$. Since the cokernel of the map $(\mu_{1},\mu_{2}):\cL_{\omega}\to\cL_{\omega}\oplus\cL_{\omega}$ is the map $(\mu_{2},-\mu_{1}):\cL_{\omega}\oplus\cL_{\omega}\to\cL_{\omega}$, 
the universal property of cokernels implies there is a commutative diagram
\begin{equation*}
	\xymatrix@C=12mm{
		& & 0\ar[d] & 0\ar[d] & \\
		& & \cO_{E}\ar@{=}[r]\ar[d]_-{(s_{1},s_{2})} & \cO_{E}\ar[d]^-{\mu_{2}s_{1}-\mu_{1}s_{2}} & \\
		0\ar[r] & \cL_{\omega}\ar[r]_-{(\mu_{1},\mu_{2})}\ar[d]_-{\cong} & \cL_{\omega}\oplus\cL_{\omega}\ar[r]_-{(\mu_{2},-\mu_{1})}\ar[d]_-{\pi} & \cL_{\omega}\ar[r]\ar[d] & 0 \\
		0\ar[r] & \cI_{D}\cL\ar[r]_-{i_{D}} & \cL\ar[r]\ar[d] & \cO_{D}\ar[r]\ar[d] & 0. \\
		& & 0 & 0 &
	}
\end{equation*}
From the vertical exact sequence on the right we see that $D=(\mu_{2}s_{1}-\mu_{1}s_{2})_0$.

\cref{item.fd.xi.disj}
Let $D_1$ and $D_2$ be distinct points in $\fd(\xi)$ and let $s_1', s_2'$ be points in $\CC s_1+\CC s_2$ such that $(s_1')_0=D_1$ and 
 $(s_2')_0=D_2$. Since $D_1 \ne D_2$, $\CC s_1' \ne \CC s_2'$. Hence $\CC s_1+\CC s_2=\CC s_1'+\CC s_2'$. Since $s_1$ and $s_2$ have no 
 common zero neither do $s_1'$ and $s_2'$; i.e.,  $D_1 \cap D_2=\varnothing$.
\end{proof}

The next result answers the question after  \cref{prop.8.15.GvB-H}.

\begin{proposition}\label{prop.ssnsmth}
  For $n=2r$ and $\omega\in \Omega$, the singular locus of the  locally closed subscheme $\Sec_{r,\omega}(E)-\Sec_{r-1}(E)$ is 
  $  L(\cL_{\omega}\oplus \cL_{\omega})$. 
 The smooth locus of $\Sec_{r,\omega}(E)-\Sec_{r-1}(E)$ is therefore $L(\cE_{\omega})$. 
\end{proposition}
\begin{proof}
  The second claim follows from the first and \Cref{th:nsplitev}\cref{item:disju},
  so we will prove the first claim.

  The argument parallels \cite[p.~18, proof of Theorem, part 1]{copp}.

Let $D\in E_{\omega}^{[r]}$ be such that $\xi\in\overline{D}$. Since $\overline{D}$ is a linear space contained in $\Sec_{r,\omega}(E)$, it is tangent to $\Sec_{r,\omega}(E)-\Sec_{r-1}(E)$ at $\xi$. Therefore, $\overline{D}$ is contained in the embedded tangent space to $\Sec_{r,\omega}(E)-\Sec_{r-1}(E)$ at $\xi$ (see, for example, \cite[p.~17, Notation]{copp} for the definition of the embedded tangent space). In particular, $D$ itself is contained in this embedded tangent space. But there is a $\PP^1$ of such $D$'s by the previous lemma and every point in $E$ belongs to one such divisor (in fact, a unique one by disjointness) 
so $E$ is contained in the embedded tangent space. Since $E$ is not contained in any hyperplane in $\PP_{\cL}$, the embedded tangent space is $\PP_{\cL}$. Hence $\Sec_{r,\omega}(E)-\Sec_{r-1}(E)$ is singular at $\xi$, proving that $L(\cL_{\omega}\oplus \cL_{\omega})$ is contained in the singular locus of $\Sec_{r,\omega}(E)-\Sec_{r-1}(E)$.

The complement of $L(\cL_{\omega}\oplus \cL_{\omega})$ is $L(\cE_{\omega})$, which is smooth by \cref{th:smth}\cref{item:smth}.
 The singular locus of $\Sec_{r,\omega}(E)-\Sec_{r-1}(E)$ is  therefore $L(\cL_{\omega}\oplus \cL_{\omega})$.
\end{proof}

%%%%%%%%%%%%%%%%%%%%%%%%%%%%%%%%%%%%%%%%%%%%%%%%%%%%%%%%%%%%%%%%
\section{The $L(\cE)$'s are the symplectic leaves}
\label{sect.symp.leaves}
%%%%%%%%%%%%%%%%%%%%%%%%%%%%%%%%%%%%%%%%%%%%%%%%%%%%%%%%%%%%%%%%

The symplectic leaves for $(\PP_\cL,\Pi)$ are the members of the unique partition of $\PP_\cL$ into 
connected immersed submanifolds $Y$ such that $T_\xi Y=$ the image of $\Pi_\xi: T_\xi^*\PP_\cL \longrightarrow T_\xi \PP_\cL$ for every $\xi \in Y$.   
Since the $L(\cE)$'s are smooth and (Zariski-)locally closed, they are connected immersed submanifolds of $\PP_\cL$. 

We begin with descriptions of $T_\xi \PP_\cL$ and $T_\xi^* \PP_\cL$ that are adapted to our earlier analysis and notation.
Although $\Pi_\xi:T_\xi^{*} \PP_\cL \to T_\xi \PP_\cL$  is described in \cite[\S2.1, p.~3]{HP3} we need a slightly different description.

\subsection{The tangent and cotangent spaces $T_\xi \PP_\cL$ and $T_\xi^* \PP_\cL$ at a point $\xi \in \PP_\cL$}
Let
\begin{equation*}
\GG \; :=\; \GG(n-1,\Hom(\cO_E,\cL))
\end{equation*}
be the Grassmannian of codimension-one subspaces of $\Hom(\cO_E,\cL)$. 
We will consider the isomorphism 
\begin{equation*}
F:\PP_\cL\to \GG, \qquad  F(\xi)=\xi^\perp,
\end{equation*}
given by
\begin{align}
\PP_{\cL}   &   \; = \;   \GG(1,\Ext^1(\cL,\cO_E)) \;   = \; \GG(1,\Hom(\cO_E,\cL)^*)  \; \stackrel{\sim}{\longrightarrow}   \;   \GG 
\label{eq:identifications}
\\
&     \phantom{xxxxxxxxxxx} \CC\xi   \phantom{xixxx}  =  \phantom{xxx} \im(\pi_*)^\perp   \phantom{xxixxx} \xmapsto{\quad \quad}   \phantom{i} \im(\pi_*)=\xi^\perp,
\notag
  \end{align}
 and its differential   $dF_\xi:T_\xi \PP_\cL \to T_{\xi^\perp} \GG$.

\begin{lemma}\label{lem.T.xi.PL}
Let $\xi \in \PP_\cL$ be the isomorphism class of a non-split extension 
\begin{equation*} 
    \xymatrix{
    0 \ar[r] &  \cO_E \ar[r]^f & \cE  \ar[r]^\pi \ar[r] & \cL  \ar[r] &  0.
    }
  \end{equation*}
The tangent spaces to $\PP_\cL$ at $\xi$  and to $\GG$ at $\xi^\perp$ are  
   \begin{align}
    T_\xi \PP_\cL  & \;=\; \Hom\big( \CC \xi, \Ext^1(\cL,\cO_E)/\CC \xi \big)   \qquad \text{and} 
     \label{eq:T.xi.0}
   \\
T_{\xi^\perp}\GG    &  \;  = \;     \Hom\big(\xi^\perp,\Hom(\cO_E,\cL)/\xi^\perp\big) 
    \label{eq:T.xi.1}
    \\
  & \; \cong \;  \big\{ \a: \Hom(\cO_E,\cE) \to  \Hom(\cO_E,\cL)/{\xi^\perp} \; | \; \a(f)=0\big\}.
  \label{eq:T.xi.2}
  \end{align}
 \end{lemma}
 \begin{proof}  
The equalities in \cref{eq:T.xi.0}  and  \Cref{eq:T.xi.1}  follow from the proof of \cite[Thm.~3.5, p.~96]{3264} (see, for example, the sentence just after its proof).

To establish the isomorphism in \Cref{eq:T.xi.2}, we consider the exact sequence
   \begin{equation*}
\label{eq:xi.perp.2}
\xymatrix{
0 \ar[r] &  \Hom(\cO_E,\cO_E) \ar[r]^-{f_*}  & \Hom(\cO_E,\cE) \ar[r]^-{\pi_*}  &  \Hom(\cO_E,\cL).
}
  \end{equation*}
  where $f_*(a)=fa$ and $\pi_*(b)=\pi b$. Since  
\begin{equation*}
 \im(\pi_*)   \; \cong \;  \frac{\Hom(\cO_E,\cE)}{\ker(\pi_*)}   \; = \;   \frac{\Hom(\cO_E,\cE)}{\im(f_*)} 
 \; = \;   \frac{\Hom(\cO_E,\cE)}{\CC f} \, ,
\end{equation*} 
we have
\begin{equation*}
  \Hom\big(\im(\pi_*),\Hom(\cO_E,\cL)/\im(\pi_*)\big) \; \cong \;   \Hom\bigg( \frac{\Hom(\cO_E,\cE)}{\CC f} \,  , \,  \frac{\Hom(\cO_E,\cL)}{\im(\pi_*)}  \bigg),
  \end{equation*} 
thus giving the isomorphism in \cref{eq:T.xi.2}.
\end{proof}

\begin{lemma}
  \label{lem.T*.xi.PL}
Suppose  $\xi \in \PP_\cL$ is isomorphic to $  \xymatrix{     0 \ar[r] &  \cO_E \ar[r]^f & \cE  \ar[r]^\pi \ar[r] & \cL  \ar[r] &  0.     }$
  \begin{enumerate}
  \item\label{item:lem.T*.xi.PL-1} The cotangent space to $\PP_\cL$ at $\xi$ is $T^*_\xi \PP_\cL = \; \xi^\perp \subseteq \Hom(\cO_E,\cL)$.
  \item\label{item:lem.T*.xi.PL-2} The cotangent space to $\GG$ at $\xi^\perp = \im(\pi_*)$ is
    \begin{align}
      T^*_ {\xi^\perp} \GG &   \;=\;   \Hom\big(\Hom(\cO_E,\cL)/\xi^\perp, \xi^\perp \big)  
                             \notag
      \\
                           &   \; \cong \;   \{ \a:\Hom(\cO_E,\cL) \to \xi^\perp \; | \; \a(\xi^\perp)=0\}.  
                             \label{eq:T*.xi.1}
    \end{align}
  \item\label{item:lem.T*.xi.PL-3} Let $\d:\Hom(\cO_E,\cL) \to \Ext^1(\cO_E,\cO_E)$ be the connecting homomorphism, and fix $s \in (t\d)^{-1}(1)$.  The map $T^*_ {\xi^\perp} \GG \to T_\xi ^* \PP_\cL $, $\a \mapsto \a(s)$, that sends $\a$ in \cref{eq:T*.xi.1} to $\a(s)$ is an isomorphism.
  \end{enumerate} 
\end{lemma}
\begin{proof}
{\bf \Cref{item:lem.T*.xi.PL-1}}  The cotangent space $T^*_xM$ at a point $x$ on a manifold $M$ is the dual to $T_xM$, so
    \begin{align*}
      T_\xi^*\PP_\cL & \;=\;  \Hom \! \big( \CC \xi, \Ext^1(\cL,\cO_E)/\CC \xi \big)^* 
      \\
                     &\; = \;  \Hom \! \big(\!\Ext^1(\cL,\cO_E)/\CC \xi,  \CC \xi \big) 
      \\
                     & \; = \;  \Hom\!\big(\!\Hom(\cO_E,\cL)^*/\CC \xi,  \CC \xi \big). 
    \end{align*}
    The last of these spaces is the kernel of the map $\Hom\!\big(\!\Hom(\cO_E,\cL)^*, \CC \xi \big)  \to \Hom(\CC\xi,\CC \xi)$, $\rho \mapsto \rho|_{\CC \xi}$, 
    which is $\xi^\perp$.

{\bf \Cref{item:lem.T*.xi.PL-2}} Since  $T_{\xi^\perp} \GG=\Hom\!\big(\xi^\perp,\Hom(\cO_E,\cL)/\xi^\perp \big)$,
    \begin{align*}
      T_{\xi^\perp}^* \GG & \;=\; \Hom\!\big(\xi^\perp,\Hom(\cO_E,\cL)/\xi^\perp \big)^* 
      \\
                          &\; = \;  \Hom\!\big(\Hom(\cO_E,\cL)/\xi^\perp, \xi^\perp \big),
    \end{align*}
    as claimed.

{\bf \Cref{item:lem.T*.xi.PL-3}}  Since $\dim(\Hom(\cO_E,\cL)/\xi^\perp)=1$, the last of these spaces is isomorphic  to $\xi^\perp$; one such isomorphism is implemented by fixing a
    point $s \in (t\d)^{-1}(1) \subseteq \Hom(\cO_E,\cL)$ and sending each $\a:\Hom(\cO_E,\cL) \to \xi^\perp$ that vanishes on $\xi^\perp$ 
    to $\a(s) \in \xi^\perp$.
\end{proof}

Fix $\xi \in \PP_\cL$ and a representative 
 \begin{equation} \label{eq:dpsi}
    \xi: \quad 
    \xymatrix{
    0 \ar[r] &  \cO_E \ar[r]^f & \cE  \ar[r]^\pi \ar[r] & \cL  \ar[r] &  0
    }
  \end{equation}
for it.  Thus $f$ is in $X(\cE)$ and $\pi \in {\rm Epi}(\cE,\cL)$ is the corresponding epimorphism.

We will consider the differentials of the vertical maps in the commutative diagram 
\begin{equation}
\label{change.psi.to.phi}
\xymatrix{
X(\cE) \ar[rr]^\tau \ar[d]_\Psi  && {\rm Epi}(\cE,\cL) \ar[dll] | -{\Psi\tau^{-1}}  \ar[d]^\Phi
\\
 \PP_\cL  \ar[rr]_F && \GG,             
}
\end{equation}  
where $\GG=\GG(n-1,\Hom(\cO_E,\cL))$, $\Psi:=\Psi_\cE$ (defined in \cref{cor.defn.Psi.cE}), 
$\tau$ is the restriction of the linear isomorphism $\theta \nu_*:\Hom(\cO_E,\cE) \to \Hom(\cE,\cL)$ in \cref{re:kercoker}, 
and $\Phi$ is the map $\pi \mapsto \im(\pi_*)=\xi^\perp$. By definition, $\Psi(f)=\xi$ so the image of $\Psi$ is $L(\cE)$.
 
For brevity, we will sometimes write  $X$ rather than $X(\cE)$.

The tangent spaces at points in the spaces in \cref{change.psi.to.phi} are
\begin{align*}
 T_fX(\cE)  & \; = \; \Hom(\cO_E,\cE)
 \\
 T_\xi \PP_\cL  & \;=\; \Hom\big( \CC \xi, \Ext^1(\cL,\cO_E)/\CC \xi \big), 
 \\   
 T_{\xi^\perp} \GG      &  \; =  \;     \Hom\big(\xi^\perp,\Hom(\cO_E,\cL)/\xi^\perp\big),  \qquad \text{and} 
\\
T_\pi {\rm Epi}(\cE,\cL) & \; = \; \Hom(\cE,\cL),
\end{align*}
the first of these because $X(\cE)$ is a dense open subset of $\Hom(\cO_E,\cE)$, and 
the last of these because ${\rm Epi}(\cE,\cL)$ is a non-empty Zariski-open subset of  $\Hom(\cE,\cL)$.

 Our next result, \cref{le:dPhi}, gives an explicit description of the differential $d\Phi_\pi$. 
 Its proof is expressed in the language of {\it infinitesimal deformations}. See \cite[Exers.~II.2.8, II.8.6, II.8.7, III.4.10]{hrt} for some general background, 
and \cite[pp.~18--20]{copp} for an application of these techniques very close in spirit to what we need. 
For the convenience of some readers, and these authors, we review some details about tangent spaces to Grassmannians in 
\Cref{ssect.T_x.GG}.

The case of interest concerns a point $\xi \in \PP_\cL$  that is represented by \Cref{eq:dpsi}.
The differentials of the morphisms in \cref{change.psi.to.phi}  give linear maps
\begin{equation*}
\label{diffl.change.psi.to.phi}
\xymatrix{
\Hom(\cO_E,\cE) \ar[rr]^{d\tau_f} \ar[d]_{d\Psi_f}  && \Hom(\cE,\cL)  \ar[d]^{d\Phi_\pi}
\\
\Hom\big( \CC \xi, \Ext^1(\cL,\cO_E)/\CC \xi \big)  \ar[rr]_{dF_\xi}  && \Hom\big(\xi^\perp,\Hom(\cO_E,\cL)/\xi^\perp\big).    
}         
\end{equation*}  
We will also use the following fact: since $\xi^\perp=\im(\pi_*)$, there is an  exact sequence 
\begin{equation}
\label{eq:ses.xi.perp}
\xymatrix{
0 \ar[r] & \Hom(\cO_E,\cO_E) \ar[r]^{f_*} & \Hom(\cO_E,\cE) \ar[r]^>>>>>{\pi_*} & \xi^\perp    \ar[r] & 0
}
\end{equation}
where $f_*(a)=fa$. 
 
\begin{lemma}
\label{le:dPhi}
Let $\varphi \in \Hom(\cE,\cL)=T_\pi {\rm Epi}(\cE,\cL)$.
Then $( d\Phi_\pi)(\varphi)$ belongs to
\begin{align}
T_{\xi^\perp} \GG   &\;=\;   \Hom\!\big(\xi^\perp,\Hom(\cO_E,\cL)/\xi^\perp\big) 
\notag
\\
& \; \cong \; \{ \a : \Hom(\cO_E,\cE) \to \Hom(\cO_E,\cL)/\xi^\perp \; | \; \a(f)=0\}
\label{isom.T.xi.perp.GG}
\end{align}
 has two interpretations:
 \begin{enumerate}
 \item\label{item:dPhi.1}  as a map $\xi^\perp \to \Hom(\cO_E,\cL)/\xi^\perp$,
   \begin{equation*}
  d\Phi_\pi(\varphi)(\pi s) = \varphi s +\xi^\perp
   \end{equation*}
   for all $\pi s \in \xi^\perp$;
   equivalently,  if we first fix a   linear map  $\mu:  \xi^\perp   \to \Hom(\cO_E,\cE)$  such that $\pi_* \mu=\id_{\xi^\perp}$, then $ d\Phi_\pi(\varphi)(x) = \varphi \circ \mu(x) + \xi^\perp$  for all $x \in \xi^\perp$;
\item\label{item:dPhi.2} 
 as a map $\Hom(\cO_E,\cE) \to \Hom(\cO_E,\cL)/\xi^\perp$ that vanishes at $f$, 
  \begin{equation*}
  d\Phi_\pi(\varphi)(s) \;=\;   \varphi s+ \xi^\perp;
 \end{equation*}
 i.e., $d\Phi_\pi(\varphi)$ is the composition
 \begin{equation}
 \label{eq:d.psi.f.varphi}
 \xymatrix{ 
 \Hom(\cO_E,\cE) \ar[r]^>>>>>{\varphi_*} &  \Hom(\cO_E,\cL) \ar[r]^>>>>>>\gamma &   \Hom(\cO_E,\cL)/\xi^\perp
 }
 \end{equation}
  where $\gamma(v) = v+\xi^\perp$ and $\varphi_*(a)=\varphi a$.
\end{enumerate}
 \end{lemma}
\begin{proof}
We use the ideas and notation in \Cref{ssect.T_x.GG}.

When we write $T_\pi{\rm Epi}(\cE,\cL) = \Hom(\cE,\cL)$ we are identifying $\varphi$, which belongs to $\Hom(\cE,\cL)$,
with the homomorphism $\pi+\varepsilon \varphi: \cE \otimes \CC[\varepsilon] \to \cL \otimes \CC[\varepsilon]$ 
over the ring $\CC[\varepsilon]$ of dual numbers. 
By  \Cref{ssect.T_x.GG}, the $\varepsilon$-thickening of the map $\Phi$ in \cref{change.psi.to.phi}  is the map  $   \widetilde{\Phi}$ given by the formula
  \begin{align*}
    \widetilde{\Phi}(\pi+\varepsilon\varphi )  &  \; :=\;     \im (\pi+\varepsilon \varphi)_* 
    \\
    & \phantom{:} \;=\;   \{(\pi+\varepsilon\varphi)\circ (s+\varepsilon s') = \pi s+\varepsilon(\pi s'+\varphi s) \; | \;  s,s'\in \Hom(\cO_E,\cE)\}.
  \end{align*}
 The image  $\im (\pi+\varepsilon \varphi)_* $ is a $\CC[\varepsilon]$-point of $\GG$.
 The map $\widetilde{\Phi}$ specializes back to $\Phi$.

{\bf \Cref{item:dPhi.1}} 
 The discussion in \Cref{ssect.T_x.GG} 
 shows that, regarded as a map  $\xi^\perp \to \Hom(\cO_E,\cL)/\xi^\perp$, the tangent vector 
  $d\Phi_{\pi}(\varphi)$ is the map $\pi s \mapsto y+\xi^\perp$ where $y \in  \Hom(\cO_E,\cL)$ is the unique element modulo $\xi^\perp$ such that 
  \begin{equation*}
    \pi s+\varepsilon y \in \im(\pi+\varepsilon \varphi)_*.
  \end{equation*}  
  Since $\pi s+\varepsilon \varphi s \in \im(\pi+\varepsilon \varphi)_*$, we may take $y= \varphi s$; thus $d\Phi_\pi(\varphi)(\pi s) = \varphi s +\xi^\perp$.
  This gives the first description of  $d\Phi_\pi(\varphi)$ in \Cref{item:dPhi.1}. 
  
  To verify the second description of $d\Phi_\pi(\varphi)$ in \Cref{item:dPhi.1}, use the fact that every $x$ in $\xi^\perp$ equals $\pi s$ for some $s \in \Hom(\cO_E,\cE)$, then observe that 
  $\pi_*(s)= \pi s = \pi_*( \mu(\pi s))$ whence
$s-\mu(x)=s - \mu(\pi s) \in \ker(\pi_*)= \CC f$; then, since $\varphi f \in \xi^\perp$, by \cref{lem.varphi.f}, $\varphi s - \varphi \circ \mu(x) \in \xi^\perp$; i.e., $\varphi s +\xi^\perp = \varphi \circ \mu(x) +\xi^\perp$.\footnote{It is implicit in this paragraph that $\varphi \circ \mu(x) + \xi^\perp$ does not depend on the choice of $\mu$. 
One can also see this  directly by the following argument: 
if $\mu_1,\mu_2 : \xi^\perp \to \Hom(\cO_E,\cE)$ are such that $\pi_* \mu_1= \pi_* \mu_2=\id_{\xi^\perp}$, then $(\mu_1-\mu_2)(x) \in \ker(\pi_*)=\CC f$ for all 
$x \in \xi^\perp$;  hence $\varphi \circ (\mu_1-\mu_2)(x) \in \CC \varphi  f \subseteq \im(\pi_*) \subseteq  \xi^\perp$; thus 
$\varphi \circ \mu_1(x) + \xi^\perp = \varphi \circ \mu_2(x) + \xi^\perp$.}

{\bf \Cref{item:dPhi.2}} 
As an element belonging to the set in \cref{isom.T.xi.perp.GG} the tangent vector $d\Phi_\pi(\varphi)$ is the composition 
\begin{equation*}
\Hom(\cO_E,\cE) \to \Hom(\cO_E,\cE)/\CC f \stackrel{\sim}{\longrightarrow} \xi^\perp \to \Hom(\cO_E,\cL)/\xi^\perp, \quad s \mapsto s +\CC f \mapsto \pi s \mapsto \varphi s +\xi^\perp;
\end{equation*}   
i.e., $s \mapsto \varphi s +\xi^\perp =  \gamma(\varphi s) =( \gamma \circ \varphi_*)(s)$, which is the map in \Cref{eq:d.psi.f.varphi}. The proof is complete.\footnote{Since
the map displayed on the previous line factors through $\Hom(\cO_E,\cE)/\CC f$ it sends $f$ to 0, i.e., \
$\varphi f \in \xi^\perp$, because $\pi f=0$. It follows that $\varphi f=\pi s'$ for some $s' \in \Hom(\cO_E,\cE)$.}
\end{proof}

\begin{proposition}\label{pr:whichker}
\label{lem.xi.pi.0}
If $\varphi \in \Hom(\cE,\cL)$, then 
\begin{equation*}
\xi \cdot \varphi =0  \; \Longleftrightarrow \;  d\Phi_\pi(\varphi) =0 \;  \Longleftrightarrow \;  \varphi \in \pi \End(\cE).
\end{equation*}
Equivalently, 
\begin{equation} 
\label{eq:pi.End.cE}
\{ \varphi \in \Hom(\cE,\cL)  \; | \; \xi \cdot \varphi =0\}  \;=\; \pi \End(\cE) \;=\;  \ker(d\Phi_\pi).
\end{equation} 

\end{proposition}
\begin{proof}
Applying the functor $\Hom(\cE,-)$ to $\xi$ yields the exact sequence
\begin{equation}
\label{xi.dot.pi}
\xymatrix{
\cdots \ar[r] &
 \Hom(\cE,\cE)  \ar[r]^-{\pi \circ -} &    \Hom(\cE,\cL)  \ar[r]^-\delta & \Ext^1(\cE,\cO_E) \ar[r]& \Ext^1(\cE,\cE) \ar[r] & \cdots.
}
\end{equation}
The first equality in \cref{eq:pi.End.cE}  holds because  $\ker(\d)=\im(\pi \circ -)$: 
the left-most term in \cref{eq:pi.End.cE} is the kernel of $\d$ because $\d(\varphi)= \xi \cdot \varphi$, and 
the middle term in \cref{eq:pi.End.cE} is the image of $(\pi \circ-): \Hom(\cE,\cE) \to \Hom(\cE,\cL)$.

Suppose $d\Phi_\pi(\varphi)=0$. 
Then  $\varphi s  \in \xi^\perp = \im(\pi_*)$ for all $s \in \Hom(\cO_E,\cE)$; i.e., $\varphi s= \pi s'$ 
for some $s' \in \Hom(\cO_E,\cE)$. 
But $\xi \cdot \pi=0$ by  the first equality in \cref{eq:pi.End.cE}, so $(\xi\cdot \varphi) \cdot s =0$ for all $s \in  \Hom(\cO_E,\cE)$. 
In other words, $(\xi\cdot \varphi, s)$ is in the kernel of the pairing 
\begin{equation*}
    \Ext^1(\cE,\cO_E)\times \Hom(\cO_E,\cE) \, \longrightarrow \, \Ext^1(\cO_E,\cO_E) \; \cong \; \CC
  \end{equation*}
  for all $s \in  \Hom(\cO_E,\cE)$; but this pairing is non-degenerate so $\xi\cdot \varphi=0$. Hence $\varphi \in   \pi \End(\cE)$.

Suppose $\varphi \in \pi \End(\cE)$. Then  $\varphi \Hom(\cO_E,\cE)  \subseteq  \pi \Hom(\cO_E,\cE)$. Thus,
  given $s \in \Hom(\cO_E,\cE)$ there is an $s' \in \Hom(\cO_E,\cE)$ such that $\varphi   s= \pi s' \in \xi^\perp$.
  Hence $\varphi s + \xi^\perp=0$; i.e., $d\Phi_\pi(\varphi)=0$.   
\end{proof}

Part \cref{item.cor.tgorb.ker} of the next result says that the differential $d\Phi$ vanishes precisely on the tangent spaces to the $\Aut(\cE)$-orbits in ${\rm Epi}(\cE,\cL)$.
Since $\Aut(\cE)$ acts freely on ${\rm Epi}(\cE,\cL)$, $d\Phi$ has constant rank.

\begin{corollary}\label{cor:tgorb}  \label{cor:ctrk}
If $L(\cE) \ne \varnothing$, and $\pi \in {\rm Epi}(\cE,\cL)$, then
\begin{enumerate}
  \item\label{item.cor.tgorb.ker} 
   the kernel of $d\Phi_\pi: T_\pi {\rm Epi}(\cE,\cL) \to T_{\xi^\perp} \GG$ is $T_\pi (\Aut(\cE) \cdot \pi)$;
     \item\label{item.cor.tgorb.im} 
the image of $d\Phi_\pi$ equals $\{ \xi\cdot \varphi \; | \; \varphi \in \Hom(\cE,\cL)\} \subseteq  \Ext^1(\cE,\cO_E)$;
\end{enumerate}
  The differential $d\Phi$ has constant rank. 
\end{corollary}
\begin{proof}
Let $G=\Aut(\cE)$. 

\cref{item.cor.tgorb.ker}
  Let $a:G \to G\! \cdot \! \pi$ be the map $g \mapsto g \cdot \pi$. Since $a$ is an isomorphism by the principality of the action of $G$ (\cref{th:actloctriv}\cref{item:th:actloctriv.3}) and $T_1G$ naturally identifies with $\End(\cE)$, the differential
  \begin{equation*}
    da_1:T_1G
    \xrightarrow{\quad}
    T_{\pi}(G.\pi)
  \end{equation*}
  at the identity is an isomorphism $\End(\cE) \to T_\pi(G .\pi) =\pi \End(\cE) \subseteq \Hom(\cE,\cL)= T_\pi {\rm Epi}(\cE,\cL)$.  By \Cref{pr:whichker}, $\ker(d\Phi_\pi) = \pi \End(\cE) =T_\pi (G.\pi)$.

\cref{item.cor.tgorb.im}
  Fix a non-zero $\varphi \in \Hom(\cE,\cL)$. Let $\varphi_*:\Hom(\cO_E,\cE) \to \Hom(\cO_E,\cL)$ be the map $s \mapsto \varphi s$. 
 The kernel of $d\Phi_\pi(\varphi)=\gamma \circ \varphi_*:\Hom(\cO_E,\cE) \to  \Hom(\cO_E,\cL)/\xi^\perp$, given by $s \mapsto \varphi s + \xi^\perp =\varphi _*(s) + \xi^\perp $, is $\varphi_*^{-1}(\xi^\perp)$. The kernel of  $\xi_{*}\varphi_{*}:\Hom(\cO_E,\cE) \to \Ext^1(\cO_E,\cO_E)$, given by $s \mapsto \xi\cdot \varphi s =  \xi\cdot \varphi_*( s)$,  is also $\varphi_*^{-1}(\xi^\perp)$. 
 Since $\xi^\perp$ has codimension one in $\Hom(\cO_E,\cL)$, $\varphi_*^{-1}(\xi^\perp)$ has codimension one in $\Hom(\cO_E,\cE)$. 
There is therefore an isomorphism $\a:  \Hom(\cO_E,\cL)/\xi^\perp  \to  \Ext^1(\cO_E,\cO_E)$ such that $\a \circ d\Phi_\pi(\varphi) =(\xi\cdot \varphi)_*$.

Finally, all orbits $\Aut(\cE) \! \cdot \! \pi$ are isomorphic to $\Aut(\cE)$ 
so they, and their tangent spaces, have the same dimension. 
\end{proof}

Consider a Lie group $G$ acting smoothly on a smooth manifold $X$.  If the action of $G$ on $X$ is free and proper, then $X/G$ is a smooth manifold and the quotient $\psi:X \to X/G$ is a smooth submersion, by \cite[Thm.~21.10]{lee2013introduction}, and $G.x$ is an embedded submanifold of $X$ \cite[Prop.~21.7]{lee2013introduction}. It now follows from \cite[Prop.~5.8]{lee2013introduction} that
 \begin{equation*}
  \ker\big(d\psi_x:T_xX \to T_{\psi(x)}(X/G)\big) \;=\; T_x(G.x).
   \end{equation*}
 Since  $\dim(X/G)=\dim X-\dim(G.x)$ it follows that $d\psi_x$ is surjective and hence, for all $x \in X$, there are natural isomorphisms
 \begin{equation*}
 T_{\psi(x)}(X/G) \; \cong \; T_xX/T_x(G.x).
   \end{equation*}

  \begin{theorem}
  \label{th:leleaves}
  The symplectic leaves for $(\PP_\cL,\Pi)$ are the homological leaves $L(\cE)$.
  \end{theorem}
  \begin{proof}
  Let $G=\Aut(\cE)$.
By \cref{cor:tgorb}, 
\begin{equation*}
\ker(d\Phi_\pi: T_\pi  {\rm Epi}(\cE,\cL) \to T_{\xi^\perp} \GG) \; = \; T_\pi (G.\pi)
\end{equation*}
so, by the paragraph before the statement of this theorem and \cref{le:dPhi}\Cref{item:dPhi.2},
 \begin{align*}
 T_{\Phi(\pi)} \bigg(\frac{ {\rm Epi}(\cE,\cL)}{\Aut(\cE)} \bigg) & \; \cong \; \frac{T_\pi {\rm Epi}(\cE,\cL)}{T_\pi(G.\pi)} 
 \\
 & \;=\; \text{the image of } d\Phi_\pi 
  \\
 & \;=\; \{  d\Phi_\pi(\varphi) = \gamma \circ \varphi_*:\Hom(\cO_E,\cE) \to  \Hom(\cO_E,\cL)/\xi^\perp \; | \; \varphi \in \Hom(\cE,\cL)\}.
 \end{align*}
 Hence, 
 to prove that the $L(\cE)$'s are the symplectic leaves we must show that, under the identification via $dF_{\xi}:T_{\xi}\PP_{\cL}\to T_{\xi^{\perp}}(\GG)$ induced by the isomorphism $F:\PP_{\cL}\to\GG$, that
  \begin{equation*}
  \text{the image of  $\Pi_\xi$}\;=\;\text{the image of $d\Phi_\pi$}. 
  \end{equation*}
 
We will show that the map $\im(\Pi_\xi) \to \im (d\Phi_\pi)$, $\eta \mapsto \eta \cdot \pi$ is an isomorphism. 
 
 First, it follows from the exact sequence
 \begin{equation*}
\xymatrix{
0 \ar[r] & \Hom(\cL,\cL) \ar[r]^-{-\circ \pi}  & \Hom(\cE,\cL) \ar[r]^-{-\circ f} & \Hom(\cO_E,\cL) \ar[r]^-{\cdot \xi}  & \Ext^1(\cL,\cL)  \ar[r] & 0
}
\end{equation*}
that  $\xi^\perp = \{\varphi f \; | \; \varphi \in \Hom(\cE,\cL)\}$.

The proof of \cite[Prop.~2.3]{HP3} says that $\Pi_\xi$, which is their $\Pi_\phi$, fits into the commutative diagram
\begin{equation}
\label{eq:HP.diag}
\xymatrix{
\ar[d]_{\delta}  \Hom(\cE,\cL) \ar[rr]^{-\circ f} && \xi^\perp \ar[d]^{\Pi_\xi} \, \ar@{^{(}->}[r] & \Hom(\cO_E,\cL) 
\\
\Ext^1(\cE,\cO_E)  &&   \ar[ll] \Ext^1(\cL,\cO_E)/\CC \xi   
}
\end{equation}
where the bottom arrow is $\eta + \CC\xi \mapsto \eta \cdot \pi$ and $\d$ is the connecting homomorphism in \cref{xi.dot.pi}, 
namely $\delta(\varphi)=\xi \cdot \varphi$. 
Thus $\xi\cdot \varphi = \Pi_\xi(\varphi f) \cdot \pi $.   
Hence, as a map to $\im (d\Phi_\pi)=\{\xi \cdot \varphi \; | \; \varphi \in \Hom(\cE,\cL)\}$, $\eta + \CC\xi \mapsto \eta \cdot \pi$ is surjective. 
It follows from the exact sequence
\begin{equation*}
0 \to \Hom(\cO_E,\cO_E) \to \Ext^1(\cL,\cO_E) \stackrel{\cdot  \pi}{\longrightarrow} \Ext^1(\cE,\cO_E) \to \Ext^1(\cO_E,\cO_E)  \to 0
\end{equation*}
that $\eta + \CC\xi \mapsto \eta \cdot \pi$ is also injective.\footnote{The image of the map $\Hom(\cO_E,\cO_E) \to \Ext^1(\cL,\cO_E)$ in  
the exact sequence 
\begin{equation*}
0 \to \Hom(\cO_E,\cO_E) \to \Ext^1(\cL,\cO_E) \stackrel{\cdot  \pi}{\longrightarrow} \Ext^1(\cE,\cO_E) \to \Ext^1(\cO_E,\cO_E)  \to 0
\end{equation*}
 is $\CC \xi$ so $\{\eta \in \Ext^1(\cL,\cO_E) \; | \; \eta\cdot \pi =0\} = \CC\xi$.}
The proof is now complete.
\end{proof}

%%%%%%%%%%%%%%%%%%%%%%%%%%%%%%%%%%%%%%%%%%%%%%%%%%%%%%%%%%%%%%%%
\subsection{(Quasi-)Affineness of $L(\cE)$}
%%%%%%%%%%%%%%%%%%%%%%%%%%%%%%%%%%%%%%%%%%%%%%%%%%%%%%%%%%%%%%%%

\begin{proposition}\label{pr:affine-reductive}
  The leaf $L(\cE)$ is
  \begin{enumerate}
  \item\label{item:aff} 
  affine if either
    \begin{enumerate}
    \item 
    \label{item:aff.a}
    $n$ is odd  and $\cE$ is indecomposable, or
    \item 
     \label{item:aff.b}
     $n$ is even and $\cE$ is a direct sum of two line bundles of degree $\frac n2$;
    \end{enumerate}
  \item\label{item:naffev} quasi-affine but not affine if  $n$ is even and $\cE$ is indecomposable.  
  \end{enumerate}      
\end{proposition}
\begin{proof}
By \Cref{th:smth}\cref{item:git}, $L(\cE)$ is the geometric quotient $X(\cE)/\Aut(\cE)$. 
Hence  $L(\cE)$ is also the geometric quotient $\PP X(\cE) / \PP\! \Aut(\cE)$.

The scheme $\PP X(\cE)$ being acted upon is affine by \Cref{cor:isaff} and, by \Cref{le:dim.Aut.cE},  $\PP\! \Aut(\cE)$ is 
  \begin{itemize}
  \item[-]
   trivial when $\cE$ is indecomposable of odd degree;
  \item[-]
   the additive group $\GG_a$ when $\cE$ is indecomposable of even degree;
  \item[-] 
  the semidirect product $\GG_a^{d_2-d_1}\rtimes \GG_m=\GG_a^{n-2d_1}\rtimes \GG_m $ 
    when $\cE\cong \cN_1\oplus \cN_2$ where $d_1=\deg\cN_1\le \deg\cN_2=d_2$ and $\cN_1\not\cong \cN_2$;
  \item[-]
   ${\rm PGL}(2)$ when $\cE\cong \cL_\omega \oplus \cL_\omega$. 
  \end{itemize}

\cref{item:aff}
 In \cref{item:aff.a} and \cref{item:aff.b}, $\PP\Aut(\cE)$ is {\it reductive} in (\cite[Defn.~1.4]{mumf-git}), as seen from the above list. 
 The geometric quotient is therefore affine by \cite[Thm.~1.1 and Amplification~1.3]{mumf-git}, for example.

\cref{item:naffev}
When $n=2r$ and $\cE$ is indecomposable, $\cE \cong \cE_\omega$ for a unique $\omega \in \Omega$. 
By \Cref{th:nsplitev}\Cref{item:disju}, 
\begin{equation*}
L(\cE_\omega)\sqcup L(\cL_\omega \oplus \cL_\omega) \;=\; \Sec_{r,\omega}(E) \, - \,  \Sec_{r-1}(E).
\end{equation*}
By \Cref{th:nsplitev}\Cref{item:e0} and \Cref{item:e1}, $L(\cE_\omega)$ has dimension $n-2$ and 
$L(\cL_\omega \oplus \cL_\omega)$ has dimension $n-4$.
We will use \cite[Prop.~5.1]{neem_steins} which says the following: 
If $U \subseteq X$ is a Zariski-open affine subset of a scheme $X$ and $f:Y \to X$ is a morphism from a noetherian scheme $Y$, then
$Y-f^{-1}(U)$  has pure codimension one in $Y$. We apply {\it loc. cit.} with $X=Y=L(\cE_{\omega})\sqcup L(\cL_{\omega}\oplus L_{\omega})$,
$U=L(\cE_{\omega})$ and $f=\id_X$: since $U$ has codimension two in $Y$, it is not affine.  
 
 However, $L(\cE_{\omega})$ is quasi-affine by \cite[Prop.~3]{fm-ga} because it equals
  $\PP X(\cE_\omega)/\PP\Aut(\cE_\omega)$ and $\PP X(\cE_\omega)$ is affine (\Cref{cor:isaff}) and 
   $\PP\Aut(\cE_\omega)\cong \GG_a$ is unipotent \cite[\S 17.5]{hmph-lag}.
\end{proof}

\begin{example}\label{ex:dd}
  Suppose $n=2r \ge 4$. Assume $\cE = \cE_{r,x}$ for some $x \in E-\Omega$. Thus $\cE= \cN_1 \oplus \cN_2$ where $\cN_1$ and $\cN_2$ are non-isomorphic line bundles of degree $r$.  
 By \Cref{pr:affine-reductive}\cref{item:aff}\Cref{item:aff.b}, $L(\cE)$ is affine. We now give  a direct proof of this to show how the 
 familiar  $n=4$ case extends to larger even $n$. When $n=4$, $E \subseteq \PP_\cL \cong \PP^3$ is contained in a pencil of quadrics, four of which are singular and, 
 if $Q$ is one of the smooth quadrics, then $Q-E$ is an affine variety (and a symplectic leaf).\footnote{If $Q$ is one of the singular quadrics,
  then it is a cone whose vertex is a 
 symplectic leaf (the vertices are the leaves in \cref{th:nsplitev}\cref{item:e1}), 
 and $Q-(E \cup \text{the vertex})$ is a symplectic leaf that is quasi-affine but not affine (\cref{pr:affine-reductive}\cref{item:naffev}); the leaves  $Q-(E \cup \text{the vertex})$ are those in \cref{th:nsplitev}\cref{item:e0}.
 More about the geometry of $E \subseteq \PP^3$  can be found in \cite{Hulek86}.}
  
 Clearly, $\Aut(\cN_1 \oplus \cN_2) \cong \CC^\times \times \CC^\times$.

Let $D_1$ and $D_2$ be effective divisors such that $\cN_1 \cong \cO_E(D_1)$ and  $\cN_2 \cong \cO_E(D_2)$. 
 
Let $f=(s_1,s_2):\cO_E \to \cE=\cN_1 \oplus \cN_2$ where  $s_i \in \Hom(\cO_E,\cN_i)=H^0(E,\cN_i)$.  
Let $(s_i)_0 \in E^{[r]}$ be the divisor of zeroes of $s_i$.
Then $f$ belongs to $X(\cE)$ (i.e., $\coker(f)$ is torsion-free) if and only if $s_1 \ne 0$, $s_2 \ne 0$ and $(s_1)_0  \cap (s_2)_0=\varnothing$.
Hence
\begin{equation}
\label{eq:bad.maps}
 Z_\cE \cup \{0\} \; \supseteq \;  \big( \Hom(\cO_E,\cN_1)  \times \{0\} \big)  \;   \cup   \;   \big( \{0\} \times \Hom(\cO_E,\cN_2)   \big),
\end{equation}
and
\begin{equation*}
X(\cE) \;=\; \Hom(\cO_E, \cE)-Z_\cE \cup \{0\}  \; \subseteq \;  \big( \Hom(\cO_E,\cN_1) - \{0\} \big)  \times \big( \Hom(\cO_E,\cN_2)   - \{0\}  \big) .
\end{equation*}
Since each copy of $\CC^\times$ in $\Aut(\cE)$ acts on the appropriate $\Hom(\cO_E,\cN_i)$ by scaling, 
  \begin{equation*}
L(\cE) \; \cong \; X(\cE)/\Aut(\cE) \; \subseteq \;  \PP H^0(E,\cN_1)   \times   \PP H^0(E,\cN_2)   \; \cong \;    \PP^{r-1}\times \PP^{r-1}.  
  \end{equation*}
  (When $n=4$, the right-hand side is a smooth quadric in $\PP_\cL \cong \PP^3$ that contains $E$.)

It now follows from the condition that $(s_1)_0 \cap (s_2)_0=\varnothing$ that the image of $Z_{\cE}\cup \{0\}$    under the map
    \begin{equation*}
    \Hom(\cO_E,\cE)\,-\,   (\text{the right-hand side of \cref{eq:bad.maps}})       \, \longrightarrow \, \PP^{r-1}\times \PP^{r-1}
  \end{equation*}
  is the effective divisor $D \in \Div(\PP^{r-1}\times \PP^{r-1})$ consisting of those pairs
  \begin{equation*}
    (s_1,s_2) \; \in \;  \PP H^0(E,\cN_1)\times \PP H^0(E,\cN_2)
  \end{equation*}
  whose zero loci, which are effective degree-$r$ divisors, equivalent to $D_1$ and $D_2$ respectively, do {\it not} intersect.
  For a generic $s_1 \in \PP H^0(E,\cN_1)$, $(s_1)_0$  consists of $r$ distinct points on $E$. 
  Avoiding these in the zero locus of $s_2 \in \PP H^0(E,\cN_2)$ means avoiding $r$ hyperplanes in $\PP H^0(E,\cN_2)$. Hence 
  the intersection of $D$ with the generic $\PP^{r-1}$ fiber of the second projection
  \begin{equation*}
    \PP^{r-1}\times \PP^{r-1}\, \longrightarrow \, \PP^{r-1}
  \end{equation*}
  consists of $r$ hyperplanes. Interchanging the roles of $\cN_1$ and $\cN_2$, the same goes for the other projection. 
  But the Picard group of $\PP^{r-1}\times \PP^{r-1}$ is $\ZZ\oplus \ZZ$ (\cite[Exer.~III.12.6]{hrt}) and the class of $D$ in it is 
  $(r,r)$, so  $D$ is ample; its complement is  therefore affine (see the footnote to the proof of \cref{cor:isaff}).
  \reqed
\end{example}

\begin{remark}
  \Cref{pr:affine-reductive}\Cref{item:naffev} is an example of a $\GG_a$-action on an affine scheme whose geometric quotient is quasi-affine
  but not affine; \cite[Ex.~2.5]{gp1} is another such example, as is the quotient
  \begin{equation*}
    {\textrm SL}(2)\to {\textrm SL}(2)/\left(\text{upper triangular unipotent matrices}\right) \; \cong \; \AA^2-\{(0,0)\}
  \end{equation*}
  that makes an oblique appearance in the above proof of non-affineness. 

  However, under mild conditions that do not concern us here, quotients of quasi-affine schemes by actions of
  unipotent-group  are   quasi-affine \cite[Prop.~3 and Thm.~4]{fm-ga}. This is not so for $\GG_m$:
  \begin{equation*}
    \text{(quasi-affine)}
    \quad
    \AA^n - \{0\}\xrightarrow{\qquad} (\AA^n - \{0\})/\GG_m \; \cong \; \PP^{n-1}
    \quad
    \text{(non-quasi-affine)}.
  \end{equation*}
  See \cite[Thm.~3]{faunt-cat}, though, for a positive result: if a linear algebraic group does {\it not} map onto $\GG_m$ then quotients of quasi-affine varieties by its actions are quasi-affine (under mild conditions).  \reqed
\end{remark}

%%%%%%%%%%%%%%%%%%%%%%%%%%%%%%%%%%%%%%%%%%%%%%%%%%%%%%%%%%%%%%%%
%%%%%%%%%%%%%%%%%%%%%%%%%%%%%%%%%%%%%%%%%%%%%%%%%%%%%%%%%%%%%%%%
\appendix
%%%%%%%%%%%%%%%%%%%%%%%%%%%%%%%%%%%%%%%%%%%%%%%%%%%%%%%%%%%%%%%%
%%%%%%%%%%%%%%%%%%%%%%%%%%%%%%%%%%%%%%%%%%%%%%%%%%%%%%%%%%%%%%%%

%%%%%%%%%%%%%%%%%%%%%%%%%%%%%%%%%%%%%%%%%%%%%%%%%%%%%%%%%%%%%%%%
%%%%%%%%%%%%%%%%%%%%%%%%%%%%%%%%%%%%%%%%%%%%%%%%%%%%%%%%%%%%%%%%
\section{Isomorphism classes of extensions}
\label{sect.appx.extns}
%%%%%%%%%%%%%%%%%%%%%%%%%%%%%%%%%%%%%%%%%%%%%%%%%%%%%%%%%%%%%%%%
%%%%%%%%%%%%%%%%%%%%%%%%%%%%%%%%%%%%%%%%%%%%%%%%%%%%%%%%%%%%%%%%

Let $\Bbbk$ be a commutative ring with identity, and let $\cA$ be a $\Bbbk$-linear abelian category.
Given objects $A$ and $C$ in $\cA$, we define the category $\AA(C,A)$ as follows: The objects in $\AA(C,A)$ are the exact sequences
\begin{equation}
\label{defn.nx.bugs}
0 \to A' \to B' \to C' \to 0
\end{equation}
 in which $A' \cong A$ and $C' \cong C$. 
 We call  \cref{defn.nx.bugs}  an {\sf extension} of $C'$ by $A'$.
A morphism $\xi_1 \to \xi_2$ between exact sequences
\begin{equation}
\label{eq:defn.morphism}
\xymatrix{
\xi_i: \;0 \ar[r] & A _i \ar[r]^{f_i}  & B_i \ar[r]^{g_i} &   C_i \ar[r]   &  0 , \qquad (i=1,2), 
}
\end{equation}
is a triple $(\mu,\nu,\tau)$ such that the  diagram
\begin{equation}
\label{eq:comm.diag}
\xymatrix{
0 \ar[r] & A_1 \ar[r]^{f_1}  \ar[d]_{\mu} & B_1 \ar[r]^{g_1}\ar[d]^{\nu} &C_1  \ar[d]^{\tau}\ar[r]  & 0
\\
0 \ar[r] & A_2 \ar[r]_{f_2} & B_2 \ar[r]_{g_2} &C_2 \ar[r] & 0
}
\end{equation}
 commutes.\footnote{Mac Lane has a section on the category of short exact sequences in his book \cite[Ch.~XII \S6]{Mac95}.}
If $(\mu,\nu,\tau)$ are such that this  diagram commutes, then $(\mu,\nu,\tau)$
 is an isomorphism if and only if $\mu$, $\nu$ and $\tau$ are isomorphisms; we then write $\xi_1 \cong \xi_2$.

If $f:A \to B$ is monic and $g_i:B \to C_i$, $i=1,2$, are cokernels for $f$, then the sequences 
\begin{equation*}
0 \to A \stackrel{f}{\longrightarrow} B \stackrel{g_i}{\longrightarrow} C_i \to 0 \qquad (i=1,2)
\end{equation*}
are isomorphic.

$\AA$ is a $\Bbbk$-linear category: the addition $(\mu,\nu,\tau)+(\mu',\nu',\tau'):=(\mu+\mu',\nu+\nu',\tau+\tau')$ gives $\Hom_\AA(\xi_1,\xi_2)$ the structure of an abelian group; it is a $\Bbbk$-module with scalar multiplication given by $\lambda(\mu,\nu,\tau):=(\lambda\mu,\lambda\nu,\lambda\tau)$ for $\lambda \in \Bbbk$.

Assume $\Bbbk$ is a field. \cref{prop.isom.extns} shows that in the situation of interest in this paper, where $A$ and $C$ are sheaves of sections of stable vector bundles on $E$, the points in the projective space $\PP\Ext^1_\cA(C,A)$ of 1-dimensional subspaces of $\Ext^1_\cA(C,A)$ are the isomorphism classes of exact sequences $0 \to A' \to B' \to C' \to 0$ in which $A' \cong A$ and $C' \cong C$.

Let $\a:A' \to A$ and $\gamma:C' \to C$ be isomorphisms. The map that sends an extension 
\begin{equation}
\label{eq:extn1}
\xymatrix{
\xi': \quad 0 \ar[r] & A' \ar[r]^{f}  & B' \ar[r]^{g}&C' \ar[r]  & 0  
}
\end{equation}
to the extension 
\begin{equation}
\label{eq:extn2}
\xymatrix{
\xi: \quad      0 \ar[r] & A \ar[r]^{ f\a^{-1}}  & B' \ar[r]^{\gamma g}&C \ar[r]  & 0  
}
\end{equation}
induces a $\Bbbk$-linear isomorphism 
\begin{equation*}
\phi_{\a,\gamma}:\Ext^1_\cA(C',A') \to \Ext^1_\cA(C,A).
\end{equation*}
Let $\Phi_{\a,\gamma}$ be the unique morphism such that the diagram
\begin{equation}
\label{eq:defn.Phi}
\xymatrix{
\Ext^1_\cA(C',A') -\{0\}  \ar[r]^{\phi_{\a,\gamma}} \ar[d] &  \Ext^1_\cA(C,A) - \{0\} \ar[d]
\\
\PP\Ext^1_\cA(C',A') \ar[r]_{\Phi_{\a,\gamma}} &  \PP\Ext^1_\cA(C,A)
}
\end{equation}
commutes.
Although $\phi_{\a,\gamma}$ depends on the choice of $\a$ and $\gamma$,  \cref{prop.isom.extns} shows that  in the situation of 
interest in this paper the isomorphism $\Phi_{\a,\gamma}$ is the same for all choices of $\a$ and $\gamma$ and, as a consequence, 
$\PP\Ext^1_\cA(C,A)$ is in natural bijection with isomorphism classes of non-split exact sequences 
$0 \to A' \to B' \to C' \to 0$ in which $A' \cong A$ and $C' \cong C$.

\begin{proposition}
\label{prop.isom.extns}
Let $\Bbbk$ be a field, and $A$ and $C$ objects in a $\Bbbk$-linear abelian category $\cA$.
If $\End(A)=\End(C)=\Bbbk$, then the map $\Phi_{\a,\gamma}$ in \cref{eq:defn.Phi} does not depend on the choice of $\a$ or $\gamma$,
and  the map
\begin{equation}
\label{eq:isom.extns}
\begin{cases}
\; \text{non-split exact sequences}
\\
\; 0 \to A' \to B' \to C' \to 0
\end{cases}
\Bigg\vert 
\; 
 \text{$A' \cong A$ and $C' \cong C$}
\Bigg\} \Bigg\slash \!\!  \cong
\quad \longrightarrow \quad
\PP\Ext^1_\cA(C,A)
\end{equation}
that sends \cref{eq:extn1} to  \cref{eq:extn2}  is a well-defined bijection that does not depend on the choice of $\a$ or $\gamma$.
\end{proposition}
\begin{proof}
Consider two non-split isomorphic exact sequences
\begin{equation}
\label{eq:two.ses}
\xymatrix{
\xi_i: \;0 \ar[r] & A_i \ar[r]^{f_i}  & B_i \ar[r]^{g_i}&C_i \ar[r]  & 0 \qquad (i=1,2)
}
\end{equation}
in which $A_i \cong A$ and $C_i \cong C$. 
There is a commutative diagram 
\begin{equation*}
\xymatrix{
0 \ar[r] & A_1 \ar[r]^{ f_1}  \ar[d]_{\mu} & B_1 \ar[r]^{g_1}  \ar[d]_{\nu} &C_1 \ar[r] \ar[d]_{\tau} & 0 
\\
0 \ar[r] & A_2 \ar[r]_{ f_2}  & B_2 \ar[r]_{g_2}&C_2 \ar[r]  & 0 
}
\end{equation*}
in which $\mu$, $\nu$, and $\tau$ are isomorphisms.

Fix isomorphisms $\a_i:A_i \to A$ and $\gamma_i:C_i \to C$. For brevity write $\phi_i=\phi_{\a_i,\gamma_i}$ and $\Phi_i=\Phi_{\a_i,\gamma_i}$. Then $\phi_1(\xi_1)$ and $\phi_2(\xi_2)$ are the exact sequences
\begin{equation}
\label{eq:phi.two.ses}
\xymatrix{
  0 \ar[r] & A \ar[r]^{ f_i\a_i^{-1}}  & B_i \ar[r]^{\gamma_i g_i}&C \ar[r]  & 0  \qquad (i=1,2).
}
\end{equation}
The maps  $\lambda:=\a_2\mu\a_1^{-1}$ and  $\theta:=\gamma_2\tau\gamma_1^{-1}$  are isomorphisms $A \to A$ and $C \to C$, respectively.

Since the diagram
\begin{equation*}
\xymatrix{
\phi_1(\xi_1): \quad    
0 \ar[r] & A \ar[r]^{ f_1\a_1^{-1}}  \ar[d]_{ \lambda } & B_1  \ar[d]_{\nu} \ar[r]^{\gamma_1 g_1}&C \ar[r] \ar[d]^{\theta} & 0 
\\
\phi_2(\xi_2): \quad    
0 \ar[r] & A \ar[r]_{ f_2\a_2^{-1}}  & B_2 \ar[r]_{\gamma_2 g_2}&C \ar[r]  & 0 
}
\end{equation*}
commutes, $(\lambda,\nu,\theta)$ is an isomorphism $\phi_1(\xi_1) \to \phi_2(\xi_2)$. 

We will now show that  the equivalence classes  of $\phi_1(\xi_1)$ and $\phi_2(\xi_2)$, which belong to $\Ext^1_\cA(C,A)$,  are
 scalar multiples of each other.

Because $\End_\cA(A)=\End_\cA(C)=\Bbbk$, $\lambda$ and $\theta$ are scalar multiples of the identity maps $\id_A$ and $\id_C$,
so, with a small abuse of notation, we can view $\lambda$ and $\theta$ as scalars and deduce that the diagram
 \begin{equation*}
\xymatrix{
0 \ar[r] & A \ar[rr]^{f_1\a_1^{-1}}  \ar@{=}[d]  && B_1 \ar[rr]^{\theta  \gamma_1 g_1}\ar[d]^{\nu} && C  \ar@{=}[d]  \ar[r]  & 0
\\
0 \ar[r] & A \ar[rr]_{\lambda  f_2\a_2^{-1}} && B_2 \ar[rr]_{\gamma_2 g_2} &&C \ar[r] & 0
}
\end{equation*}
commutes. But the rows of this diagram are scalar multiples of $\phi_1(\xi_1)$ and $\phi_2(\xi_2)$ so  the equivalence classes  of $\phi_1(\xi_1)$ and $\phi_2(\xi_2)$ become equal in $\PP\Ext^1_\cA(C,A)$, i.e., $\Phi_1(\xi_1)=\Phi_2(\xi_2)$.
This completes the proof that the map in \cref{eq:isom.extns} is well defined.

If we consider the case $A_1=A_2$ and $C_1=C_2$, the argument above also shows that the morphisms 
\begin{equation*}
\Phi_1,\Phi_2: \PP\Ext^1_\cA(C_1,A_1)  \, \longrightarrow \,  \PP \Ext^1_\cA(C,A)
\end{equation*}
are the same.

The map  in \cref{eq:isom.extns} is certainly surjective. It remains to show it is injective.

Let $\xi_1$ and $\xi_2$ be non-split exact sequences as in \cref{eq:two.ses}. 
As before, fix   isomorphisms $\a_i:A_i \to A$ and $\gamma_i:C_i \to C$,
and write $\phi_1(\xi_1)$ and $\phi_2(\xi_2)$ for  the exact sequences in \cref{eq:phi.two.ses}. Suppose that $\Phi_1(\xi_1)=\Phi_2(\xi_2)$. Then the equivalence classes containing $\phi_1(\xi_1)$ and $\phi_2(\xi_2)$ are scalar multiples of each other.
Hence there is a non-zero scalar, $\d \in \Bbbk$, such that $\d\phi_1(\xi_1)$ and $\phi_2(\xi_2)$ are equivalent, i.e., there is 
an isomorphism $\nu:B_1 \to B_2$ such that the diagram
\begin{equation*}
\xymatrix{
0 \ar[r] & A \ar[rr]^{f_1\a_1^{-1}}  \ar@{=}[d]  && B_1 \ar[rr]^{\d  \gamma_1 g_1}\ar[d]^{\nu} && C  \ar@{=}[d]  \ar[r]  & 0
\\
0 \ar[r] & A \ar[rr]_{f_2\a_2^{-1}} && B_2 \ar[rr]_{\gamma_2 g_2} &&C \ar[r] & 0
}
\end{equation*}
commutes. It follows that the diagram
\begin{equation*}
\xymatrix{
0 \ar[r] & A_1 \ar[r]^{ f_1}  \ar[d]_{\a_2^{-1}\a_1} & B_1 \ar[r]^{g_1}  \ar[d]_{\nu} &C_1 \ar[r] \ar[d]^{\d\gamma_2^{-1}\gamma_1} & 0 
\\
0 \ar[r] & A_2 \ar[r]_{ f_2}  & B_2 \ar[r]_{g_2}&C_2 \ar[r]  & 0 
}
\end{equation*}
commutes. Thus, $\xi_1 \cong \xi_2$. The map in \cref{eq:isom.extns} is therefore injective.
\end{proof}

\subsubsection{Equivalence of extensions}
Two  extensions
\begin{equation*}
\xymatrix{
\xi_i: \;0 \ar[r] & A \ar[r]^{f_i}  & B_i \ar[r]^{g_i}&C \ar[r]  & 0 && (i=1,2)
}
\end{equation*}
are  {\sf equivalent}, denoted $\xi_1 \equiv \xi_2$, if there is a commutative diagram
\begin{equation*}
\xymatrix{
0 \ar[r] & A \ar[r]^{f_1}  \ar@{=}[d] & B_1 \ar[r]^{g_1}\ar[d]^{\cong} &C  \ar@{=}[d] \ar[r]  & 0
\\
0 \ar[r] & A \ar[r]_{f_2} & B_2 \ar[r]_{g_2} &C \ar[r] & 0.
}
\end{equation*}

By definition, $\Ext^1_\cA(C,A)$ is the set of equivalence classes of 
extensions of $C$ by $A$.
It is a $\Bbbk$-module.
If $p$ is an odd prime and $\pi:\ZZ \to \ZZ/\ZZ p$ the map $\pi(x)=x+\ZZ p$, then the sequences 
$0 \to \ZZ \stackrel{p}{\longrightarrow} \ZZ \stackrel{\pi}{\longrightarrow} \ZZ/\ZZ p \to 0$
and $0 \to \ZZ \stackrel{p}{\longrightarrow} \ZZ \stackrel{2\pi}{\longrightarrow} \ZZ/\ZZ p \to 0$  are isomorphic but not equivalent.

If $X$ and $Y$ are objects in $\cA$ there is a ``composition'' map
\begin{equation}
\label{eq:Ext1.bimodule}
\Hom(X,A) \times \Ext^1_\cA(Y,X) \times \Hom(C,Y) \; \longrightarrow \; \Ext^1_\cA(C,A), \quad (\a,\xi,\b) \mapsto \a \cdot \xi \cdot \b.
\end{equation}

%%%%%%%%%%%%%%%%%%%%%%%%%%%%%%%%%%%%%%%%%%%%%%%%%%%%%%%%%%%%%%%%
%%%%%%%%%%%%%%%%%%%%%%%%%%%%%%%%%%%%%%%%%%%%%%%%%%%%%%%%%%%%%%%%
\section{Group actions on objects in rigid monoidal categories}
\label{ssect.rigid.symm.mon.cats}
%%%%%%%%%%%%%%%%%%%%%%%%%%%%%%%%%%%%%%%%%%%%%%%%%%%%%%%%%%%%%%%%
%%%%%%%%%%%%%%%%%%%%%%%%%%%%%%%%%%%%%%%%%%%%%%%%%%%%%%%%%%%%%%%%

Let $(\cC,\otimes,\mathbf{1})$ be a rigid monoidal category (see \cite{egno} for the terminology on monoidal categories).

In this section, the letter $G$ always denotes a group.

A \emph{$G$-action} on an object $x\in\cC$ is a group homomorphism $\rho_{x}:G\to\Aut(x)$. If $G$ acts on $x$ and $y$, then it acts on $x\otimes y$ by $\rho_{x\otimes y}=\rho_{x}\otimes\rho_{y}$. 
It also acts on the \emph{set} $\Hom_{\cC}(x,y)$ by
\begin{equation*}
	\alpha\triangleright f\;:=\;\rho_{y}(\alpha)\circ f\circ\rho_{x}(\alpha)^{-1}\quad \text{($\alpha\in G$, $f\in\Hom_{\cC}(x,y)$)}.
\end{equation*}
If we are given a $G$-action on only one of $x$ and $y$ we can impose the trivial action of $G$ on the other one and so obtain an action of $G$ on $\Hom_{\cC}(x,y)$.

A morphism $f\in\Hom_{\cC}(x,y)$ is said to be \emph{$G$-equivariant} if $\alpha\triangleright f=f$ for all $\alpha\in G$.

\begin{proposition}\label{prop.equiv.map.mor}
	Suppose $G$ acts on $x,y,z,w\in\cC$. If a map $\varphi:\Hom_{\cC}(x,y)\to\Hom_{\cC}(z,w)$ is $G$-equivariant, then $\varphi$ preserves $G$-equivariance of morphisms.
\end{proposition}
\begin{proof}
	Suppose $f\in\Hom_{\cC}(x,y)$ is $G$-equivariant. If $\alpha\in G$, then $\alpha\triangleright f=f$ so $\alpha\triangleright\varphi(f)=\varphi(\alpha\triangleright f)=f$.
\end{proof}

\begin{proposition}\label{prop.comp.equiv}
	Suppose $G$ acts on $x,x',y\in\cC$. If $f:x\to x'$ is $G$-equivariant, then the maps
	\begin{enumerate}
		\item $f_{*}:\Hom_{\cC}(y,x)\to\Hom_{\cC}(y,x')$, $h\mapsto f\circ h$, and
		\item $f^{*}:\Hom_{\cC}(x',y)\to\Hom_{\cC}(x,y)$, $h\mapsto h\circ f$,
	\end{enumerate}
	are $G$-equivariant. 
	
	Consequently, the maps $f_{*}$ and $f^{*}$ preserve $G$-equivariance of morphisms.
\end{proposition}
\begin{proof}
	Let $\alpha\in G$ and $h\in\Hom_{\cC}(y,x)$. Since $f$ is $G$-equivariant, $f\circ\rho_{x}(\alpha)=\rho_{x'}(\alpha)\circ f$, so
	\begin{align*}
		f_{*}(\alpha\triangleright h)&=f_{*}(\rho_{x}(\alpha)\circ h\circ\rho_{y}(\alpha)^{-1})=f\circ\rho_{x}(\alpha)\circ h\circ\rho_{y}(\alpha)^{-1}=\rho_{x'}(\alpha)\circ f\circ h\circ \rho_{y}(\alpha)^{-1}\\
		&=\rho_{x'}(\alpha)\circ f_{*}(h)\circ\rho_{y}(\alpha)^{-1}=\alpha\triangleright f_{*}(h).
	\end{align*}
	The second statement is proved dually. The last statement follows from \cref{prop.equiv.map.mor}.
\end{proof}

Given an object $x \in \cC$, let $x^*$ denote its left dual (when $\cC$ is symmetric, which it will be after the next result, 
the left dual is isomorphic to the right dual, but we will consistently use the {\it left} dual).

If $G$ acts on $x$, then it also acts on $x^{*}$; for each $\alpha\in G$, the automorphism $\rho_{x}(\alpha):x\to x$ induces the automorphism $\rho_{x}(\alpha)^{*}:x^{*}\to x^{*}$, and $\rho_{x^{*}}(\alpha):=(\rho_{x}(\alpha)^{*})^{-1}$ defines the action on $x^{*}$.

\begin{proposition}\label{prop.adj.equiv}
	Suppose $G$ acts on $x,y,z\in\cC$. The bijection
	\begin{equation}\label{eq.adj.map}
		\Hom_{\cC}(y\otimes x,z)\to\Hom_{\cC}(y,z\otimes x^{*})
	\end{equation}
	given by the adjunction $(-\otimes x) \dashv (-\otimes x^*)$ is $G$-equivariant, so it preserves $G$-equivariance of morphisms.
\end{proposition}

\begin{proof}
	Let $f\in\Hom_{\cC}(y\otimes x,z)$. There is a commutative diagram
	\begin{equation}\label{eq.adj.act}
		\begin{tikzcd}[column sep=30mm]
			y\ar[r,"\id\otimes\operatorname{coev}"]\ar[d,"\rho_{y}(\alpha)"'] & y\otimes x\otimes x^{*}\ar[r,"f\otimes\id"]\ar[d,"\rho_{y}(\alpha)\otimes\rho_{x}(\alpha)\otimes\rho_{x}(\alpha)^{*-1}"] & z\otimes x^{*}\ar[d,"\rho_{z}(\alpha)\otimes\rho_{x}(\alpha)^{*-1}"] \\
			y\ar[r,"\id\otimes\operatorname{coev}"'] & y\otimes x\otimes x^{*}\ar[r,"(\alpha\,\triangleright f)\otimes\id"'] & z\otimes x^{*}
		\end{tikzcd}
	\end{equation}
	(in which we have omitted the isomorphism $(y\otimes x)\otimes x^{*}\to y\otimes (x\otimes x^{*})$); the left-hand square is $y\otimes(-)$ applied to the diagram
	\begin{equation*}
		\begin{tikzcd}
			\mathbf{1}\ar[r,"\operatorname{coev}"]\ar[d,"\id"'] & x\otimes x^{*}\ar[d,"\rho_{x}(\alpha)\otimes\rho_{x}(\alpha)^{*-1}"] \\
			\mathbf{1}\ar[r,"\operatorname{coev}"'] & x\otimes x^{*}
		\end{tikzcd}
	\end{equation*}
	whose commutativity follows from the definition of the left dual $\rho_{x}(\alpha)^{*}$.
	
	The first and the second rows of \cref{eq.adj.act} are the images of $f$ and $\alpha\triangleright f$, respectively, under the map \cref{eq.adj.map},  
so the commutativity of \cref{eq.adj.act} implies that \cref{eq.adj.map} is $G$-equivariant.
\end{proof}

For the rest of this section we assume $\cC$ is symmetric with braiding $c_{x,y}:x\otimes y\to y\otimes x$, and  abelian. 

We define the second exterior power $\wedge^{2}x$ of $x\in\cC$ to be the image of the morphism
\begin{equation*}
	\begin{tikzcd}[column sep=large]
		x \otimes x \ar[r,"1-c_{x,x}"] & x\otimes x.
	\end{tikzcd}
\end{equation*}
If $G$ acts on $x$, then $\rho_{x\otimes x}(\alpha)=\rho_{x}(\alpha)\otimes\rho_{x}(\alpha)$ induces an automorphism $\rho_{\wedge^{2}x}(\alpha)$ of $\wedge^{2}x$. Thus $G$ also acts on $\wedge^{2}x$, and the quotient morphism $x\otimes x\to\wedge^{2}x$ is $G$-equivariant.

\begin{remark}\label{rem.equiv.ex}
	The equivariance of the maps $\nu$ in \cref{le:edete} and $\theta\circ\nu_{*}$ in \cref{re:kercoker} can now be deduced as follows: the group $\Aut(x)$ acts on $x$ and the quotient morphism $x\otimes x\to\wedge^{2}x$ is $\Aut(x)$-equivariant, so the morphism $\nu:x\to\wedge^{2}x\otimes x^{*}$ obtained by the adjunction $(-\otimes x) \dashv (-\otimes x^*)$ is $\Aut(x)$-equivariant by \cref{prop.adj.equiv}. Therefore the map $\Hom_{\cC}(\mathbf{1},x)\to\Hom_{\cC}(x,\wedge^{2}x)$ that is the composition
	\begin{equation*}
		\begin{tikzcd}
			\Hom_{\cC}(\mathbf{1},x)\ar[r,"\nu_{*}"] & \Hom_{\cC}(\mathbf{1},\wedge^{2}x\otimes x^{*})\ar[r,"\operatorname{adj}"] & \Hom_{\cC}(x,\wedge^{2}x),
		\end{tikzcd}
	\end{equation*}
	where the latter map is the adjunction, is $\Aut(x)$-equivariant by \cref{prop.comp.equiv,prop.adj.equiv}, where $\mathbf{1}$ 
is equipped with the trivial action of $\Aut(x)$.
\end{remark}

%%%%%%%%%%%%%%%%%%%%%%%%%%%%%%%%%%%%%%%%%%%%%%%%%%%%%%%%%%%%%%%%
%%%%%%%%%%%%%%%%%%%%%%%%%%%%%%%%%%%%%%%%%%%%%%%%%%%%%%%%%%%%%%%%
\section{Tangent spaces to points on Grassmannians}
\label{ssect.T_x.GG}
%%%%%%%%%%%%%%%%%%%%%%%%%%%%%%%%%%%%%%%%%%%%%%%%%%%%%%%%%%%%%%%%
%%%%%%%%%%%%%%%%%%%%%%%%%%%%%%%%%%%%%%%%%%%%%%%%%%%%%%%%%%%%%%%%

 Let $\CC[\varepsilon]:=\CC\oplus \CC \varepsilon$ be the  ring of dual numbers; i.e., $\varepsilon$ is a formal variable whose square is zero.
  
  As a preliminary, we recall two equivalent pictures of the tangent space to the Grassmannian $\GG(k,V)$ of $k$-planes in a 
  finite-dimensional vector space $V$. 
  With that in mind, fix a $k$-dimensional subspace $V' \subseteq V$ and consider it as a point of $\GG(k,V)$.

  On the one hand, \cite[Thm.~3.5]{3264} gives an identification
  \begin{equation}\label{eq:tggrashom}
    T_{V'}\GG(k,V) \; = \;  \Hom(V',V/V')
  \end{equation}
  that is canonical in the sense that it glues over all $V'$ to give an isomorphism of bundles.

  On the other hand, by \cite[Exer.~II.2.8]{hrt},
  \begin{equation*}
    T_{V'}\GG(k,V) \;=\;  \{\Psi:\Spec\CC[\varepsilon]\to \GG(k,V)\ |\ \Psi(\text{closed point})=V'\}
  \end{equation*}
 By the universal property of the Grassmannians (as in \cite[Exer.~VI-18]{EH00} and \cite[discussion immediately preceding \S 3.2.4]{3264}), this amounts to
  \begin{align*}
    T_{V'}\GG(k,V) \; \cong&\quad\text{rank-$k$ $\CC[\varepsilon]$-summands } V'_{\varepsilon}\subseteq V[\varepsilon]:=V\otimes_{\CC}\CC[\varepsilon] \text{ for which}
                             \numberthis\label{eq:tggraseps}
    \\
                           &\quad  
                             \text{the map $V'_{\varepsilon}/\varepsilon V'_{\varepsilon}\to V=  V[\varepsilon]/\varepsilon V[\varepsilon]$   is an isomorphism onto } V'.
  \end{align*}
Items \cref{item.tangent.a}, \cref{item.tangent.b}, \cref{item.tangent.c} below describe the natural bijection between \Cref{eq:tggrashom} and \Cref{eq:tggraseps}. First, let $\iota:V' \to V$ and $\a:V \to V/V'$
be the inclusion and quotient maps, and fix a linear map $\b:V/V' \to V$ such that $\a\b=\id_{V/V'}$.

\begin{enumerate}[(a),inline]
\item\label{item.tangent.a}
Given $\theta \in  \Hom(V',V/V')$ as in \Cref{eq:tggrashom}, we define  
    \begin{equation}
    \label{eq:V_eps.defn}
   V'_{\varepsilon} \; := \;   \{x+\varepsilon \b\theta(x) \; | \; x \in V'\}+ \varepsilon V' \;=\; \varepsilon V' + \im(\iota+\varepsilon \b\theta).
    \end{equation}

Roughly speaking, the construction $\theta\mapsto V'_{\varepsilon}$ sends $\theta$ to an isomorphic copy of its graph:
 first, $V'_{\varepsilon}$ always contains $\varepsilon V'\subset V[\varepsilon]$;  second, there is an isomorphism; 
      \begin{equation*} 
   \mu:     V\times (V/V')   \, \longrightarrow \, V[\varepsilon]/\varepsilon V',    \qquad \mu(w,z) :=  w+\varepsilon \beta(z) + \varepsilon V';
      \end{equation*}
third, $V'_{\varepsilon}/\varepsilon V'  = \mu(\Gamma)$ where $\Gamma:= \{(x,\theta(x))\in V\times V/V'\ |\ x \in V'\}$
is the graph of $\theta$.     

\item\label{item.tangent.b}
  Conversely, given $V'_{\varepsilon}$ as in \Cref{eq:tggraseps}, 
 the condition that the map $V'_{\varepsilon}/\varepsilon V'_{\varepsilon}\to V$
  is an isomorphism onto $V'$ implies that $V'_\varepsilon \cap \varepsilon V = \varepsilon V'_\varepsilon$; hence 
  for each $x\in V'$ there is an element $y \in V$, unique modulo $V'$, such that  $x+\varepsilon y \in V'_{\varepsilon}$;
 we may therefore define $ \theta \in  \Hom(V',V/V')= T_{V'}\GG(k,V) $ by the formula
    \begin{equation}
      \label{eq:V_eps.theta.defn}
      \theta(x) \; := \; y+V' 
    \end{equation}
    to recover the equality in \Cref{eq:tggrashom}. 

\item\label{item.tangent.c}
The procedures in \cref{item.tangent.a} and \cref{item.tangent.b}  are mutual inverses.

Suppose we are given $\theta_1 \in  \Hom(V',V/V')$. Define $V'_{\varepsilon}:=\varepsilon V' + \im(\iota+\varepsilon \b\theta_1)$ as in \cref{eq:V_eps.defn}.
If $x \in V'$, then $x+\varepsilon \b\theta_1(x)=(\iota+\varepsilon \b\theta_1)(x) \in V'_\varepsilon$ so, in \cref{item.tangent.b}, we can take $y=\b\theta_1(x)$; the
map $\theta_2$ defined by \cref{eq:V_eps.theta.defn} is therefore $\theta_2(x)=\b\theta_1(x) + V'=\a\b\theta_1(x) =\theta_1(x)$; i.e., $\theta_2=\theta_1$. 

Suppose we are given $V'_{\varepsilon,1}$ satisfying \Cref{eq:tggraseps}. Define $\theta$ by \cref{eq:V_eps.theta.defn},
then define $V'_{\varepsilon,2}$ as in \cref{item.tangent.a}; i.e., $V'_{\varepsilon,2}:=\varepsilon V' + \text{im}(\iota+\varepsilon \b\theta)$. 
Let $x \in V'$ and let $y \in V$ be such that $x+\varepsilon y \in  V'_{\varepsilon,1}$; since $\theta(x)= y+V'=\a(y)$, 
 \begin{equation*}
 (\iota+\varepsilon \b\theta)(x) \;=\;  x+\varepsilon \b\a(y) \;=\;  x+\varepsilon y + \varepsilon( \b\a(y)-y) \; \in \; V'_{\varepsilon,1}+ \varepsilon V' =V'_{\varepsilon,1}
 \end{equation*}
 because $\beta\alpha(y)-y \in \ker(\a)=V'$. Hence $\text{im}(\iota+\varepsilon \b\theta) \subseteq V'_{\varepsilon,1}$ and $V'_{\varepsilon,2}\subseteq 
 V'_{\varepsilon,1}$.  But $\dim V'_{\varepsilon,2}=\dim V'_{\varepsilon,1}$  so  $V'_{\varepsilon,2}=V'_{\varepsilon,1}$.  
\end{enumerate}

%%%%%%%%%%%%%%%%%%%%%%%%%%%%%%%%%%%%%%%%%%%%%%%%%%%%%%%%%%%%%%%%
%%%%%%%%%%%%%%%%%%%%%%%%%%%%%%%%%%%%%%%%%%%%%%%%%%%%%%%%%%%%%%%%

%\bibliography{biblio4}
%\bibliographystyle{customamsalpha}

\def\cprime{$'$}
\providecommand{\bysame}{\leavevmode\hbox to3em{\hrulefill}\thinspace}
\providecommand{\MR}{\relax\ifhmode\unskip\space\fi MR }
\providecommand{\MRhref}[2]{%
  \href{http://www.ams.org/mathscinet-getitem?mr=#1}{#2}
}
\providecommand{\href}[2]{#2}

\end{document}